\documentclass[11pt,a4paper]{amsart}
\usepackage{fourier}
\usepackage{calrsfs}

\usepackage{a4wide}
\usepackage{amsthm,amsmath,amssymb,amscd}
\usepackage{latexsym}
\usepackage[all]{xy}
\usepackage{comment}
\usepackage{appendix}
\usepackage{pdfsync}

\usepackage{graphicx}
\usepackage{wrapfig}
\usepackage{caption}

\usepackage{color}
\definecolor{Chocolat}{rgb}{0.36, 0.2, 0.09}
\definecolor{BleuTresFonce}{rgb}{0.215, 0.215, 0.36}

\usepackage[colorlinks,final,hyperindex]{hyperref}
\hypersetup{citecolor=BleuTresFonce, linkcolor=Chocolat}

\DeclareFontFamily{U}{shuffle}{} 
\DeclareFontShape{U}{shuffle}{m}{n}{<5-8>shuffle7  <8->shuffle10}{}
\DeclareSymbolFont{Shuffle}{U}{shuffle}{m}{n}
\DeclareMathSymbol\shuffle{\mathbin}{Shuffle}{"001} 
\DeclareMathSymbol\cshuffle{\mathbin}{Shuffle}{"002}

\DeclareMathAlphabet{\mathbbold}{U}{bbold}{m}{n}

\def\k{\mathbbold{k}}

\DeclareMathAlphabet{\pazocal}{OMS}{zplm}{m}{n}

\def\calA{\pazocal{A}}
\def\calB{\pazocal{B}}

\def\calD{\pazocal{D}}

\def\calF{\pazocal{F}}
\def\calG{\pazocal{G}}
\def\calH{\pazocal{H}}
\def\calI{\pazocal{I}}

\def\calL{\pazocal{L}}
\def\calM{\pazocal{M}}
\def\calN{\pazocal{N}}
\def\calO{\pazocal{O}}
\def\calP{\pazocal{P}}
\def\calQ{\pazocal{Q}}

\def\calS{\pazocal{S}}
\def\calT{\pazocal{T}}

\def\calZ{\pazocal{Z}}

\DeclareMathOperator{\Ord}{Ord}
\DeclareMathOperator{\Hom}{Hom}
\DeclareMathOperator{\End}{End}
\DeclareMathOperator{\id}{id}

\DeclareMathOperator{\Diff}{ncDiff}

\DeclareMathOperator{\ncW}{ncW}

\DeclareMathOperator{\As}{As}
\DeclareMathOperator{\BV}{BV}
\DeclareMathOperator{\ncBV}{ncBV}
\DeclareMathOperator{\qncBV}{qncBV}
\DeclareMathOperator{\Gerst}{Gerst}
\DeclareMathOperator{\ncGerst}{ncGerst}
\DeclareMathOperator{\twoncGerst}{2\text{-}ncGerst}
\DeclareMathOperator{\Grav}{Grav}
\DeclareMathOperator{\ncGrav}{ncGrav}
\DeclareMathOperator{\HyperCom}{HyperCom}
\DeclareMathOperator{\ncHyperCom}{ncHyperCom}

\DeclareMathOperator{\tAs}{tAs}
\DeclareMathOperator{\pAs}{pAs}

\DeclareMathOperator{\IBL}{IBL}
\DeclareMathOperator{\ncFrob}{ncFrob}
\DeclareMathOperator{\IIB}{IIB}
\DeclareMathOperator{\myspan}{span}
\DeclareMathOperator{\Gap}{Gap}
\DeclareMathOperator{\D}{D}

\def\Im{\mathop{\mathrm{Im}}}
\def\Ker{\mathop{\mathrm{Ker}}}
\def\C{\pazocal{C}}

\def\qi{\xrightarrow{\sim}}
\def\g{\mathfrak{g}}
\newcommand{\MC}{\mathrm{MC}}

\def\1{\mathbb{1}}
\def\E{\varepsilon}
\def\Sy{\mathbb{S}}
\def\V{\mathop{\mathrm{Var}}} 

\makeatletter

\theoremstyle{plain}

\newtheorem{theorem}{Theorem}[subsection]

\newtheorem {corollary}{Corollary}[subsection]
\newtheorem {conjecture}{Conjecture}[subsection]

\newtheorem {proposition}{Proposition}[subsection]

\theoremstyle{definition}

\newtheorem {definition}{Definition}[subsection]
\newtheorem {remark}{Remark}[subsection]
\newtheorem {example}{Example}[subsection]

\newcommand{\ac}{\scriptstyle \text{\rm !`}}



\thanks{S.S. was supported by the Netherlands Organisation for Scientific Research. B.V. was supported by the ANR  SAT grant.}

\begin{document}

\title[Noncommutative cohomological field theories]{Toric varieties of Loday's associahedra\\ and noncommutative cohomological field theories}

\author{Vladimir Dotsenko}
\address{School of Mathematics, Trinity College, Dublin 2, Ireland}
\email{vdots@maths.tcd.ie}

\author{Sergey Shadrin}
\address{Korteweg-de Vries Institute for Mathematics, University of Amsterdam, P. O. Box 94248, 1090 GE Amsterdam, The Netherlands}
\email{s.shadrin@uva.nl}

\author{Bruno Vallette}
\address{Laboratoire Analyse, G\'eom\'etrie et Applications, Universit\'e Paris 13, Sorbonne Paris Cit\'e, CNRS, UMR 7539, 93430 Villetaneuse, France.}
\email{vallette@math.univ-paris13.fr}

\begin{abstract}
We introduce and study several new topological operads that should be regarded as nonsymmetric analogues of the operads of little $2$-disks, framed little $2$-disks, and Deligne--Mumford compactifications of moduli spaces of genus zero curves with marked points. These operads exhibit all the remarkable algebraic and geometric features that their classical analogues possess; in particular, it is possible to define a noncommutative analogue of the notion of cohomological field theory with similar Givental-type symmetries. This relies on rich geometry of the analogues of the Deligne--Mumford spaces, coming from the fact that they admit several equivalent interpretations: as the toric varieties of Loday's realisations of the associahedra, as the brick manifolds recently defined by Escobar, and as the De Concini--Procesi wonderful models for certain subspace arrangements. 
\end{abstract}

\maketitle

\setcounter{tocdepth}{1}

\tableofcontents

\section{Introduction}
In topology, operads emerged in the context of recognition of loop spaces; in particular, the little $2$-disks operad $\calD_2$ has been one of the protagonists from the early years of topological operads. Together with the operad of framed little $2$-disks $f\calD_2$ and the Deligne--Mumford operad of compactifications $\overline{\calM}_{0,n+1}$ of moduli spaces of genus zero curves with marked points, it fits into a remarkable diagram 
\begin{equation}\label{eq:ComDiagM0nGeom}
 \xymatrix@M=6pt{
 	\left\{\overline{\calM}_{0,n+1}\right\}\ar@{<<--}[rr]  & &  f\calD_2=\calD_2\rtimes S^1\   \\ 
 	\calD_2 / S^1=\left\{{\calM}_{0,n+1}\right\}\ar@{^{(}->}[u]\ar@{<<-}[rr]  & & \calD_2 \ar@{^{(}->}[u]     
 }
\end{equation}
expressing the following facts. First, one can obtain the operad $f\calD_2$ as the extension of $\calD_2$ by the $S^1$-action: $f\calD_2\cong\calD_2\rtimes S^1$ \cite{Markl1999}. Then, the homotopy quotient of the latter operad by the circle action is represented by the Deligne--Mumford operad $\big\{\overline{\calM}_{0,n+1}\big\}$ \cite{DC2014}. Finally, 
the quotient of $\calD_2$ by the circle action is equivalent to the collection of open moduli spaces $\big\{\calM_{0,n+1}\big\}$, which is the open piece of the Deligne--Mumford operad \cite{Get94} that has a normal crossing divisor as its complement. 

The homology of the Deligne--Mumford operad $\big\{\overline{\calM}_{0,n+1}\big\}$ is an algebraic operad that controls the algebraic structure known, in different contexts and with small differences in definitions, under the names of  hypercommutative algebras, formal Frobenius manifolds, or genus zero reductions of Gromov--Witten theory \cite{KontsevichManin94}. It is related to the classical operads of Gerstenhaber algebras, Batalin--Vilkovisky algebras and gravity algebras, fitting into a diagram of algebraic operads 
\begin{equation}\label{eq:ComDiagM0nAlg}
 \xymatrix@M=6pt{
\HyperCom\ar@{<<--}[rr] \ar@{<->}[d] & &  \BV   \\ 
 \Grav\ar@{^{(}->}[rr]  & & \Gerst \ar@{^{(}->}[u]     
 }
\end{equation}
obtained by computing the homology of the above-mentioned topological operads. This diagram expresses the following facts. First, the operads $\HyperCom$ and (the suspension of) $\Grav$ are related by Koszul duality \cite{Get95}. Next,  the operad $\Grav$ is a suboperad of the operad $\Gerst$ which in turn is a suboperad of the operad $\BV$ \cite{Get94}. Finally, the operad $\HyperCom$ is a representative for the  homotopy quotient of the operad $\BV$ by $\Delta$ \cite{DCV2013,KMS2013}.

 In addition, algebra structures over the operad $\HyperCom$ --- genus zero cohomological field theories --- admit an unexpectedly rich group of symmetries due to Givental \cite{Giv01} coming from different trivialisations of the circle action of the homotopy quotient \cite{DSV2013b}. The latter group of symmetries can also be used to encode Koszul braces on commutative algebras \cite{Koszul85} and a particular type of homotopy $\BV$-algebras, so called commutative homotopy $\BV$-algebras \cite{Kra00}. Last but not least, each space $\overline{\calM}_{0,n+1}$ can be realised as a De Concini--Procesi wonderful model \cite{DCP95} of the Coxeter hyperplane arrangement of type $A_{n-1}$, and the operad structure on these spaces comes from a general operad-like structure that exists for wonderful models \cite{Rains10}. 

In this paper, we define analogues of all these objects (the little $2$-disks operad, the framed little $2$-disks operad, the Deligne--Mumford operad) that ultimately lead to a noncommutative analogue of the notion of a cohomological field theory. This new notion does not appear to be related to various notions of a noncommutative topological field theory (e.g. \cite{AlNat06,Nat04}); in fact, it shares many more common properties with its commutative counterpart. Notably, the corresponding operad is not cyclic (the output of the operation $\nu_n$ cannot be put on the same ground as its inputs). Therefore, this operad cannot be extended to include higher genera; somehow, in the noncommutative world only genus zero seems to be ``visible''.

Remarkably, our analogues possess all the expected algebraic and geometric properties of the classical picture we just described. The only aspect that one has to sacrifice is the definition as a moduli space; we do not use any ``noncommutative curves of genus zero with marked points''  as such in the construction, although some geometric intuition of this sort is present. More precisely, the simplest way to implement noncommutativity for curves with marked points is to consider configurations of points on the real projective line, which leads to the definition of Stasheff associahedra; we however do not consider the collection of associahedra themselves, but rather the collection of toric varieties associated to Loday's realisations of associahedra~\cite{JLL2004}. We denote these spaces by $\calB(n)$ to emphasize their remarkable interpretation as ``brick manifolds'' \cite{Escobar2014}. The first step in unravelling the geometry of noncommutative moduli spaces is to endow these brick manifolds with a nonsymmetric (ns) operad structure. 

The geometric properties of the ns operad of brick manifolds generalising~\eqref{eq:ComDiagM0nGeom} are summarised by the diagram 
\begin{equation}\label{eq:ComDiagNCM0nGeom}
\xymatrix@M=6pt{
	\calB\ar@{<<--}[rr]  & &  \As_{S^1}\rtimes S^1\   \\ 
	 \As_{S^1} / S^1\ar@{^{(}->}[u]\ar@{<<-}[rr]  & & \As_{S^1} \ar@{^{(}->}[u]     
}
\end{equation}
where $\As_{S^1}$ is a noncommutative counterpart of the operad of little $2$-disks, 
which we introduce in this paper, where $\As_{S^1}\rtimes S^1$ is its extension by the $S^1$-action, and 
where the quotients of components of the operad $\As_{S^1}$ by the circle action are equivalent to open strata $(\mathbb{C}^\times)^{n-2}$ of the brick manifolds. We also conjecture that the brick operad is equivalent to the homotopy quotient of the operad $\As_{S^1}\rtimes S^1$ by the circle action; to that end we exhibit a different operad,  homotopy equivalent to it,  which resembles the construction of Kimura--Stasheff--Voronov and Zwiebach in the classical case \cite{KSV,Zwi93}. 

On the algebraic level, the homology ns operad of the ns operad of brick manifolds, which we denote $\ncHyperCom$, admits a presentation
which is very much reminiscent of Getzler's presentation of the operad $\HyperCom$ \cite{Get95}. A ``noncommutative hypercommutative algebra'' is a graded vector space $A$ equipped with operations $\nu_n\colon A^{\otimes n}\to A$ of degree $2n-4$ ($n\ge 2$) satisfying, for each $n\ge 3$ and each $i=2,\ldots,n-1$, the  identity
\begin{multline*}
\sum_{k+l=n-1} \nu_k(a_1,\ldots,a_{i-l},\nu_l(a_{i-l+1},\ldots,a_{i-1},a_i),a_{i+1},\ldots,a_n)=\\=
\sum_{k+l=n-1} \nu_k(a_1,\ldots,a_{i-1},\nu_l(a_i,a_{i+1},\ldots,a_{i+l-1}),a_{i+l},\ldots,a_n)\ .
\end{multline*}
The algebraic properties of this operad generalising \eqref{eq:ComDiagM0nAlg} are expressed by the diagram
\begin{equation}\label{eq:ComDiagNCM0nAlg}
 \xymatrix@M=6pt{
 	\ncHyperCom\ar@{<<--}[rr] \ar@{<->}[d] & &  \ncBV   \\ 
 	\ncGrav\ar@{^{(}->}[rr]  & & \ncGerst \ar@{^{(}->}[u]     
 }
\end{equation}
with the exact same properties. We also define an analogue of the Givental group action on noncommutative cohomological field theories using the intersection theory for brick manifolds, which we develop in detail. It can be used, on the one hand, to construct an explicit identification of $\ncHyperCom$ with the homotopy quotient $\ncBV / \Delta$, similarly to \cite{KMS2013}, and, on the other hand, to provide a conceptual framework for  B\"orjeson products. Finally, we are able to show that the toric varieties of Loday associahedra are De Concini--Procesi wonderful models of certain hyperplane arrangements which we call ``noncommutative braid arrangements''. The toric varieties of the permutahedra \cite{Pro90}, also known as Losev--Manin moduli spaces \cite{LosevManin}, are, in fact, wonderful models of the same subspace arrangements but a different ``building set''.\footnote{Losev--Manin moduli spaces are  examples of moduli spaces of curves with weighted marked points introduced by Hassett~\cite{Has03}. Hence, it is most natural to ask whether the toric varieties of the associahedra are examples of Hassett spaces as well. It turns out that the answer to this question is negative; this follows from a recent paper \cite{FJR14} where all graph associahedra for which the toric varieties are Hassett spaces are classified. }

The interpretation of the moduli spaces of genus zero curves with marked points as De Concini--Procesi wonderful models was behind  the result of \cite{EHKR10} on the homology operad of the real loci of the moduli spaces of curves. We demonstrate that their result admits a natural noncommutative counterpart as well: we introduce a noncommutative version of $2$-Gerstenhaber algebras and prove that the homology ns operad of the real brick operad is isomorphic to the operad encoding them. 

On the algebraic level, small hints as to what types of algebras should arise as such analogues are contained in two recent papers. 
On the one hand, a noncommutative version of Koszul brackets  defined by B\"orjeson in \cite{Bor13} leads to 
the notion of noncommutative Batalin--Vilkovisky algebras. On the other hand, the noncommutative deformation theory developed by  Ginzburg and Schedler in \cite{GS} contains an appropriate deformation complex whose  structure  may  naturally be viewed as a noncommutative version of a Gerstenhaber algebra.

\subsection*{Layout} 
The paper is organised as follows. In Section \ref{sec:recollections}, we recall some  definitions used throughout the paper; for notions that are specific to particular sections, we  defer the recollection until the corresponding section on some occasions.  In Section \ref{sec:nc-gerst-bv}, we define noncommutative versions of the operads of Gerstenhaber algebras and Batalin--Vilkovisky algebras. In Section \ref{sec:minimal-model-ncBV}, we compute the minimal model of the operad $\ncBV$ of noncommutative Batalin--Vilkovisky algebras; at this stage the operad of noncommutative gravity algebras enters the game. In Section \ref{sec:three-ncHyperCom}, we present three constructions of the operad $\ncHyperCom$, the algebraic definition mentioned above and two geometric definitions, via brick manifolds and via wonderful models of subspace arrangements. In Section \ref{sec:Givental}, we develop the intersection theory on brick manifolds, define a noncommutative version of the notion of a cohomological field theory, and explain that noncommutative cohomological field theories have a rich Givental-like group of symmetries. In Section \ref{sec:applications}, we discuss several applications and further questions prompted by our results. First, we obtain a definition of B\"orjeson products via the Givental action and discuss a connection between those products and noncommutative Frobenius bialgebras. Second, we utilise the Givental action to prove that the operad $\ncHyperCom$ represents the homotopy quotient of the operad $\ncBV$ by $\Delta$, and set up a framework in which we can conjecture a precise geometric version of that statement in the spirit of \cite{KSV,Zwi93}. In Section \ref{sec:RealCase}, we use the viewpoint of wonderful model of subspace arrangements to compute the homology operad of the real brick operad, obtaining noncommutative counterparts of the key results of \cite{EHKR10}. Finally, in the Appendix, we provide an alternative proof of a result of Escobar \cite{Escobar2014}, who established that brick manifolds are toric varieties of Loday polytopes.

\subsection*{Acknowledgements} This paper was completed during the authors\rq{} stays at University of Amsterdam, Trinity College Dublin, and Universit\'e Nice Sophia Antipolis. The authors would like to thank these institutions for the excellent working conditions enjoyed during their stays there. 
The authors would also like to acknowledge useful discussions with Raf Bocklandt, Mikhail Kapranov, Yuri Ivanovich Manin, Arkady Vaintrob, and Thomas Willwacher during the final stages of the preparation of this paper. 

\section{Background, notations, and recollections}\label{sec:recollections}

Unless otherwise specified, all vector spaces and (co)chain complexes throughout this paper are defined over an arbitrary field~$\k$. 

\subsection{Ordinals}

Most objects defined throughout this paper depend functorially on a choice of a totally ordered finite set. In many cases, it is enough to consider the set $\underline{n}=\{1,\ldots,n\}$ with its standard ordering, but on some occasions a more functorial view would be much more beneficial. 
To this end, we shall consider the category $\Ord$ whose objects are finite ordinals (totally ordered sets), and morphisms are order-preserving bijections. This is a monoidal category where the monoidal structure is given by $+$, the ordinal sum.

For every pair of disjoint ordinals $I,J$, and every $i\in I$, we may consider the set $I\sqcup_i J:=I\sqcup J\setminus\{i\}$ with the ordering given by
\[
a<b \text{ \ \ if\  \  }
\begin{cases}
a,b\in I, \text{ and } a<b,\\
a,b\in J, \text{ and } a<b,\\ 
a\in I, b\in J, \text{ and } a<i,\\
a\in J, b\in I, \text{ and } i<b.
\end{cases}  
\]
In other words, we can present $I\sqcup_i J$ as an ordinal sum
 \[
I\sqcup_i J:=\left(\sum_{k\in I, k<i}\underline{1}\right)+J+\left(\sum_{k\in I, k>i}\underline{1}\right)\ .  
 \]

For a finite ordinal $I$, we shall be using the two partial functions $p,s\colon I\to I$. The element $p(i)$ is the \emph{predecessor} of $i$, i.e. the element immediately before $i$ relative to the given order on $I$ (if exists), and the element $s(i)$ is the \emph{successor} of $i$, i.e the element immediately after $i$ relative to the given order on $I$ (if exists). 

\begin{definition}
Let $I$ be a finite ordinal. The \emph{gap set} $\Gap(I)$ is the set of all pairs $(p(i),i)$ for $\min(I)<i\in I$; this set has a natural total order induced by the product order (in fact, in this case it is merely the order of first components).
\end{definition}

\begin{proposition}\label{prop:operad-gap}
There is a natural bijection
	\[
\Gap(I)\sqcup\Gap(J)  \to \Gap(I\sqcup_i J)\ , 
	\]
which sends each pair $(p(k),k)$ from the left-hand side set to the same pair of the right-hand side set with the exception of the pairs where one of the elements is $i$: the pair $(p(i),i)$, if exists, is sent to $(p(i),\min(J))$, and the pair $(i,s(i))$, if exists, is sent to $(\max(J),s(i))$.
\end{proposition}

\begin{proof}
Direct inspection.
\end{proof}

\subsection{Nonsymmetric operads}

In this paper, we predominantly deal with nonsymmetric operads. The notion of a nonsymmetric operad builds upon the notion of a nonsymmetric collection, which is defined as follows.  A \emph{nonsymmetric collection} is a functor from the category $\Ord$ to the category of vector spaces, which we usually assume graded, with the Koszul sign rule for symmetry isomorphisms of tensor products. 
The category of nonsymmetric collections admits a monoidal structure, called the \emph{nonsymmetric composite product}, and denoted $\circ$, which is defined by the following formula:
 \[
(\calF\circ\calG)(I):=\bigoplus_{k\geqslant 1}\calF(\underline{k})\otimes\bigoplus_{I_1+I_2+\cdots+I_k=I}\calG(I_1)\otimes\cdots\otimes\calG(I_k) .   
 \] 
This operation is associative, and the nonsymmetric collection $\calI$ defined as
 \[
\calI(I)=\begin{cases}
\k\, \id, \quad\# I=1,\\
0, \quad\text{ otherwise}, 
\end{cases} 
 \] 
is the unit for this operation, that is, $\calF\circ\calI\cong\calI\circ\calF\cong\calF$ for each nonsymmetric collection $\calF$. A monoid in the monoidal category of nonsymmetric collections with respect to $\circ$ is called a \emph{nonsymmetric operad}. Throughout this paper, we use the abbreviation `ns' instead of the word `nonsymmetric'.  

An equivalent way to present operads is via \emph{infinitesimal compositions}. When $\calP$ is a ns operad,  there exists a map, called the \emph{infinitesimal composition at the slot $i$}, defined by 
\begin{multline*}
\circ_{I,i}^J\colon\calP(I)\otimes\calP(J)\cong\calP(I)\otimes\left(\bigotimes_{k\in I, k<i}\calI\right)\otimes\calP(J)\otimes\left(\bigotimes_{k\in I, k>i}\calI\right)\hookrightarrow\\ \hookrightarrow  \calP(I)\otimes\left(\bigotimes_{k\in I, k<i}\calP(\underline{1})\right)\otimes\calP(J)\otimes\left(\bigotimes_{k\in I, k>i}\calP(\underline{1})\right)
\hookrightarrow (\calP \circ \calP) (I\sqcup_i J)
\to\calP(I\sqcup_i J).   
\end{multline*}
Conversely, from infinitesimal compositions, one can recover all the structure maps of a ns operad.

It is common to work with nonsymmetric operads in a slightly different way, restricting one's attention to the components $\calP(n):=\calP(\underline{n})$, and denoting the corresponding structure maps by $\circ_i$, making the set $J$ implicit. We shall use this notation on some occasion where the categorical definition makes the definition too technical and does not add clarity.

To handle suspensions of chain complexes, we introduce an element~$s$ of degree~$1$, and define, for a graded vector space~$V$, its suspension $sV$ as $\k s\otimes V$. The endomorphism operad $\End_{\k s}$ is denoted by $\calS$. For a nonsymmetric collection $\calP$, its operadic suspension $\calS\calP$ is the Hadamard tensor product $\calS\underset{\mathrm{H}}{\otimes}\calP$.\\

For information on operads in general, we refer the reader to the book~\cite{LV}, for information on Gr\"obner bases for operads in not necessarily quadratic case --- to the book~\cite{BD}. There exists a notion dual to operad called \emph{cooperad}, which can be relaxed up to homotopy to define the notion of \emph{homotopy cooperad}. Homotopy cooperads, and in particular their homotopy transfer theorem, are discussed at length in~\cite{DCV2013}.

\subsection{Toric varieties}

In this section, we briefly summarise basic information on toric varieties. We refer the reader to \cite{Danilov78,Fulton93} for more details.

Throughout the paper, we denote by $\mathbb{G}_m$ the algebraic group $\mathop{\mathrm{Spec}}\left(\k[x,x^{-1}]\right)$, i.e. the multiplicative group $\k^\times$ of the ground field, whenever we would like to emphasize that we deal with algebro-geometric constructions that do not depend on the ground field.

An \emph{algebraic torus} is a product of several copies of $\mathbb{G}_m$. 
A \emph{toric variety} is a normal algebraic variety $X$ that contains a dense open subset $U$ isomorphic to an algebraic torus, for which the natural torus action on $U$ extends to an action on $X$. 

Toric varieties may be constructed from some combinatorial data, namely a \emph{lattice} (free finitely generated Abelian group) $N$ and a \emph{fan} (collection of strongly convex rational polyhedral cones closed under taking intersections and faces) $\Phi$ in $N\otimes_{\mathbb{Z}}\mathbb{R}$. Each cone in a fan gives rise to an affine variety, the affine spectrum of the semigroup algebra of the dual cone. Gluing these affine varieties together according to face maps of cones gives an algebraic variety denoted $X(\Phi)$ and called the toric variety associated to the fan $\Phi$. 
	
It is known that a toric variety $X(\Phi)$ is projective if and only if $\Phi$ is a normal fan of a convex polytope $P$. In such situation, we also use the notation $X(P)$ instead of $X(\Phi)$. The variety $X(P)$ is smooth if and only if $P$ is a \emph{Delzant polytope}, that is a polytope for which the slopes of the edges of the vertex form a basis of the lattice $N$. 

\begin{proposition}[{\cite{Fulton93}}]\label{prop:h-vector}
For an $n$-dimensional Delzant polytope, the Betti numbers of $X(P)$ are given by the components $h_i$ of the \emph{$h$-vector} of $P$, i.e. the integers defined by 
 \[
\sum_{i=0}^n h_it^i=\sum_{i=0}^n f_i(t-1)^i , 	
 \]
where $f_i$ denotes the number of faces of $P$ of dimension $i$.
\end{proposition}

\subsection{Associahedra and Loday polytopes}

\emph{Associahedra}, or \emph{Stasheff polytopes}, go back to works of Tamari and Stasheff \cite{Stasheff63,Stasheff12,Tamari51}. The $n$-th associahedron $K_n$ is a (possibly curvilinear) polytope of dimension $n-2$ where vertices correspond to different ways of putting parentheses in a word of $n$ letters, and each edge corresponds to a single application of the associativity rule. Historically, associahedra were mainly viewed as abstract CW-complexes, but in the past decades many different realisations of them as lattice polytopes have been obtained, see e.g.~\cite{CSZ14} and references therein. 

We recall the realisation of associahedra from \cite{JLL2004}, which we shall denote $L_n$ and call the \emph{Loday polytopes} for brevity. For each $n\ge 2$, the polytope $L_n$ is a convex span of $\frac1n\binom{2n-2}{n-1}$ points (its vertices) in an affine hyperplane in $\mathbb{Z}^{n-1}$. The vertices are in one-to-one correspondence with the trivalent rooted planar trees with $n$ leaves ${l}_i$, for $i=1,\dots,n$, and labelled by $1,\dots,n$ from left to right. 

Consider such a tree $T$. It has  $n-1$ internal vertices ${v}_1, \dots, {v}_{n-1}$, where the vertex ${v}_i$ is the one that separated the leaves ${l}_i$ and ${l}_{i+1}$. That is, denote by $D^l({v}_i)$ (respectively, $D^r({v}_i)$) the set of leaves that are the left (respectively, right) descendants of ${v}_i$. Then we require that ${l}_i\in D^l({v}_i)$ and ${l}_{i+1}\in D^r({v}_i)$; this specifies ${v}_i$ uniquely. We associate to a tree $T$ the point $p_T\in\mathbb{Z}^{n-1} $, whose $i$ coordinate is equal to $|D^l({v}_i)|\cdot |D^r({v}_i)|$, that is,
\[
p_T:=\left( |D^l({v}_1)|\cdot | D^r({v}_1)|, \, |D^l({v}_2)|\cdot | D^r({v}_2)|, \, \dots,\,  |D^l({v}_{n-1})|\cdot | D^r({v}_{n-1})|\right) 
\]
It is easy to check that $\sum_{i=1}^{n-1} |D^l({v}_i)|\cdot |D^r({v}_i)| = \binom{n}{2}$. Therefore, we associate to every trivalent planar rooted tree $T$ a point $p_T\in \mathbb{Z}^{n-1}$ that belongs to the affine hyperplane in $\mathbb{Z}^{n-1}$ given in coordinates $(x_1,\dots,x_{n-1})$ by the equation $\sum_{i=1}^{n-1} x_i = \binom{n}2$.

\begin{proposition}[\cite{JLL2004}]
The points $\{p_T\}$, where $T$ ranges over trivalent planar rooted trees with $n$ leaves, are the vertices of their convex hull $L_n$, which is a convex polytope of dimension $n-2$. This lattice polytope is combinatorially equivalent to the associahedron $K_n$. 	
\end{proposition}

In particular, the correspondence between the vertices of $L_n$ and the trivalent planar rooted trees explained above descends to the standard combinatorial correspondence in Stasheff's description of the associahedra.

\begin{remark}
Let us note here that $\underline{n-1}\cong\Gap(\underline{n})$ as ordered sets, and that by the very nature of the construction of Loday polytopes, each polytope $L_n$ may be naturally viewed as a lattice polytope in $\mathbb{Z}^{\Gap(\underline{n})}$. This leads to a functorial definition of a lattice polytope $L_I$ for every finite ordinal $I$ that is in fact behind the key constructions of this paper.
\end{remark}

Let us also recall a useful equivalent definition of Loday polytopes due to Postnikov~\cite{Post09}.

\begin{proposition}[\cite{Post09}]\label{prop:Minkowski}
The polytope $L_n$ is the Newton polytope of the polynomial 
	\[
\prod_{1\leqslant i\leqslant j\leqslant n-1}(t_i+t_{i+1}+\cdots+t_j) ,  
	\]
and hence is equal to the Minkowski sum of simplexes corresponding to intervals of $\underline{n-1}\cong\Gap(\underline{n})$.	
\end{proposition}

\section{Noncommutative versions of little 2-disks and framed little 2-disks operads}\label{sec:nc-gerst-bv}

In this section, we propose noncommutative versions of the topological operads of little $2$-disks and framed little $2$-disks that lead, respectively, to the algebraic operads of  noncommutative Gerstenhaber and noncommutative Batalin--Vilkovisky algebras controlling the algebraic structures arising in \cite{Bor13,GS}.

\subsection{Noncommutative \texorpdfstring{$\calD_2$}{LD2} and noncommutative Gerstenhaber algebras}

\subsubsection{The M-associative operad, and its linear version}

In this section, we define a functor $\As_M$ from the category of sets to the category of set-theoretic ns operads; it generalises the associativity property in the sense that $\As_{\lbrace * \rbrace}$ is the usual ns associative operad $\As$. 

\begin{definition}
For a set $M$, we define the \emph{$M$-associative operad $\As_M$} as the set-theoretical  ns  operad whose $n$-th component is $M^{n-1}$, for any $n\ge 1$, and $\As_M(0)=\emptyset$. Its operadic compositions are given by 
 \[
\circ_i\colon \As_M(n_1)\times\As_M(n_2)=M^{n_1-1}\times M^{n_2-1}\to M^{n_1+n_2-2}=\As_M(n_1+n_2-1) 
 \]
are given by the formula 
 \[ 
(m_1,\ldots,m_{n_1-1})\circ_i(m_{n_1},\ldots,m_{n_1+n_2-2}):=
(m_1,\ldots,m_{i-1},m_{n_1},\ldots,m_{n_1+n_2-2},m_i,\ldots,m_{n_1}). 
 \]
\end{definition}

\begin{remark}
In a more invariant way, we can define the $M$-associative operad $\As_M$ as the ns operad whose component $\As_M(I)$ is $M^{\Gap(I)}$, and the operadic compositions 
\[
\circ_{I,i}^J\colon\As_M(I)\times\As_M(J)=M^{\Gap(I)}\times M^{\Gap(J)}\to 
M^{\Gap(I\sqcup_i J)}
=
\As_M(\Gap(I\sqcup_i J)) 
\]
arise from the natural bijection $\Gap(I\sqcup_i J)\cong\Gap(I)\sqcup\Gap(J)$ of Proposition~\ref{prop:operad-gap}. 
\end{remark}

\begin{proposition}\label{prop:AsMRelations}
The operations defined above give the collection $\As_M$ a structure of a set-theoretic ns operad. This operad can presented via generators and relations, $\As_M=\calT(V)/(R)$, where $V:=\{(m)\}_{m\in M}=\As_M(2)$ and $R\in\As_M(3)$ is the set of relations
\begin{equation}\label{M-as}
(m)\circ_1 (m') =(m')\circ_2 (m) .  
\end{equation}
\end{proposition}

\begin{proof}
For $S_1\in\As_M(n_1)$, $S_2\in\As_M(n_2)$, $S_3\in\As_M(n_3)$, and $1\le i<j\le n_1$, $1\le k\le n_2$, the 
parallel and sequential axioms
\begin{gather*}
(S_1\circ_j S_2)\circ_i S_3=(S_1\circ_i S_3)\circ_{j+n_3-1} S_2,\\
S_1\circ_i(S_2\circ_k S_3)=(S_1\circ_i S_2)\circ_{k-i+1}S_3
\end{gather*}
are checked by direct inspection. It is clear that the operations $(m)$, $m\in M$, generate $\As_M$, since
\[
(m_1,\ldots,m_{n-1})=(\cdots(((m_1)\circ_2(m_2))\circ_3(m_3))\cdots)\circ_{n-1}(m_{n-1}). 
\]
Since $(m)\circ_1 (m')=(m',m)=(m')\circ_2 (m)$, we see that Equation~\eqref{M-as} holds in $\As_M$.  Finally, there are no extra relations in $\As_M$, because already using Equation~\eqref{M-as} alone, we can rewrite every element in $\As_M$ as ``right comb'', that is an iteration of compositions only using the last slot of operations as above, and the set of such right combs of arity $n$ is precisely $M^{n-1}=\As_M(n)$.  
\end{proof}

The definition of the operad $\As_M$ makes sense when $M$ is an object of any symmetric monoidal category. In particular, we shall use it for when $M$ is a topological space, a chain complex, or simply a graded vector space; in the two latter cases, this will produce appropriate Koszul signs because those are involved in symmetry isomorphisms of the corresponding categories. 
If $M$ is a topological space, the collection $\As_M$ is a ns topological operad, and its homology is a ns operad whose components are graded vector spaces. 

\begin{proposition}\label{prop:presentationAsM}
For every topological space $M$, we have an isomorphism	
 \[
H_\bullet(\As_M)\cong\As_{H_\bullet(M)} . 
 \]
For any basis $B$ of $H_\bullet(M)$, the linear ns operad spanned by the set-theoretic ns operad $\As_B$ is canonically isomorphic to $\As_{H_\bullet(M)}$. Hence, this latter ns operad admits for presentation 
$\calT(V)/(R)$, with $V:=\{(b)\}_{b\in B}$ and $R$ as in Proposition~\ref{prop:AsMRelations}.  
\end{proposition}

\begin{proof}
The first point is a direct consequence of the fact that the homology functor  is monoidal. 
\end{proof}

In what follows, the corresponding operads for $M=S^1$ will receive the special attention; those operads provide a sensible noncommutative analogue of the topological operad of little disks and the algebraic operad of Gerstenhaber algebras.  

\subsubsection{Noncommutative Gerstenhaber algebras}

\begin{definition}
The ns operad $\ncGerst$ of noncommutative Gerstehnaber algebras is the homology of the $S^1$-associative operad:
 \[
\ncGerst :=  H_\bullet(\As_{S^1}) .
 \] 
\end{definition}

\begin{proposition}\label{prop:ncGerstRelations}
 The ns operad $\ncGerst$ is isomorphic to the ns operad with two binary generators $m$ and $b$ of respective degree $0$ and $1$, satisfying the relations
 \begin{gather}\label{ncGerst-rel}
 m\circ_1m-m\circ_2m=0,\\
 m\circ_1b-b\circ_2m=0,\\
 b\circ_1m-m\circ_2b=0,\\
 b\circ_1b+b\circ_2b=0.\label{eq:shifted-as}
 \end{gather}
\end{proposition}

\begin{proof}
This a direct corollary of Proposition~\ref{prop:presentationAsM} with $H_\bullet{S^1}=\k m \oplus \k b$, where $|m|=0$ and $|b|=1$. 
\end{proof}

The simplest nontrivial example of a noncommutative Gerstenhaber algebra that we are aware of is the following one. 

\begin{example}\label{ex:bar-ncGerst}
Let $(A,m_A)$ be an associative algebra. 
Recall that its \emph{bar construction} \cite[Section~$2.2$]{LV} is made up of the (non-unital) free associative algebra $\big(\overline{T}\left(s^{-1}A\right),m\big)$ on the desuspension of $A$.
This associative algebra structure $m$  extends to a structure of a $\ncGerst$-algebra by putting 
 \[
b:=\id^{\otimes(p-1)}\otimes(s^{-1}m_A\circ(s\otimes s))\otimes\id^{\otimes(q-1)}\ ,
 \]
 on $\left(s^{-1}A\right)^{\otimes p}\otimes \left(s^{-1}A\right)^{\otimes q}$.
\end{example}

The following proposition shows how noncommutative Gerstenhaber algebras  arise naturally in the context of noncommutative deformation theory of Ginzburg and Schedler \cite{GS}, that is deformation theory where the parameter of deformation is not required to commute with the algebra elements. In that formalism, one associates to an associative algebra $A$ the algebra $A_t$, the $t$-adic completion of the coproduct of unital algebras $A\sqcup\k[t]$, and considers associative products $A\otimes A\to A_t$ which agree with the product on $A$ under the projection $A_t/A_ttA_t\cong A$.  

\begin{proposition}[\cite{GS}]\label{prop:def-ncGerst}
The Hochschild cochain complex 
 \[
C^\bullet(A,\overline{T}(A)):=\bigoplus_{p,k\geqslant 1}C^p(A,A^{\otimes k})=\bigoplus_{p,k\geqslant 1}\Hom((sA)^{\otimes p},A^{\otimes k}) 
 \]	
with coefficients in the  ``outer'' $A$-bimodule structure on $\overline{T}(A)$ 
$$(a'\otimes a'').(a_1\otimes a_2\otimes\cdots \otimes a_k):=(a'a_1)\otimes a_2\otimes\cdots \otimes (a_ka'')$$ 
admits a noncommutative Gerstenhaber algebra structure defined by 
\begin{equation}
m(f,g):=\left(\id^{\otimes(k-1)}\otimes m_A\otimes\id^{\otimes(l-1)}\right)\circ(f\otimes g)\in C^{p+q}(A,A^{\otimes(k+l-1)}) \ ,\label{eq-m-for-def}
\end{equation}
\begin{multline}
b(f,g):=\left(f\otimes\id^{\otimes(l-1)}\right)\circ\left(\id^{\otimes(p-1)}\otimes s\otimes\id^{\otimes(l-1)}\right)\circ\left(\id^{\otimes(p-1)}\otimes g\right)+\\
+\left(\id^{\otimes (k-1)}\otimes g\right)\circ\left(\id^{\otimes(k-1)}\otimes s\otimes\id^{\otimes(q-1)}\right)\circ\left(f \otimes \id^{\otimes(q-1)}\right) \in C^{p+q-1}(A,A^{\otimes(k+l-1)})\ , \label{eq:b-for-def}
\end{multline}
for $f\in C^p(A,A^{\otimes k})$ and $g\in C^q(A,A^{\otimes l})$. 

The equivalence classes of noncommutative deformations of $A$ are in one-to-one correspondence with gauge equivalence classes of solutions to the Maurer--Cartan equation
 \[
d_{\mathrm{Hoch}}\omega+b(\omega,\omega)=0 \ .
 \]
\end{proposition}

It turns out that, like the classical Gerstenhaber operad, the operad $\ncGerst$ behaves very well with respect to homotopical properties like operadic Koszul duality. 
 This will be useful for us at later stages of the paper.

\begin{proposition}\label{prop:KoszulGerst}
The operad $\ncGerst$ is obtained from the operads $\As$ and  $\As_1:=\calS^{-1}\As$ by a distributive law. In particular, the underlying ns collection of this operad is 
$$\ncGerst\cong\As\circ\As_1$$
 and the operad $\ncGerst$ is
Koszul. The Koszul dual operad  of $\ncGerst$ is isomorphic to the suspension of $\ncGerst$: 
 \[
\ncGerst^!\cong \calS\ncGerst \ .  
 \]	
\end{proposition}

\begin{proof}
The first statement is proved by a  calculation of the dimensions of the two collections 
$$\dim \ncGerst_k(n)=\binom{n-1}{k}=\dim  \left(\As\circ\As_1\right)_k(n)$$
 and by the standard statements on distributive laws \cite[Section~$8.6$]{LV}. It directly implies the two claims that follow. Inspecting the definitions immediately proves the last statement. 
\end{proof}

\begin{remark}
Note a slight difference of signs between \eqref{eq:shifted-as}, \eqref{eq:b-for-def} and the respective formulas from \cite{GS}. It comes from the fact that the operation $b$ has homological degree $1$, so some signs arise from evaluating this operation on elements. These signs are ignored in \cite{GS}; it did not affect the validity of their main results. However, it is absolutely crucial that Equation \eqref{eq:shifted-as} holds as we state it for many purposes. (In particular, for the other choice of signs the corresponding operad fails to be Koszul and so  does not possess the same neat homotopical properties).
\end{remark}

\subsection{Noncommutative \texorpdfstring{$f\calD_2$}{fLD2} and noncommutative Batalin--Vilkovisky algebras}

\subsubsection{Extensions of operads by groups and bialgebra actions}
\label{sec:g-action-operadic-composition}

If $G$ is a topological group and $\calA$ is a ns operad in the monoidal category of topological $G$-spaces, or $G$-modules, one can define, following \cite{Markl1999}, their \emph{semidirect product $\calA\rtimes G$}, which is a ns operad in the category of ordinary topological spaces or modules. 
It satisfies the property that the category of $\calA$-algebras in $G$-spaces is isomorphic to the category of $\calA\rtimes G$-algebras in topological spaces. 
Its components are described by the formula
 \[
(\calA\rtimes G)(n):=\calA(n)\times G^n, 
 \]  
with composition maps
 \[
\gamma_{\calA\rtimes G}\colon (\calA\rtimes G)(k)\times  (\calA\rtimes G)(n_1)\times\cdots\times  (\calA\rtimes G)(n_k)\to (\calA\rtimes G)(n_1+\cdots+n_k)
 \]
given by  
 \[
\gamma_{\calA\rtimes G}\big((a,\underline{g}),(b_1,\underline{h}^1),\ldots,(b_k,\underline{h}^k)\big):=
\big(\gamma_{\calA}(a,g_1b_1,\ldots,g_kb_k),g_1\underline{h}^1,\ldots,g_k\underline{h}^k\big) \ ,
 \]
where we denote $\underline{g}=(g_1,\ldots,g_k)$, $\underline{h}^i=(h_1^i,\ldots,h_{n_i}^i)$, and $g_i\underline{h}^i=\big(g_ih_1^i,\ldots,g_ih_{n_i}^i\big)$.\\

The definition of the semidirect product makes sense when $\calA$ is a ns operad in the monoidal category of $H$-modules for an associative bialgebra $H$, see~\cite{Bellier14}. In that case, 
 \[
 (\calA\rtimes H)(n):=\calA(n)\otimes H^{\otimes n}\ , 
 \]  
with appropriate composition maps. 

If $G$ is a topological group and $\calA$ is a ns operad in the category of $G$-spaces, then $\calA\rtimes G$ is a ns topological operad, and its homology is a ns operad whose components are graded vector spaces. With the associative bialgebra structure on $H_\bullet(G)$, one can obtain the following result, by a direct inspection.

\begin{proposition}\label{prop:semidirect-Kunneth}
For every topological group $G$ and every ns operad $\calA$ in $G$-spaces, we have an isomorphism  of algebraic ns operads
 \[
H_\bullet(\calA\rtimes G)\cong H_\bullet(\calA)\rtimes H_\bullet(G)\ . 
 \]
\end{proposition}

We shall mainly use this construction in the special case where $\calA=\As_{S^1}$ and $G=S^1$ acting diagonally on $\As_{S^1}(n)=(S^1)^{n-1}$. The operad $\As_{S^1}$ should be viewed as a noncommutative analogue of the little $2$-discs operad, while the operad $\As_{S^1}\rtimes S^1$ should be viewed as a noncommutative analogue of the framed little $2$-disks operad. This intuitive understanding is supported by a wide range of results completely parallel to the classical ones that are proved in this paper.  

\subsubsection{Noncommutative Batalin--Vilkovisky algebras}

\begin{definition}
The ns operad $\ncBV$ of \emph{noncommutative Batalin--Vilkovisky algebras} is the ns operad 
$$\ncBV:=H_\bullet(\As_{S^1}\rtimes S^1)\ .$$
\end{definition}

\begin{proposition}\label{prop:ncBV-ql}
The operad $\ncBV$ is generated by two binary generators $m$ and $b$ of respective degree $0$ and $1$, and a unary generator $\Delta$ of degree $1$, satisfying the relations
\begin{gather}
\Delta\circ_1\Delta=0\ ,\label{DeltaSquare}\\
m\circ_1m-m\circ_2m=0\ ,\label{firstG}\\
m\circ_1b-b\circ_2m=0\ ,\\
b\circ_1m-m\circ_2b=0\ ,\\
b\circ_1b+b\circ_2b=0\ ,\label{lastG}\\
\Delta\circ_1 m-m\circ_1\Delta-m\circ_2\Delta=b\ ,\label{action}\\
\Delta\circ_1 b+b\circ_1\Delta+b\circ_2\Delta=0\ .\label{auxiliaryaction}
\end{gather}	
\end{proposition}

\begin{proof}
Let us first note that, by Proposition~\ref{prop:semidirect-Kunneth}, we have an isomorphism of ns collections
 \[
\ncBV\cong H_\bullet(\As_{S^1})\rtimes H_\bullet(S^1)\cong \ncGerst\circ H_\bullet(S^1) ,
 \]
where $H_\bullet(S^1)$ is concentrated in arity $1$. 
In particular, this ns operad is generated by the generators $m$, $b$ of $\ncGerst$ and the generator $\Delta$ of $H_\bullet(S^1)$. Therefore, the relation \eqref{DeltaSquare}, as well as relations \eqref{firstG}--\eqref{lastG} are satisfied in the ns operad $\ncBV$ because they are satisfied in the respective ns suboperads. Furthermore, the relation \eqref{action} expresses the fact that the operation $b$ is obtained from the operation $m$ by the (infinitesimal) circle action, and the relation \eqref{auxiliaryaction} follows formally from the relations \eqref{DeltaSquare} and \eqref{action}. Therefore, there is a surjective map of ns operads $g\colon\calO\twoheadrightarrow\ncBV$, if when denote by $\calO$ the ns operad defined by the above mentioned presentation. 
Finally, we note that the relations \eqref{DeltaSquare}-\eqref{auxiliaryaction} define a rewriting rule according to \cite[Section~$8.6.2$]{LV}. Therefore, there is an epimorphism 
$\ncBV\cong  \ncGerst\circ H_\bullet(S^1) \twoheadrightarrow \calO $. So the map $g$ is an isomorphism, which concludes the proof.
\end{proof}

Similarly to the case of the usual operad BV, where the Lie bracket is a redundant generator, the operad $\ncBV$ also admits a presentation that does not require the operation $b$. This presentation is described as follows.

\begin{proposition}
The ns operad $\ncBV$ is isomorphic to the ns operad with a binary generator $m$ of degree $0$ and a unary generator $\Delta$ of degree $1$ satisfying the relations 
\begin{gather}
\Delta\circ_1\Delta=0\ ,\\
m\circ_1 m-m\circ_2 m=0\ ,\\
\Delta\circ_1m\circ_2m-m\circ_1(\Delta\circ_1m)-m\circ_2(\Delta\circ_1m)+m\circ_2(m\circ_1\Delta)=0\ .\label{4term-relation}
\end{gather}
\end{proposition}

\begin{proof}
An easy computation shows that in any operad with a binary generator $m$, $|m|=0$, and a unary generator $\Delta$, $|\Delta|=1$, satisfying only the first two relations
\begin{gather*}
\Delta\circ_1\Delta=0\ ,\\ 
m\circ_1 m-m\circ_2 m=0\ ,
\end{gather*}
the operation $b$ defined by the formula 
 \[
b:=\Delta\circ_1 m-m\circ_1\Delta-m\circ_2\Delta	
 \]
satisfies the relations
\begin{gather*}
m\circ_1b-b\circ_2m=0\ ,\\
b\circ_1m-m\circ_2b=0\ ,\\
b\circ_1b+b\circ_2b=0
\end{gather*}
if and only if Equation \eqref{4term-relation} holds. 
\end{proof}

\begin{example}\label{ex:borjeson}
The simplest example of a noncommutative Batalin--Vilkovisky algebra comes from the work of B\"orjeson~\cite{Bor13}. Namely, for every associative algebra $(A, m_A)$, the noncommutative Gerstenhaber algebra structure on its bar construction of Example~\ref{ex:bar-ncGerst} is actually a noncommutative BV-algebra structure with
  \[
\Delta:=
 \sum_{1\leqslant i<n}\id^{\otimes(i-1)}\otimes\big(s^{-1} m_A\circ(s\otimes s)\big)\otimes\id^{\otimes(n-1-i)}
  \]
being the differential of the bar construction. This is a direct consequence of \cite[Th.~6]{Bor13}, and is related to the fact that \eqref{4term-relation} can be rewritten in the form 
\[
b_3^\Delta=0\ , 
\]
where $b_3^\Delta$ is the third B\"orjeson product defined for any associative algebra; see Section~\ref{sec:DifOp} for details. One can view this as a noncommutative version of the classical observation \cite{Koszul85} that the Chevalley--Eilenberg differential on the bar construction of a Lie algebra $\mathfrak{g}$ is a differential operator of order at most $2$ with respect to the \emph{algebra} structure on the bar construction.
\end{example}

Similarly to a construction of $\BV$-algebras from involutive bialgebras Lie \cite[Sec.~2.10]{Ginz06}, it is possible to construct $\ncBV$-algebras from their noncommutative analogues; see Example \ref{ex:IIB-BV} in Section~\ref{subsubsec:BarHoInvInfBi}, where we also present a conceptual reason for this class of examples to exist. 

\section{Algebraic structures of noncommutative cohomological field theories}\label{sec:minimal-model-ncBV}

In this section, we adapt various definitions and results of \cite{DCV2013,GTV09,Get94,Get95} to the noncommutative case; the operad $\ncHyperCom$ of noncommutative hypercommutative algebras, or noncommutative cohomological field theories arises, in this setting, from homotopy theory of noncommutative Batalin--Vilkovisky algebras. We also consider the Koszul dual ns operad $\ncGrav$ of the operad $\ncHyperCom$, which provides us with a noncommutative version of the duality between the gravity operad and the operad of hypercommutative algebras coming from geometry of the moduli spaces of curves with marked points. In the subsequent sections, we shall see that this new  Koszul dual pair of ns operads is also the noncommutative algebraic counterpart of a remarkable geometric structure. 

\subsection{The Koszul model of the operad \texorpdfstring{$\ncBV$}{ncBV}}

This section follows the same general strategy of \cite{GTV09}, see \cite[Section~$7.8$]{LV} for more details. More precisely, the operad $\ncBV$ presented by quadratic-and-linear relations \eqref{DeltaSquare}--\eqref{auxiliaryaction} is an inhomogeneous Koszul operad, which can be used to construct a particular (non-minimal) model of it. \\

Let us denote by $L$ the space of generators of the ns operad $\ncGerst$, and by $M=s\ L$ the suspension of that space. We put $\mu=s\, m$, $\beta=s\, b$. An important ingredient of our computations is the Koszul dual cooperad $\ncGerst^{\ac}$ of the operad $\ncGerst$. It is a subcooperad of the cofree ns cooperad $\calT^c(M)$.  The subspace of quadratic elements of $\calT^c(M)$ that belong to the cooperad $\ncGerst^{\ac}$ is spanned by the elements 
 \[
\mu\circ_1\mu-\mu\circ_2\mu,\quad  \mu\circ_1\beta+\beta\circ_2\mu,\quad  \beta\circ_1\mu+\mu\circ_2\beta,\quad  \beta\circ_1\beta+\beta\circ_2\beta . 
 \]
 
The homogeneous quadratic analogue of the operad $\ncBV$ is the operad $\qncBV$ defined by the same set of generators, but with relations
\begin{gather*}
\Delta\circ_1\Delta=0\ ,\\
m\circ_1m=m\circ_2m\ ,\\
m\circ_1b=b\circ_2m\ ,\\
b\circ_1m=m\circ_2b\ ,\\
b\circ_1b+b\circ_2b=0\ ,\\
\Delta\circ_1m-m\circ_1\Delta-m\circ_2\Delta=0\ ,\\
\Delta\circ_1b+b\circ_1\Delta+b\circ_2\Delta=0\ .
\end{gather*}
Let us denote by 
$$\D:=H_\bullet(S^1)= \k[\Delta]/(\Delta^2)\ ,$$
the algebra of dual numbers, which we shall view as a ns operad concentrated in arity $1$.

\begin{proposition}\label{prop:qncBVDistLaw}
The relations of the ns operad $\qncBV$ define a distributive law between the ns operads 
$\ncGerst$ and $\D$.
In particular, the underlying ns collection of the operad $\qncBV$ is isomorphic to 
$$\qncBV\cong \ncGerst \circ \D\cong \As\circ \As_1\circ \D\ , $$ where $\As_1:=\calS^{-1}\As$, and the ns operad $\qncBV$ is Koszul. 
\end{proposition}

\begin{proof}
Since the ns operads $\ncGerst$ and $\D$ are Koszul, we can apply \cite[Theorem~$8.6.5$]{LV}.
In this case, it is enough to prove that the map $\ncGerst\circ \D \to \qncBV$ is injective in weight $3$. Since the relations of the ns operad $\qncBV$ are homogeneous with respect to the numbers $a_1$, $a_2$, and $a_3$ of the generators $\Delta$, $b$ and $m$ respectively, it is enough to check this property on each components labelled by
$(a_1, a_2, a_3)$ with   $a_1+a_2+a_3=3$. The two cases $a_1=0$ and $a_1=3$ are trivial. For $a_1=1$, we can treat the cases $(1,0,2)$, $(1,1,1)$ and $(1,2,0)$ separately for degree reason. For instance, in the case $(1,2,0)$, the dimension of the component $\calT(b, \Delta)$ with two $b$'s and one $\Delta$ is equal to $10$ and the computation of the dimension of its intersection with ideal $(R)$ generated by the relations is equal to $7$. Therefore, the dimensions of the components of weight $(1,2,0)$ of the ns operads $\qncBV$ and $\ncGerst$ are both equal to $3$. The other cases are checked in the same way. 
\end{proof}

This proposition implies that the underlying ns collection of its Koszul dual cooperad $\qncBV^{\ac}$ is isomorphic to  $T^c(\delta)\otimes\ncGerst^{\ac}$, where $\delta=s\, \Delta$.  By a direct corollary of Proposition~\ref{prop:qncBVDistLaw}, the two conditions $(ql_1)$, minimality of the space of generators, and $(ql_2)$, maximality of the space of relations, are satisfied.
Therefore,  the ns cooperad $\qncBV^{\ac}$ admits a square-zero coderivation $d_\varphi$, which extends the map $\varphi$ sending the element 
\[
\delta\circ_1\mu+\mu\circ_1\delta+\mu\circ_2\delta
\] 
to $\beta$ and all other quadratic relations of $\qncBV$ to zero. Equipped with this coderivation, the ns cooperad $\qncBV^{\ac}$ is the Koszul dual
dg ns cooperad $\ncBV^{\ac}$ of the ns operad $\ncBV$ with its quadratic-linear presentation from Proposition \ref{prop:ncBV-ql}.

\begin{proposition}\label{prop:KoszulModel}
The cobar construction $\Omega(\ncBV^{\ac})$ is a resolution of the ns operad $\ncBV$. 
\end{proposition}

\begin{proof}
It is a particular case of the general \cite[Theorem~$38$]{GTV09} since the ns operad $\qncBV$ (the ``quadratic analogue'' of $\ncBV$) is homogeneous Koszul by Proposition~\ref{prop:qncBVDistLaw}.
\end{proof}

This resolution $\Omega(\ncBV^{\ac})$ is called the \emph{Koszul resolution} of the ns operad $\ncBV$.

\subsection{Noncommutative gravity algebras}

In computation of the minimal model of the operad BV in~\cite{DCV2013}, the linear dual cooperad of the operad of gravity algebras of Getzler \cite{Get94} plays a prominent role. For the ns operad $\ncBV$, there is an ns  operad that plays a similar role in the picture. We shall call this operad the ns operad of \emph{noncommutative gravity algebras}. 

\begin{definition}
The ns  operad $\ncGrav$ of \emph{noncommutative gravity algebras} is the ns suboperad of $\ncGerst$ generated by the operations 
 \[
\lambda_k:=\sum_{i=1}^{k-1} m^{(k-2)}\circ_i b\ , \quad \text{for}\  \ k\ge 2\ ,  
 \]
where $m^{(0)}=\id\in\ncGerst(1)$, and for $k\ge 1$, $m^{(k)}:=m\circ_2 m^{(k-1)}\in\ncGerst(k+1)$ is the $(k+1)$-fold associative product.	
\end{definition}

\begin{theorem}\label{th:GravPresentation}
The ns operad $\ncGrav$ is isomorphic to the ns operad 
with generators $\lambda_k$ of arity $k$ and degree $1$, for $k\ge 2$, satisfying the relations
\begin{gather}
\sum_{j=r}^{r+k-2}\lambda_{n-1}\circ_j\lambda_2=\lambda_{n-k+1}\circ_r\lambda_k, \quad \text{for}\ \ n\ge 4, \ \  3\le k<n, \ \ 1\le r\le n-k+1\ ,\label{eq:ncGrav1}\\
\sum_{j=1}^{n-1}\lambda_{n-1}\circ_j\lambda_2=0, \quad \text{for}\ \  n\ge 3\ .\label{eq:ncGrav2}
\end{gather}
Moreover, these relations form a quadratic Gr\"obner basis for the ns operad $\ncGrav$ for a certain admissible ordering.
\end{theorem}

\begin{proof}
First of all, a direct computation shows that the operations $\lambda_k$ of the operad $\ncGrav$ satisfy the relations \eqref{eq:ncGrav1} and \eqref{eq:ncGrav2}. Therefore, there exists a surjection
 \[
f\colon\calP\twoheadrightarrow\ncGrav \ , 
 \]
where $\calP$ is the ns operad with the above presentation. 
Our goal is then to show that this surjection is an isomorphism. 

\smallskip

Let us first show that $\dim\calP(n)\le 2^{n-2}$ for all $n\ge 2$.  We shall use Gr\"obner bases for ns operads. Let us define an ordering of planar tree monomials in the free ns operad generated   operations $\lambda_k$ of degree $1$ and arity $k$, one generator for each $k\ge 2$. For that, we define the $2$-weight of $\lambda_2$ to be equal to $0$, the $2$-weight of all other $\lambda_i$ to be equal to $1$, and the $2$-weight of a planar tree monomial to be equal to the sum of the $2$-weights of the operations labelling its vertices. 
To compare two planar tree monomials, we first compare their $2$-weights, and if their $2$-weights are equal, compare their path sequences \cite{BD} using the lexicographic ordering of sequences (for that, we define $\lambda_p>\lambda_q$, for $p<q$).  It is straightforward to check that this order is an admissible ordering, that is  it makes operadic compositions into increasing maps.

Note that from the definition of the ordering, it is immediate that the leading terms of relations of $\calP$ in arity $n$ are $\lambda_{n-k+1}\circ_r\lambda_k$, where $3\le k<n$ and $1\le r\le n-k+1$, for \eqref{eq:ncGrav1}, and $\lambda_{n-1}\circ_1\lambda_2$, for \eqref{eq:ncGrav2}. The set $\calN$ of normal monomials with respect to these leading terms can be described as follows. The root vertex of such a monomial can be any $\lambda_p$, the planar subtree grafted at the first child of that vertex must be a leaf, and the planar subtrees grafted at all other children of that vertex are either leaves or right combs made up of vertices labelled $\lambda_2$. Therefore, the number of elements of $\calN$ with $n$ leaves is equal to the number of compositions of $n-1$ into several positive parts. The usual ``stars and bars'' enumeration method identifies those with subsets of a set consisting of $n-2$ elements. This number is equal to $2^{n-2}$, and since every Gr\"obner basis of $\calP$ contains the given relations, and the normal monomials with respect to a Gr\"obner basis form a linear basis for a ns operad, it follows that the dimension of the component $\calP(n)$ is at most $2^{n-2}$.

\smallskip

Next, we shall show that the elements of the set $\calN$ described above are linearly independent in $\calP$. For that, it is enough to show that their images under the surjection $f$ are linearly independent in $\ncGrav\subset\ncGerst$. From Proposition \ref{prop:KoszulGerst}, the underlying collection of $\ncGerst$ is isomorphic to that of $\As\circ\As_1$. In particular, $\ncGerst(n)$ has a basis  consisting of all compositions \[m^{(l-1)}(b^{(k_1-1)},\ldots,b^{(k_l-1)}) \quad \text{ with }\quad  k_1+\cdots+k_l=n,\] where the iterations of the product $b$ are defined similarly to iterations of the product $m$ above. We define an ordering of these basis elements by comparing the associated sequences $(k_1,\ldots,k_l)$ degree-lexicographically. From the definition of operations $\lambda_k$, it is immediate to see that the image of the planar tree monomial
\[
\lambda_l(\id,\lambda_2^{(k_1-1)},\ldots,\lambda_2^{(k_{l-1}-1)})  
\]
under the surjection $f$ has the leading term (with respect to the ordering we just defined)
\[
\pm m^{(l-2)}(b^{(k_1)},b^{(k_2-1)},\ldots,b^{(k_{l-1}-1)})\ ,    
\]
so the images of these monomials are linearly independent. 

Our results so far imply that the defining relations of $\calP$ form a Gr\"obner basis (otherwise, extra relations in the reduced Gr\"obner basis will be linear dependencies between elements of $\calN$), and the  elements of $\calN$ form a basis of $\calP$. This, in turn, implies that the surjection $f$ is an isomorphism, since it maps linearly independent elements to linearly independent elements. 
\end{proof}

\begin{corollary}\label{cor:GravKoszul}
The ns operad $\ncGrav$ is Koszul.
\end{corollary}	

\begin{proof}
We just checked that $\ncGrav$ is isomorphic to the quadratic ns operad $\calP$ that admits a quadratic Gr\"obner basis. It is well known~\cite{LV} that an operad with a quadratic Gr\"obner basis is Koszul. 
\end{proof}

Finally, let us obtain an homological interpretation of the ns operad $\ncGrav$ similar to that of the gravity operad  \cite{Get94}. By direct inspection, there exist unique square-zero derivations $D_1$ and $H_1$ of the operad $\ncGerst$ such that 
\begin{gather*}
D_1(m)=b,\quad  D_1(b)=0\ ,\\
H_1(b)=m, \quad H_1(m)=0\ . 
\end{gather*}
In particular, for each arity $n$, we obtain a cochain complex $(\ncGerst(n),D_1)$. 

\begin{proposition}\label{prop:ncGrav=Ker}
For each $n>1$, the cochain complex $(\ncGerst(n),D_1)$ is acyclic, and we have \[\Ker D_1=\mathop{\mathrm{Im}} D_1=\ncGrav(n).\]
\end{proposition}

\begin{proof}
In the same way as in the classical theory of the operad $\Gerst$, we observe that the cochain complex $(\ncGerst(n), \allowbreak D_1)$ is isomorphic to the Koszul complex of an operad, in this case the ns operad $\As$. Therefore, it is  acyclic since the ns operad $\As$ is Koszul. However, we need a stronger statement here.
Since the commutator of two derivations is a derivation, and since a derivation is characterized by the image of the generators, the formula 
 \[
 D_1H_1+H_1D_1=\id_{\ncGerst(2)}, 
 \]
implies that 
$$D_1H_1+H_1D_1=(n-1)\id_{\ncGerst(n)}\ ,$$ for any $n>1$. Thus, for $n>1$, the map $\frac{1}{n-1}H_1$ is a contracting homotopy between the identity and the zero map, which proves once again the acyclicity of the cochain complex $(\ncGerst(n), D_1)$. 
	
To check that $\Ker D_1$ coincides with $\ncGrav$, we first note that $\ncGrav\subset \Ker D_1$, since a simple computation shows that the generators $\lambda_k$  of $\ncGrav$ are annihilated by $D_1$. Also, the dimension of the component $\ncGrav(n)$, that is $2^{n-2}$, is precisely one half of $2^{n-1}$, the dimension of $\ncGerst(n)=H_\bullet(S^1)^{\otimes(n-1)}$, and therefore there can be no further elements in the kernel (since that would contradict the acyclicity of the cochain complex $(\ncGerst(n),D_1)$).
\end{proof}

\begin{remark}
When $G$ be a topological group, the ns operad $\As_G$ is an ns operad in the category of $G$-spaces, where the $G$-action on each component is defined by the diagonal action. Therefore, the homology operad $H_\bullet(\As_G)$ is an algebraic operad in the category of modules over the Hopf algebra $H_\bullet(G)$. 
Since here, the action of $S^1$ is free, we have, as in the classical case \cite{Get94}:
 \[
\ncGrav= \Ker D_1\cong H^{S^1}_\bullet(\As_{S^1})\cong H_{\bullet+1}(\As_{S^1}/{S^1})\ ,
 \]
where $D_1$ denotes for once the generator of  the homology algebra $H_\bullet(S^1)$. 
\end{remark}

\subsection{The minimal model of the operad $\ncBV$.}

This section follows the same general strategy of \cite{DCV2013}. Namely, in order to describe the minimal model of the ns operad $\ncBV$, one first computes the homology of the space of generators of the Koszul resolution of the operad $\ncBV$ from Proposition~\ref{prop:KoszulModel} with respect to the differential $d_\varphi$. Then one transfers the ns dg cooperad structure on that space of generators to a ns homotopy cooperad structure on this homology, thus obtaining a complete description of the minimal model of the operad $\ncBV$.\\

By Proposition~\ref{prop:KoszulGerst}, we know that $\calS\ncGerst\cong\ncGerst^!=(\calS\ncGerst^{\ac})^*\cong\calS^{-1}(\ncGerst^{\ac})^*$,
so 
 \[
\ncGerst^{\ac}\cong(\calS^2\ncGerst)^*\ . 
 \]
This means that the square-zero derivations $D_1$ and $H_1$ of the ns operad $\ncGerst$ that we defined to state and prove Proposition~\ref{prop:ncGrav=Ker} give rise to square-zero coderivations $(\calS^2D_1)^*$ and $(\calS^2H_1)^*$ on the ns cooperad $\ncGerst^{\ac}$. To emphasize the parallel with \cite{DCV2013}, we denote by $d_\psi$ the map $(\calS^2H_1)^*$, and, by $H$ the map  $\frac{1}{n-1}(\calS^2D_1)^*$, when restricted to elements of arity $n$.  From the proof of Proposition \ref{prop:ncGrav=Ker}, it is clear that $d_\psi H+Hd_\psi =\id$ on elements of arity greater than~$1$, and that $(\overline{\calS^2\ncGrav})^*=\Im\, Hd_\psi$.

\begin{proposition}
Under the isomorphism $\qncBV^{\ac}\cong T^c(\delta)\otimes\ncGerst^{\ac}$, the differential $d_\varphi$ corresponds to the differential $\delta^{-1}\otimes d_\psi$. 
\end{proposition}

\begin{proof}
Analogous to \cite[Lemma~2.17]{DCV2013}.
\end{proof}

\begin{theorem}\label{th:BVDefRetract}
There exists a deformation retract
  \[
 \xymatrix{ *{\quad \qquad\qquad \qquad \big(\overline{\qncBV^{\ac}}, \delta^{-1}\otimes d_\psi\big)\ } \ar@(dl,ul)[]^{\delta\otimes H}\ \ar@<0.5ex>[r]^-{\mathrm{P}} & *{\
 		 \big(\overline{T}^c(\delta)\oplus(\overline{\calS^2\ncGrav})^*, 0\big) \ . }  \ar@<0.5ex>[l]^-{\mathrm{I}}} 
  \]
\end{theorem}

\begin{proof}
The proof is analogous to \cite[Th.~2.1]{DCV2013} with the identifications $\qncBV^{\ac}\cong T^c(\delta)\otimes\ncGerst^{\ac}$ and $(\overline{\calS^2\ncGrav})^*=\Im\, Hd_\psi$, and the results we already proved.
\end{proof}

By the homotopy transfer theorem for dg ns cooperads, the ns collection \[\calH:=\overline{T}^c(\delta)\oplus(\overline{\calS^2\ncGrav})^*\] acquires a structure of a homotopy ns cooperad. 

\begin{corollary}
The $\infty$-cobar construction $\Omega_\infty(\calH)$ is the minimal model of the ns operad~$\ncBV$. 	
\end{corollary}

\begin{proof}
The transferred homotopy ns cooperad structure on $\calH$ is related to the dg ns (non-unital) cooperad $\overline{\ncBV^{\ac}}$ by an $\infty$-quasi-isomorphism, which proves the result.
\end{proof}

\subsection{The operad $\ncHyperCom$ as the Koszul dual of the suspension of the ns operad $\ncGrav$}

This section is inspired by the work of Getzler \cite{Get95} who established the Koszul duality between the operad $\HyperCom$ and the operad $\calS\Grav$. 

\begin{definition}
The operad of \emph{noncommutative hypercommutative algebras} is the Koszul dual of the suspension $\calS\ncGrav$ of the operad  of noncommutative gravity algebras, with respect to the presentation of Theorem~\ref{th:GravPresentation}. We denote this operad $\ncHyperCom$. 
\end{definition}

\begin{proposition}
The operad $\ncHyperCom$ is generated by operations $\nu_k\in\ncHyperCom(k)$ of degree $2k-4$, for $k\ge 2$, satisfying the following identities:
\begin{gather}\label{eq:ncWDVV}
\sum_{j=2}^{i}\nu_{n-j+1}\circ_{i-j+1}\nu_j=\sum_{k=2}^{n-i+1}\nu_{n-k+1}\circ_i\nu_k, \quad \text{for} \ \ n\ge 3, \ \ 2\le i\le n-1.
\end{gather}
\end{proposition}

\begin{remark}
In plain words, the relations \eqref{eq:ncWDVV} mean that for each triple of consecutive elements $i-1,i,i+1$, the sum of all trees with one internal edge carrying $i-1,i$ and not $i+1$ on the top level is equal to the sum of all trees with one internal edge carrying $i,i+1$ and not $i-1$ on the top level. This is indeed reminiscent of the defining relations for the operad $\HyperCom$ \cite{Get95}.
\end{remark}

\begin{proof}
Let us denote by $\calQ$ the quadratic ns operad defined by these generators and relations. 
By a direct inspection, the defining relations of this operad are orthogonal to the ones of $\calS\ncGrav$ with respect to $\nu_k=\lambda_k^*$ and the induced ``naive'' pairings of planar tree monomials.
This is  easy to check, since we exhibited a presentation for $\ncGrav$, not $\calS\ncGrav$, so there are no extra signs arising from the operadic suspension. Also, we notice that the number of quadratic relations of $\calQ$ in arity $n$ is $n-2$, and the number of quadratic relations of $\calS\ncGrav(n)$ is $1+2+3+\cdots+n-2=\binom{n}{2}-1-(n-2)$, and that both of these groups of relations are manifestly linearly independent. Since the total dimension of the space of quadratic elements of arity $n$ in the free operad with one generator of each arity $k\ge 2$ is equal to $\binom{n}{2}-1$, the annihilator of the space of quadratic relations of $\calQ$ is the space of quadratic relations of $\ncGrav$. Therefore, $\calQ=(\calS\ncGrav)^!=\ncHyperCom$.
\end{proof}

\begin{theorem}\label{th:HyperCom-dim} The operad $\ncHyperCom$ is Koszul. All its components are concentrated in even non-negative degrees, and the dimension of the graded component of degree $2k$ of  $\ncHyperCom(n)$ is equal to the Narayana number $\frac{1}{n-1}\binom{n-1}{k}\binom{n-1}{k+1}$.  In particular, the dimension of $\ncHyperCom(n)$ is equal to the Catalan number $\frac{1}{n}\binom{2n-2}{n-1}$. 
\end{theorem}

\begin{proof}
First, note that Koszulness is preserved by operadic suspension, so Corollary \ref{cor:GravKoszul} implies that the ns operad $\calS\ncGrav$ is Koszul, and hence its Koszul dual ns operad $\ncHyperCom$ is. 
	
Second, recall the normal monomials for the ns operad $\ncGrav$ constructed in Theorem~\ref{th:GravPresentation}. It is well known \cite{LV} that the Koszul dual ns operad $\calP^!$ of a ns operad $\calP$ with a quadratic Gr\"obner basis also admits a quadratic Gr\"obner basis (for the opposite order of monomials); the normal quadratic monomials for $\calP$ become the leading terms of the Gr\"obner basis of $\calP^!$, and vice versa. Thus, the leading terms of the corresponding quadratic Gr\"obner basis of $\ncHyperCom$ are $\nu_j\circ_p\nu_2$ for $p=2,\ldots,j$. 

We shall compute the dimensions of graded components of the operad $\ncHyperCom$ using generating functions.  We denote by $f_k(q,z)$ the formal power series whose coefficient at $z^nq^l$ is equal to the number of normal planar tree monomials of  degree $l$ with $n$ leaves and with the root vertex labelled by $\nu_k$. Let us also put 
\[f(q,z):=z+\sum_{k\ge 2}f_k(q,z) \quad \text{and} \quad g(q,z):=z+\sum_{k\ge3}f_k(q,z)\ ,\]
the generating series of the ns operad $\ncHyperCom$ and a variation of it.
 We note that
	\[
f_k(q,z)=q^{2k-4}f(q,z)g(q,z)^{k-1}\ ,  
	\]
since the subtree grafted at the first leaf of the root vertex may be a leaf or an arbitrary normal monomial, and all other subtrees may be leaves or normal monomials not having $\nu_2$ at the root vertex.
Adding up all these equations, we get
	\[
f(q,z)-z=\frac{f(q,z)g(q,z)}{1-q^2g(q,z)} \ .
	\]
Using the equation $f_2(q,z)=f(q,z)g(q,z)$, we see that 
$f(q,z)-g(q,z)=f(q,z)g(q,z)$,  
which gives
	\[
g(q,z)=\frac{f(q,z)}{1+f(q,z)}\ .  
	\]
Substituting this back into the equation we obtained, we get
	\[
	q^2f(q,z)^2-f(q,z)(1-z+zq^2)+z=0\ .  
	\]
This equation, after accounting for different indexings of Narayana numbers used in the literature, coincides with the standard functional equation for the generating function of Narayana numbers~\cite{Stan2}. The sum of all Narayana numbers with fixed $n$ is known to be equal to the corresponding Catalan number.
\end{proof}

\section{Geometric definition of noncommutative \texorpdfstring{$\overline{\calM}_{0,n+1}$}{M0n}}\label{sec:three-ncHyperCom}

In this section, we present two geometric constructions of a nonsymmetric version of the operad  $\{\overline{\calM}_{0,n+1}(\mathbb{k})\}$ of the moduli spaces of stable complex curves with marked points which use, respectively, toric varieties and the theory of wonderful models of subspace arrangements. We prove that for $\k=\mathbb{C}$ those definitions provide us with a topological ns operad whose homology is equal to
the algebraic operad $\ncHyperCom$  defined in the previous section. 

\subsection{Noncommutative \texorpdfstring{$\overline{\calM}_{0,n+1}$}{M0n} via brick manifolds}\label{subsec:brick}

In this section, we argue that a sensible analogue of compactified Deligne--Mumford moduli spaces is given by toric varieties of the associahedra. In particular, we prove that the collection of complex toric varieties $\{X_{\mathbb{C}}(L_n)\}$ corresponding to Loday polytopes admits a structure of a ns topological operad for which the homology operad is the operad $\ncHyperCom$.

\subsubsection{Brick manifolds}

We associate to a finite ordinal $I$ the vector space $G(I)=\k\Gap(I)$ with the basis $e_{i_1,i_2}$, where $(i_1,i_2)\in\Gap(I)$. For each interval $J=[a,b]\subset I$, we have a natural inclusion $G(J)\subset G(I)$. In such a situation, we shall also omit square brackets, and use the notation $G(a,b)$ instead of $G(J)$ for clarity. For example, $G(a,a)=\{0\}$, and $G(a,a+1)=\myspan(e_{a,a+1})$.

\begin{definition} \label{def:brick-manifold-collection}
Points of the \emph{brick manifold} (in fact an algebraic variety) $\calB(I)$ are collections of subspaces $V_{i,j}\subset G(I)$ for all proper intervals $[i,j]\subsetneq I$ that satisfy the following constraints:
\begin{itemize}
\item $\dim V_{i,j}=\# [i,j]$, 
\item $V_{i,j}\subset V_{p(i),j}$ for all $\min(I)< i$,
\item $V_{i,j}\subset V_{i,s(j)}$ for all $j<\max(I)$,
\item $V_{\min(I),j}=G(\min(I),s(j))$, 
\item $V_{i,\max(I)}=G(p(i),\max(I))$. 
\end{itemize}
\end{definition}

\begin{remark}\label{rem:codimention-one-subspace}
	Note that if $\min(I)<i,j<\max(I)$, then the space $V_{i,j}$ is a codimension one subspace of $G(p(i),s(j))$.
\end{remark}

\begin{figure}[h]
 \[
\xymatrix@M=7pt@R=7pt@C=7pt{ & & & & & &&& \\
 & & G(1,4)& &G(1,4) & &&*{} &\ar@{..}[l] \dim 3\\
  &G(1,3) \ar@{>->}[ur]& & V_{2,3} \ar@{>->}[ur]\ar@{>->}[ul]& &G(2,4) \ar@{>->}[ul]& &*{}&\ar@{..}[l]\dim 2\\
   G(1,2) \ar@{>->}[ur]& &V_{2,2}\ar@{>->}[ur] \ar@{>->}[ul]& &V_{3,3}  \ar@{>->}[ur]\ar@{>->}[ul]& &G(3,4) \ar@{>->}[ul]&*{}&\ar@{..}[l]\dim 1 \\
    & & & & & && &}
 \]
\caption{An element of $\calB(\underline{4})$} \label{Ex:B(4)}
\end{figure}
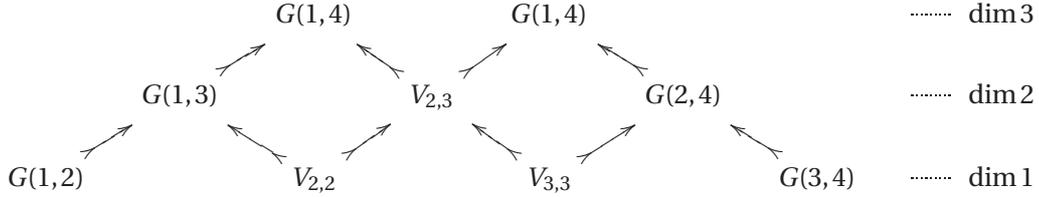
In \cite{Escobar2014}, it is proved that the brick manifold $\calB(\underline{n})$ is isomorphic to the toric variety $X(L_n)$ associated to the $n$-th Loday polytope. We shall give another proof of that result in  Theorem \ref{thm:LodayPolytope} of Appendix~\ref{sec:appendix-toric}. 

\begin{example}\label{ex:small-brick-manifolds}\leavevmode
\begin{enumerate}
\item The brick manifold $\calB(\underline{2})$ is a single point, since we only have $V_{1,1}=V_{2,2}=G(\underline{2})$.
\item The brick manifold $\calB(\underline{3})$ is the projective line $\mathbb{P}^1$, since the only variable subspace in this case is $V_{2,2}$, which can be any one-dimensional subspace of $V_{1,2}=V_{2,3}=G(\underline{3})$.  
\item The brick manifold $\calB(\underline{4})$ is the blow-up of $\mathbb{P}^1\times\mathbb{P}^1$ at a point, as one can see by noticing that a choice of subspaces $V_{2,2}\subset V_{1,2}=G(1,3)$ and $V_{3,3}\subset V_{3,4}=G(2,4)$ determines the subspace $V_{2,3}$ uniquely unless $V_{2,2}=V_{3,3}=G(2,3)$, in which case the possible choices of $V_{2,3}$ are parametrised by lines in the two-dimensional subspace $\myspan(e_{1,2},e_{3,4})$. By examining the diagram of subspaces starting from $V_{2,3}\subset G(\underline{4})$, one can also easily see that this brick manifold can be described as the blow-up of $\mathbb{P}^2$ at two points. 
\end{enumerate}	
\end{example}

Our next goal is to give the ns collection of brick manifolds a structure of a ns operad. We remark that there is a unique linear map \[f_{I,i}^J\colon G(I)\oplus G(J)\to G(I\sqcup_i J)\]
extending the bijection between the basis elements $\Gap(I)\sqcup \Gap(J)$ and $G(I\sqcup_i J)$  from  Proposition~\ref{prop:operad-gap}.

\begin{definition}\label{def:operadic-composition}
Let $I$ and $J$ be disjoint finite ordinals, and let $i\in I$. We define the map
	\[
\circ_{I,i}^J\colon\calB(I)\times\calB(J)\to\calB(I\sqcup_i J) ,
	\] 
	putting 
	\[
\left(\{V^1_{i_1,j_1}\}\circ_{I,i}^J \{V^2_{i_2,j_2}\}\right)_{a,b}:=
\begin{cases}
f_{I,i}^J\big(V^1_{a,b}\big),  \text{ for } a,b\in I, a\le b<i\  ,\\
f_{I,i}^J\big(V^2_{a,b}\big), \text{ for }  a\le b\in J, (a,b)\not=(\min(J),\max(J))\ ,\\
f_{I,i}^J\big(V^1_{a,b}\big), \text{ for } a,b\in I, i<a\le b\ , \\
f_{I,i}^J\big(V^1_{a,p(i)}\oplus 
G(\min(J), s(b))\big)\text{ for } a\in I, a<i, b\in J, b<\max(J)\ ,\\
f_{I,i}^J\big(V^1_{s(i),b}\oplus G(p(a), \max(J))
\big)\text{ for } a\in J, a>\min(J), b\in I, i<b\ ,\\
f_{I,i}^J\big(V^1_{i,i}\oplus G(J)\big)\text{ for } (a,b)=(\min(J),\max(J))\ , \\
f_{I,i}^J\big(V^1_{a,i}\oplus G(J)\big)\text{ for } 
a\in I, a<i, b=\max(J)\ , \\
f_{I,i}^J\big(V^1_{i,b}\oplus G(J)\big)\text{ for } 
b\in I, b>i, a=\min(J)\ , \\
f_{I,i}^J\big(V^1_{a,b}\oplus G(J)\big)\text{ for } a,b\in I, a<i<b\ . 
\end{cases} 
	\]
\end{definition}

\begin{example}
The example 
of the operadic composition $\circ^{\underline{4}}_{\underline{6}, 3} : \calB(\underline{6})\times 
\calB(\underline{4}) \to \calB(\underline{9})$ is depicted in Figure~\ref{Ex:OpCompo}, where we use the canonical identification 
$\underline{6}\sqcup_3\underline{4}\cong\underline{9}$.  
\end{example}
\begin{figure}[!h]
{\scriptsize
$$\xymatrix@M=0pt@R=7pt@C=2pt{
&&&&&&&G(1,9)&&G(1,9)&&&&&&& \\
&&&&&&G(1,8)&&\genfrac{}{}{0pt}{}{f(V^1_{2,5})}{\oplus G(3,6)}&&G(2,9)&&&&&& \\
&&&&&G(1,7)&&\genfrac{}{}{0pt}{}{f(V^1_{2,4})}{\oplus G(3,6)}&&\genfrac{}{}{0pt}{}{f(V^1_{3,5})}{\oplus G(3,6)}&&G(3,9)&&&&& \\
&&&&G(1,6)&&\genfrac{}{}{0pt}{}{f(V^1_{2,3})}{\oplus G(3,6)}&&\genfrac{}{}{0pt}{}{f(V^1_{3,4})}{\oplus G(3,6)}&&\genfrac{}{}{0pt}{}{f(V^1_{4,5})}{\oplus G(3,6)}&&G(4,9)&&&& \\
&&&G(1,5)&&\genfrac{}{}{0pt}{}{f(V^1_{2,2})}{\oplus G(3,6)}&&\genfrac{}{}{0pt}{}{f(V^1_{3,3})}{\oplus G(3,6)}&&
\genfrac{}{}{0pt}{}{f(V^1_{4,4})}{\oplus G(3,6)}&&\genfrac{}{}{0pt}{}{f(V^1_{4,5})}{\oplus G(4,6)}&&G(5,9)&&& \\
&&G(1,4)&&\genfrac{}{}{0pt}{}{f(V^1_{2,2})}{\oplus G(3,5)}&&G(3,6)&&G(3,6)&&\genfrac{}{}{0pt}{}{f(V^1_{4,4})}{\oplus G(4,6)}&&\genfrac{}{}{0pt}{}{f(V^1_{4,5})}{\oplus G(5,6)}&&G(6,9)&& \\
&G(1,3)&&\genfrac{}{}{0pt}{}{f(V^1_{2,2})}{\oplus G(3,4)} &&G(3,5)&&f(V^2_{2,3})&&G(4,6)&&\genfrac{}{}{0pt}{}{f(V^1_{4,4})}{\oplus G(5,6)}&&f(V^1_{4,5})&&G(7,9)& \\
G(1,2)&&f(V^1_{2,2})&&G(3,4)&&f(V^2_{2,2})&&f(V^2_{3,3})&&G(5,6)&&f(V^1_{4,5})&&f(V^1_{5,5})&&G(8,9) 
}$$
}
\caption{Example of the operadic composition $\circ^{\underline{4}}_{\underline{6}, 3} : \calB(\underline{6})\times 
\calB(\underline{4}) \to \calB(\underline{9})$.}
 \label{Ex:OpCompo}
\end{figure}

\begin{proposition} \label{proposition:operadic-composition}
The maps $\circ^J_{I,i}$ make the ns collection $\calB$ into a ns operad in the category of algebraic varieties.
\end{proposition}

\begin{proof}
First, it is easy to check that $$\left(\{V^1_{i_1,j_1}\}\circ_{I,i}^J \{V^2_{i_2,j_2}\}\right)\in \calB(I\sqcup_i J)\ ,$$ when $\{V^1_{i_1,j_1}\}\in \calB(I)$ and $\{V^2_{i_2,j_2}\} \in \calB(J)$. Then, the operadic parallel and sequential axioms are straightforward (but somewhat tedious) to check by hands. (We shall give another proof at the end of next section, when we discuss in detail the stratification of the spaces $\calB(\underline{n})$).
\end{proof}

\begin{remark}
By the results of \cite{Escobar2014}, toric varieties of other polytopal realisations of associahedra, for example the one of Chapoton--Fomin--Zelevinsky \cite{CFZ2002}, can also be realised in a similar combinatorial way. It would be interesting to see whether  the corresponding collection also exhibits a natural operadic structure.
\end{remark}

\begin{definition}
The \emph{brick operad} is the ns collection $\calB$ endowed with the operadic composition maps from Proposition~\ref{proposition:operadic-composition}.
\end{definition}

\subsubsection{Stratification} 
\label{sec:stratification}

Let us describe a stratification of the variety $\calB(\underline{n})$. It is the union of 
open strata $\calB(\underline{n},T)$ indexed by elements $T$ of the set $\calT(n)$ of all planar rooted trees with $n+1$ unbounded (external) 
edges: one for the root, and $n$ unbounded edges labelled from left to right by $1,\dots,n$ for the leaves. The valency of each vertex must be at least three, that is the number of input edges of each vertex must be at least two. For each possibly unbounded edge $e$ of the tree, we consider the set $L_e$ of all leaves of the subtree with the root $e$. 

\begin{definition} The stratum $\calB(\underline{n},T)$  consists of all collections $\{V_{i,j}\}$ satisfying the following conditions:

\begin{itemize}
\item For each edge $e$ of $T$ that is not a leaf, so that  $L_e=\{l,l+1,\dots,r-1,r\}$ for some $l<r$, we require that $V_{l,r-1} = V_{l+1,r}=G(l,r)$.

\noindent
(This implies that $V_{l,r}$, which is a $r-l+1$-dimensional space, satisfies $G(l,r)\subset V_{l,r}$).
\smallskip 

\item For each edge $e$ of $T$, which is neither the root nor the leftmost or the rightmost input edge of a vertex, denoting $L_e=\{l,l+1,\dots,r-1,r\}$ for some $l\le r$,
we require that $V_{l,r}$ is neither of the two possible $(r-l+1)$-dimensional coordinate subspaces $G(l-1, r)$ or $G(l, r+1)$ of the space $G(l-1,r+1)$.
\end{itemize}
\end{definition}

Notice that the first condition is a ``boundary'' condition and that the second one is an ``open'' condition.

\begin{remark} 
If $e$ is a leaf, which is not the leftmost or the rightmost leaf of any vertex, then the second condition implies that $V_{l,l}$, which is the same as $V_{r,r}$ in this case, is  a subspace of $G(l-1,l+1)$ different from either $G(l-1,l)$ or $G(l,l+1)$.
\end{remark}

\begin{remark}\label{rem:implication-first-condition}
The first condition, in particular, implies by Remark~\ref{rem:codimention-one-subspace}, that, for any $j=l,\dots,r-1$, we have $V_{l,j}=G(l,j+1)$
and, for any $i=l+1,\dots,r$, we have $V_{i,r}=G(i-1,r)$. In particular, if $e$ is the root, then this is just the last two requirements of Definition~\ref{def:brick-manifold-collection}.
\end{remark}

\begin{example}\label{ex:open-stratum}
Consider the planar tree $T_n\in \calT(n)$ that has no internal edges. Then for each $j=2,\dots,n-1$, the  subspace $V_{j,j}$ of $G(j-1,j+1)$ is not a coordinate line in $\calB(\underline{n},T_n)$. Therefore, all the subspaces $V_{i,j}$ are determined uniquely by the condition $V_{i,j}\supset V_{k,k}$ for $i\leq k\leq j$. So, $\calB(\underline{n},T_n)$ is the subvariety parametrized by the choice, for each $j=2,\dots,n-2$, of a subset $V_{j,j}\subset G(j-1,j+1)$ different from the two coordinate axes. This choice is parametrized by $n-2$ copies of $\mathbb{G}_m$. 
\end{example}

For $l<r$, we denote by  $T_{l,r}\in \calT(n)$ the planar tree with one internal edge, with the inputs of the corresponding vertex labelled by $l,l+1,\dots,r$.

\begin{proposition} \label{prop:InfCompoStrat}
For any $1\leqslant l<r \leqslant n$, the infinitesimal composition map 
$$\circ^{\{l,l+1,\dots,r\}}_{\{1,\dots,l-1,*,r+1,\dots,n\},*}\ : \ 
\calB(\{1,\dots,l-1,*,r+1,\dots,n\},T_{n+l-r})\times \calB(\{l,\dots,r\},T_{r-l+1}) \to 
\calB(\underline{n},T_{l,r})\ , 
$$
when  restricted respectively to the strata associated to the corollas
$T_{n+l-r}$, $T_{r-l+1}$ and to 
the  planar tree  $T_{l,r}$, 
is a bijection. 
\end{proposition}

\begin{proof}
First, it is straightforward to check that the image of the restriction of the infinitesimal composition 
to the product of strata $\calB(\{1,\dots,l-1,*,r+1,\dots,n\},T_{n+l-r})\times \calB(\{l,\dots,r\},T_{r-l+1})$ lives 
in the stratum $\calB(\underline{n},T_{l,r})$. The first condition of the stratum $\calB(\underline{n},T_{l,r})$ is given by the second point in the definition~\ref{def:operadic-composition} of the operadic composition map. The second condition for the stratum $\calB(\underline{n},T_{l,r})$ 
is a direct consequence of the same condition for the two strata $\calB(\{1,\dots,l-1,*,r+1,\dots,n\},T_{n+l-r})$ and $\calB(\{l,\dots,r\},T_{r-l+1})$: the image under the map $f^{\{l,l+1,\dots,r\}}_{\{1,\dots,l-1,*,r+1,\dots,n\},*}$ of a subspace which is not a coordinate axis is again not a coordinate axis. \\
Let us now prove that any element $\{V_{i,j}\}$ of the stratum $\calB(\underline{n},T_{l,r})$ can be written in a unique way as the infinitesimal composition of two elements
$\{V^1_{i_1,j_1}\}$ and $\{V^2_{i_2,j_2}\}$
 from the to strata $\calB(\{1,\dots,l-1,*,r+1,\dots,n\},T_{n+l-r})$ and $\calB(\{l,\dots,r\},T_{r-l+1})$ respectively. In this proof, we simply denote $f=\allowbreak f^{\{l,l+1,\dots,r\}}_{\{1,\dots,l-1,*,r+1,\dots,n\},*}$. 
 For $l\leqslant i,j \leqslant r$, we have no choice but to set 
 $V^2_{i,j}:=f^{-1}(V_{i,j})$. In the same way, for $1<k<l$ and for $r<k<n$, we have no choice but to set 
 $V^1_{k,k}:=f^{-1}(V_{k,k})$. Since $f$ is an isomorphism, there is a unique line $V^1_{*,*}\subset G(\{l-1, *, r+1\})$ such that $f(V^1_{*,*})\oplus G(l,r)=V_{l,r}$
 ; it is not a coordinate axis since the latter space is not equal to $G(l-1, r)$
 neither to $G(l, r+1)$.
Since for $1<k<n$ and $k\neq l,r$, the subspace $V_{k,k}$ is not a coordinate axis, the collection 
$\{V^2_{i_2,j_2}\}$ lives in the stratum $\calB(\{l,\dots,r\},T_{r-l+1})$ and the collection 
$\{V^1_{i_1,j_1}\}$ is completely determined by the above setting and lives in 
the stratum 
$\calB(\{1,\dots,l-1,*,r+1,\dots,n\},T_{n+l-r})$.
\end{proof}

\begin{proposition} \label{proposition:stratification}
The subvarieties $\calB(\underline{n},T)$, $T\in\calT(n)$, form a stratification of the variety $\calB(\underline{n})$. More precisely:
\begin{itemize}
\item Each subvariety $\calB(\underline{n},T)\subset \calB(\underline{n})$, $T\in\calT(n)$, is isomorphic to 
\[\calB(\underline{n},T)\cong (\mathbb{G}_m)^{n-2-n_e}\ ,\] where $n_e=n_e(T)$ is the number of the internal edges of the tree $T$. 
\item We have \[\calB(\underline{n})=\bigsqcup\limits_{T\in \calT(n)} \calB(\underline{n},T)\ .\]
\item For any $T\in\calT(n)$, the closure of $\calB(\underline{n},T)$ in $\calB(\underline{n})$ is the union of  the subvarieties  $\calB(\underline{n},T')$ of smaller dimension, that is for planar trees $T'$ such $T$ can be obtained from $T'$ by contracting of some internal (i.e. bounded) edges.
\end{itemize}
\end{proposition}

\begin{proof} One of the possible ways to prove this proposition is to introduce a toric action and to study its orbits, as we partly do in Appendix~\ref{sec:appendix-toric}. However, since our main goal is to relate this stratification with the operadic structure, we give a more direct proof here, which is recursive with respect to the structure of the planar tree $T$. 
	
\smallskip	
	
Let us explain the first statement. It is easy to see that the planar tree $T$ has precisely $n-2-n_e$ possibly unbounded edges (the root is not included here) that are not the rightmost or the leftmost inputs of a vertex. For each such edge $e$, we associate a copy of $\mathbb{G}_m$ that controls the choice of the corresponding space $V_{l,r}$. This gives a parametrization by $(\mathbb{G}_m)^{n-2-n_e}$. Let us show that there are no additional parameters. 
	
Let us denote the input edges of the root vertex by $e_1,\dots,e_k$. If we cut them off, then each of them is the root edge of a tree $T_i$, where the leaves are labelled by the index set $I_i$, $i=1,\dots,k$. We have, of course,  $I_1+\cdots+I_k = \underline{n}$, and whenever $e_i$ is a leaf, the tree $T_i$ is trivial, the only tree without vertices and just one unbounded edge (the root), and $|I_i|=1$. 

According to Remark~\ref{rem:implication-first-condition}, for each $e_i$ which is not a leaf, we have to fix the spaces $V_{\min(I_i),j}$, where $\min(I_i)\le j<\max(I_i)$ and $V_{j,\max(I_i)}$, where $\min(I_i)<j\le\max(I_i)$, unambiguously. Then the proof is completed by three simple observations:
\begin{itemize}
\item What happens with the spaces $V_{j,l}$ for $\min(I_i)<j,l<\max(I_i)$ is the subject for the induction hypothesis. In particular, those configurations of subspaces are parametrized by $(\mathbb{G}_m)^{|I_i|-2-n_i}$, where $n_i$ is the number of the internal edges of the tree $T_i$. 
\item For all other pairs of indices, we choose $V_{\min(I_i),\max(I_i)}$, $i=2,\dots,k-1$, this is parametrized by $(\mathbb{G}_m)^{k-2}$, and then we see that the conditions $V_{a,b}\subset V_{a-1,b}$ and $V_{a,b}\subset V_{a,b+1}$ given in Definition~\ref{def:brick-manifold-collection} together with natural transversality determine all other spaces $V_{a,b}$.
\item Finally, we have: 
		\[
n-2-n_e = k-2+\sum_{\substack{i=1 \\ |I_i|\geq 2}}^k \left(|I_i|-2-n_i\right).
		\]
\end{itemize}

\smallskip
	
Let us proceed to the second statement. We shall do it in a constructive way, namely we give an inductive process that allows one to associate a particular collection of spaces $\{V_{i,j}\}$ to a planar tree.

Since $V_{1,1}=G(1,2)$, $V_{n,n}=G(n-1,n)$, and $V_{i,i}\subset G(i-1,i+1)$ for all other $i$, we can always find two indices $1\leq l<r\leq n$ such that $V_{l,l}=G(l,l+1)$, $V_{r,r}=G(r-1,r)$, and $V_{i,i}$ is different from either $G(i-1,i)$ or $G(i,i+1)$ for all $l<i<r$. This implies that $V_{l,j}=G(l,j+1)$ for $l\le j<r$, and $V_{j,r}=G(j-1,r)$ for $l<j\le r$. Therefore the given collection of subspaces is in the image of the map $\circ^{\{l,l+1,\dots,r\}}_{\{1,\dots,l-1,*,r+1,\dots,n\},*}$, by the same arguments and constructions as in the proof of Proposition~\ref{prop:InfCompoStrat}. 
In the description of the stratification this corresponds to the following. The leaves with the labels $l,l+1,\dots,r$ are attached to a vertex in the tree, and are the only inputs of that vertex. We can label the ascending edge of this vertex by $*$. Therefore, on the level of configurations, we can proceed considering the configuration $\{V^1_{i,j}\}$ for the index set $I:=\{1,\dots,l-1,*,r+1,\dots,n\}$, while on the level of tree we can cut the edge $*$ and consider the rest of the tree with the leaves labelled by $I$. Now we can proceed further by induction.  Note that in this way, we simultaneously construct a tree $T$ such that $\{V_{i,j}\}\in \calB(\underline{n},T)$ and representation of this collection of subspaces as an operadic composition according to the tree $T$.	

\smallskip
	
Now let us prove the last statement. We have seen how the strata $\calB(\underline{n},T)$ are parametrized by the copies of $\mathbb{G}_m$ corresponding to the particular choices of the spaces $V_{i,j}$. The closure of the stratum is then obtained if we allow these spaces to be equal to the corresponding coordinate spaces. Let us study what happens in the limit. Since the collections of subspaces in $\calB(\underline{n},T)$ are obtained as the operadic compositions with respect to the planar tree $T$ of generic collections in $\calB(I)$, where $I$ is an interval of~$\underline{n}$, it is sufficient to prove this statement for the planar tree with no internal edges. This tree corresponds to the open stratum in $\calB(I)$ from Example~\ref{ex:open-stratum}, so in this case the statement is obvious.  
\end{proof}

\begin{corollary} \label{cor:closure-of-divisor}
For any $1\leqslant l<r \leqslant n$, the closure $\overline{\calB(\underline{n},T_{l,r})}$ of $\calB(\underline{n},T_{l,r})$ in $\calB(\underline{n})$ consists of all collections $\{V_{i,j}\}$ such that $V_{l,l}=G(l,l+1)$ and $V_{r,r}=G(r-1,r)$. Moreover, 
 the infinitesimal composition map 
$$\circ^{\{l,l+1,\dots,r\}}_{\{1,\dots,l-1,*,r+1,\dots,n\},*}\ : \ 
\calB(\{1,\dots,l-1,*,r+1,\dots,n\})\times \calB(\{l,\dots,r\}) \to 
\overline{\calB(\underline{n},T_{l,r})}\ , 
$$
with the image   restricted respectively to the closure of the stratum associated to 
the  planar tree  $T_{l,r}$, 
is a bijection. 
\end{corollary}

\begin{proof}
This follows from the description of the stratification from Proposition~\ref{proposition:stratification}; the proof uses the exact same arguments as in Proposition~\ref{prop:InfCompoStrat}. 
\end{proof}

\begin{remark}
The description of the stratification in terms of planar trees given in Proposition~\ref{proposition:stratification} and its relation to the operadic composition described in Corollary~\ref{cor:closure-of-divisor} imply directly Proposition~\ref{proposition:operadic-composition}.
\end{remark}

\subsubsection{Local structure} \label{sec:local-structure}

Here we describe how the strata are attached to each other. It is clear (by induction on the number of edges) that the general case would follow from the description of the normal bundle of $\overline{\calB(\underline{n},T_{l,r})}$ in $\calB(\underline{n})$.

Consider the two-dimensional vector bundle over $\calB(\underline{n})$ given by $V_{l,r}/V_{l+1,r-1}$. There are two natural lines in this space given by $V_{l+1,r}/V_{l+1,r-1}$ and $V_{l,r-1}/V_{l+1,r-1}$. The equation that defines (locally) $\overline{\calB(\underline{n},T_{l,r})}$ in $\calB(\underline{n})$ is that these two lines coincide. In other words, the natural map 
 \[
V_{l+1,r}/V_{l+1,r-1}\to V_{l,r}/V_{l,r-1} 
 \]
must vanish. In yet other words, the normal bundle of $\overline{\calB(\underline{n},T_{l,r})}$ in $\calB(\underline{n})$ is isomorphic to  
 \[
(V_{l+1,r}/V_{l+1,r-1})^*\otimes (V_{l,r}/V_{l,r-1}), 
 \]
which is in this case is equal to
 \[
\left(V_{l,r}/G(l,r)\right)\otimes \left(G(l,r)/V_{l+1,r-1}\right)^*.
 \]

This local analysis implies immediately the following statement (that we also know from the study of the torus action, as in Appendix~\ref{sec:appendix-toric}).

\begin{proposition}\label{prop:normalcrossingdivisor}
The variety $\calB(\underline{n})$ is a smooth compactification of $\calB(\underline{n},T_n)$ such that the complement $\calB(\underline{n})\setminus \calB(\underline{n},T_n)$ is a normal crossing divisor.
\end{proposition}

In Section \ref{subsec:wonderful} below, we shall also discuss this result from the point of view of wonderful models of De Concini and Procesi \cite{DCP95}.

\subsubsection{Homology of complex brick manifolds}

In this section, we prove the central result of this part of the paper: the complex brick operad is a topological model of the operad $\ncHyperCom$. 

\begin{theorem} \label{th:homology-ncHyperCom}
The homology of the ns complex brick operad is the ns operad of noncommutative hypercommutative algebras:
 \[
H_\bullet(\calB_{\mathbb{C}})\cong \ncHyperCom.  
 \]
\end{theorem}  

\begin{proof}
It is obvious that the fundamental classes of the strata $\calB_{\mathbb{C}}(\underline{n},T)$, $T\in\calT(n)$ are the additive generators of the homology of $\calB_{\mathbb{C}}(\underline{n})$. Moreover, Corollary~\ref{cor:closure-of-divisor} and its inductive corollaries for the planar trees with an arbitrary number of edges imply that the induced ns operadic structure on $H_\bullet(\calB_{\mathbb{C}})$ is a quotient of the natural ns operadic structure on the linear span of the collection of all planar trees $\{\calT(n)\}$, which is the free ns operad on one generator per arity. 
	
Since the components of the $h$-vector of the associahedron are Narayana numbers \cite{CWLee89}, Proposition \ref{prop:h-vector} and Theorem \ref{th:HyperCom-dim} show that each component $H_\bullet(\calB_{\mathbb{C}}(\underline{n}))$ of the ns operad $H_\bullet(\calB_{\mathbb{C}})$ has the same dimensions of graded components as $\ncHyperCom(n)$. Thus, it suffices to show that the relations of the operad $\ncHyperCom$ hold. 

There is a map $\calB_{\mathbb{C}}(\underline{n})\to \mathbb{P}(G(i-1,i+1))$ given by $V_{i,i}$. The preimages of any two points in $\mathbb{P}(G(i-1,i+1))$ are homologous. Now observe that the preimage of $G(i-1,i)$ is the union of all divisors $\overline{\calB_{\mathbb{C}}(\underline{n},T)}$, where $T$ is a tree with one internal edge such that the leaf labelled by $i$ and the root are attached to different vertices and the leaf $i$ is the rightmost at its incident vertex. The preimage of of $G(i,i+1)$ is the union of all divisors $\overline{\calB_{\mathbb{C}}(\underline{n},T)}$, where $T$ is a tree with one internal edge such that the leaf labelled by $i$ and the root are attached to different vertices and the leaf $i$ is the leftmost at its vertex. The equality of the classes of these preimages is precisely given by Equation \eqref{eq:ncWDVV}. 
\end{proof}

\subsubsection{Geometrical proof of the Koszul property}
In the same way as for the Deligne--Mumford  operad  $\{\overline{\calM}_{0,n+1}\}$ in \cite{Get95}, one can prove the Koszul property of the ns operad $\ncHyperCom$ using mixed Hodge structures with the following general me\-thod. 

First of all, since the brick manifolds $\calB_{\mathbb{C}}(n)$ are  complex algebraic varieties, their cohomology groups $H^k(\calB_{\mathbb{C}}(n))$ admit a functorial mixed Hodge structure by \cite[Proposition 8.2.2]{delignehodge3}. So the ns operad $\ncHyperCom$ can be promoted to an ns operad in the category of graded mixed Hodge structures. Then, since the stratum indexed by $T_{l,r}$ in $\calB_{\mathbb{C}}(\underline{n})$ has the following form
 \[
\calB_{\mathbb{C}}(\underline{n}, T_{l,r})\cong \calB_{\mathbb{C}}(\underline{n-r}, T_{n-r})\times\calB_{\mathbb{C}}(\underline{r+1}, T_{r+1})\ ,
 \]
the theory of logarithmic forms along a normal crossing divisor \cite[Section~$3.1$]{delignehodge2} implies, by Proposition~\ref{prop:normalcrossingdivisor}, the existence of the residue morphisms 
 \[
H^{k+l-1}\left(\calB_{\mathbb{C}}(\underline{n}, T_{n})\right)(-1)\rightarrow H^{k-1}\left(\calB_{\mathbb{C}}(\underline{n-r}, T_{n-r})\right)(-1)
 \otimes H^{l-1}\left(\calB_{\mathbb{C}}(\underline{r+1}, T_{r+1})\right)(-1)\ .
 \]
These morphisms assemble into a ns cooperad structure in the category of mixed Hodge structures on the collection 
  \[
\mathcal{S}^cH^{\bullet-1}\left(\calB_{\mathbb{C}}(\underline{n}, T_{n})\right)(-1) \ .
 \]
\begin{proposition}
The ns operad $H_\bullet(\calB_{\mathbb{C}})\cong\ncHyperCom$ is Koszul and its  Koszul dual cooperad $\ncGrav$ is isomorphic to the above-mentionned ns cooperad 
$$\ncHyperCom^{\ac}\cong \mathcal{S}^c \ncGrav^*
\cong \mathcal{S}^cH^{\bullet-1}\left(\calB_{\mathbb{C}}(\underline{n}, T_{n})\right)(-1) \ .$$
\end{proposition}
\begin{proof}
Since the ns operad $\calB_{\mathbb{C}}$ is made up of  components which are  smooth  projective complex algebraic 
varieties obtained from $\calB_{\mathbb{C}}(\underline{n}, T_{n})$ by a compactification with a normal crossing divisor, one can consider the Deligne's spectral sequence \cite[3.2]{delignehodge2}, which gives here 
$$E_1^{p,q}=\bigoplus_{T\in \calT^{[p]}(n)}H^{q-p} \left(\calB_{\mathbb{C}}(\underline{n}, T)\right)(-p)\ \Longrightarrow \ 
H^{p+q} \left(\calB_{\mathbb{C}}(\underline{n})\right)\ ,$$
where $\calT^{[p]}(n)$ is the set of  planar rooted trees with $p$ internal edges. Its first  differential map  $d_1\colon E_1^{p,q}\to E_1^{p+1,q}$ is the sum of the residue morphisms 
\[H^{q-p}\left(\calB_{\mathbb{C}}(\underline{n}, T\right)(-p) \to 
H^{q-p-1}\left(\calB_{\mathbb{C}}(\underline{n}, T'\right)(-p-1)\]
for any planar tree $T'$ such that $T$ can be obtained from $T'$ by contracting of one internal  edge. 
Therefore, the $q$-th row $E_1^{\bullet, q}$ of the first page of this spectral sequence, equal to 
\[0\to  H^q\left(\calB_{\mathbb{C}}(\underline{n}, T_n)\right) \to
 \bigoplus_{T\in \calT^{[1]}(n)}H^{q-1}\left(\calB_{\mathbb{C}}(\underline{n}, T)\right)(-1) \to
  \bigoplus_{T\in \calT^{[2]}(n)} H^{q-2}\left(\calB_{\mathbb{C}}(\underline{n}, T)\right)(-2) \to
   \cdots \ ,\]
   is a direct summand  of the cobar construction of the ns cooperad $\mathcal{S}^cH^{\bullet-1}\left(\calB_{\mathbb{C}}(\underline{n}, T_{n})\right)\allowbreak(-1)$. 
The key point now lies in the fact that each  cohomology group
 $H^{k}\left(\calB_{\mathbb{C}}(\underline{n},T_n)\right)$ has a mixed Hodge structure concentrated in Tate weight $2k$, since $\calB_{\mathbb{C}}(\underline{n},T_n)\cong (\mathbb{C}^\times)^{n-2}$, and that condition manifestly holds for $\mathbb{C}^\times$. This implies that the Deligne spectral sequence collapses at the second page, thereby inducing the following long exact sequence 
\begin{multline*}
0\to  H^q\left(\calB_{\mathbb{C}}(\underline{n}, T_n)\right) \to
 \bigoplus_{T\in \calT^{[1]}(n)}H^{q-1}\left(\calB_{\mathbb{C}}(\underline{n}, T)\right)(-1) \to
  \bigoplus_{T\in \calT^{[2]}(n)} H^{q-2}\left(\calB_{\mathbb{C}}(\underline{n}, T)\right)(-2) \to
   \cdots \\
\cdots \to   \bigoplus_{T\in \calT^{[q]}(n)} H^{0}\left(\calB_{\mathbb{C}}(\underline{n}, T)\right)(-q) \to  H^{2q}\left(
\calB_{\mathbb{C}}(\underline{n})
\right)\to 0\ .
   \end{multline*}
   In operadic terms and using the homological degree convention, this means, first,  that the homology of the cobar construction of the ns cooperad $\mathcal{S}^cH^{\bullet-1}\left(\calB_{\mathbb{C}}(\underline{n}, T_{n})\right)(-1)$ is concentrated in syzygy degree $0$, statement equivalent to the fact that it is Koszul. Then, the homology of the cobar construction of $\mathcal{S}^cH^{\bullet-1}\left(\calB_{\mathbb{C}}(\underline{n}, T_{n})\right)(-1)$ is isomorphic to the ns operad 
   $H_\bullet(\calB_{\mathbb{C}})\cong \ncHyperCom$, which is its Koszul dual; therefore we have an isomorphism of ns cooperads
   $$\mathcal{S}^c \ncGrav^*
\cong \mathcal{S}^cH^{\bullet-1}\left(\calB_{\mathbb{C}}(\underline{n}, T_{n})\right)(-1)\ ,$$ which proves the two claims of the statement. 
\end{proof}

\subsection{The operad $\ncHyperCom$ and wonderful models of hyperplane arrangements}\label{subsec:wonderful}

In this section, we demonstrate that brick manifolds can be viewed as projective wonderful models of subspace arrangements in the sense of de Concini and Procesi~\cite{DCP95}.  Our exposition follows closely that of Rains \cite{Rains10}, where, in particular, the operad-like structures of wonderful models are explained in detail. 

\subsubsection{Wonderful models of subspace arrangements}

Let $V$ be a finite-dimensional vector space. A \emph{subspace arrangement} in   $V$ is a finite collection $\calG$ of subspaces of $V^*$. We consider the lattice $L_\calG$ generated by $\calG$, it is the set of all sums of subsets of $\calG$, including the empty sum~$0$. A \emph{$\calG$-decomposition} of $U\in L_\calG$ is a collection of non-empty subspaces $U_i\in L_\calG $ for which 
 \[
U=\bigoplus_i U_i,  
 \]
and such that, for every $G\in\calG$ satisfying $G\subset U$, we have $G\subset U_i$ for some $i$. If we denote by $\overline{\calG}$ the set of all $\calG$-indecomposable subspaces, then $L_\calG=L_{\overline{\calG}}$, and the notions of $\calG$-indecomposable elements and $\overline{\calG}$-indecomposable elements coincide. A subspace arrangement for which $\calG=\overline{\calG}$ is called a \emph{building set}. Building sets form a poset with respect to inclusion; this poset always has exactly one minimal element, and exactly one maximal one.

Let us assume that $\calG$ is a building set. We denote by
$M_\calG$ the complement  in $V$ of the union of all the subspaces $G^\bot$ for $G\in\calG$, and by $\widehat{M}_\calG$ the complement in $\mathbb{P}(V)$ of all the subspaces $\mathbb{P}(G^\bot)$ for $G\in\calG$.  

For each $G\in\calG$, we have an obvious map $\psi_G\colon M_\calG\to\mathbb{P}(V/G^\bot)$, and therefore the maps
 \[
\psi\colon M_\calG\to V\times\prod_{G\in\calG} \mathbb{P}(V/G^\bot) 
 \]
and 
 \[
\widehat{\psi}\colon\widehat{M}_\calG\to\mathbb{P}(V)\times\prod_{G\in\calG} \mathbb{P}(V/G^\bot) 
 \]
\begin{definition}[Wonderful model]
The \emph{wonderful model} $Y_\calG$ is defined by the closure of $\psi(M_\calG)$ in $V\times\prod_{G\in\calG} \mathbb{P}(V/G^\bot)$, and the \emph{projective wonderful model} $\widehat{Y}_\calG$ is defined by the closure of $\widehat{\psi}(\widehat{M}_\calG)$ in $\mathbb{P}(V)\times\prod_{G\in\calG} \mathbb{P}(V/G^\bot)$. 
\end{definition}

The most important geometric properties of the projective wonderful models are summarised in the following proposition. Recall that a set $\mathsf{S}$ of subspaces in $V^*$ is called \emph{nested} if given any subspaces $U_1,\ldots,U_k\in\mathsf{S}$ which are pairwise not comparable (by inclusion), they form a direct sum, and their direct sum is not in $\mathsf{S}$.

\begin{proposition}[\cite{DCP95}]\label{prop:DeConciniProcesi}\leavevmode
\begin{enumerate}
\item The projective wonderful model $\widehat{Y}_\calG$ is a smooth projective irreducible variety. The natural projection map $\pi\colon \widehat{Y}_\calG\to\mathbb{P}(V)$ is surjective, and restricts to an isomorphism on $\widehat{M}_\calG$.
\item The complement $\widehat{D}=\widehat{Y}_\calG\setminus \pi^{-1}(\widehat{M}_\calG)$ is a divisor with normal crossings. Its irreducible components $\widehat{D}_G$ are in one-to-one correspondence with elements $G\in\calG\setminus\{V^*\}$, and we have 
	\[
\pi^{-1}(\mathbb{P}(G^\bot))=\bigcup_{G\subset F}\widehat{D}_F \ .  
	\]
\item For a subset $\mathsf{S}$ of $\calG\setminus\{V^*\}$, the intersection
	\[
\widehat{D}_\mathsf{S}=\bigcap_{X\in\mathsf{S}} \widehat{D}_X 
	\] 
is non-empty if and only if the set $\mathsf{S}$ is nested; in this case $\widehat{D}_\mathsf{S}$ is irreducible. 
\end{enumerate}
\end{proposition}

The original motivation of de Concini and Procesi came from the example of the \emph{braid arrangement}
 \[
A_{n-1}:=\big\{H_{i,j}=\myspan(x_i-x_j)\mid 1\le i,j\le n\big\}  
 \]
in $(\k^n)^*$. The intersection of all the hyperplanes $H_{i,j}^\bot$ in $\k^n$ is the subspace $N=\k(1,1,\ldots,1)$, and we can consider $V=\k^n/N$. We note that $A_{n-1}\subset V^*\subset(\k^n)^*$, so $A_{n-1}$ can be regarded as a subspace arrangement in $V$. For a building set, we consider the minimal one, that is the collection of all subspaces
 \[
H_I=\myspan(x_i-x_j\mid i,j\in I) 
 \]
for subsets $I\subset\{1,2,\ldots,n\}$ with $|I|\ge 2$. It turns out that the corresponding projective wonderful model is isomorphic to the Deligne--Mumford compactification of the moduli space of genus $0$ curves with marked points:
 \[
\widehat{Y}_{A_{n-1},\min}\cong \overline{\calM}_{0,n+1}. 
 \]

\subsubsection{Noncommutative braid arrangements and their wonderful models}
\label{sec:WonderfulModel}

We are now ready to interpret brick manifolds as wonderful models.

\begin{definition}
Let $I$ be a finite ordinal. The \emph{noncommutative braid arrangement} is the subspace arrangement 
 \[
ncA_{n-1}:=\big\{H_{i,i+1}=\myspan(x_{i}-x_{i+1})\mid 1\le i,i+1\le n\big\}  
 \]
in $(\k^n)^*$. 
\end{definition}

Similar to the usual braid arrangement, the intersection of all the hyperplanes $H_{i,i+1}^\bot$ in $\k^n$ is the subspace
$N=\k(1,1,\ldots,1)$, and we can consider $V=\k^n/N$. Once again, we have $ncA_{n-1}\subset V^*\subset(\k^n)^*$, so $ncA_{n-1}$ can be regarded as a subspace arrangement in $V$. For a building set, we consider again the minimal one, that is the collection of all subspaces
 \[
H_I=\sum_{i,i+1\in I} H_{i,i+1} \ ,
 \]
for \emph{intervals} $I\subset\underline{n}$ with $|I|\ge 2$.

\begin{theorem}\label{th:wonderful-brick}
The minimal projective wonderful model for the noncommutative braid arrangement is isomorphic to the brick manifold: 
 \[
\widehat{Y}_{ncA_{n-1},\min}\cong \calB(\underline{n})\ . 
 \]
\end{theorem}

\begin{proof}
Let us denote by $\{e_i\}$ the basis of $\k^n$ dual to the basis $\{x_i\}$ of $(\k^n)^*$, and by~$p$ the natural projection from $\k^n$ to $V=\k^n/N$. The minimal projective wonderful model of $ncA_I$ is the closure of the image of 
	\[
\widehat{M}_{ncA_{n-1}}=\mathbb{P}(V)\setminus\bigcup_{I} \mathbb{P}(p(H_I^\bot))
	\] 
in the product  
	\[
\mathbb{P}(V)\times \prod_{I \text{ an interval of } \underline{n},\, |I|\ge 2}\mathbb{P}(V/p(H_I^\bot)) .
	\]
Let us note that $H_I^\bot$ is the subspace spanned by the vectors $e_i$ with $i\notin I$, and by the vector $e_I=\sum_{i\in I}e_i$. Therefore the cosets of the vectors
	\[
v_k^I:=\sum_{i\in I, j<k}e_i\ ,
	\]
for all $k>\min(I)$, form a basis in $\k^n/H_I^\bot\cong V/p(H_I^\bot)$. Clearly, if $I\subset J$ are two intervals of $\underline{n}$, then $H_I\subset H_J$ and $H_J^\bot\subset H_I^\bot$, so there is a canonical projection 
	\[
p_{I,J}\colon V/p(H_J^\bot)\cong \k^n/H_J^\bot\to \k^n/H_I^\bot\cong V/p(H_I^\bot)\ ,
	\]
and we have 
	\[
p_{I,J}(v_i^J)=v_i^I  \ ,
	\]
for all $i\in I$. By direct inspection, we see that for each vector $v=\sum_{i=1}^n c_ie_i\in\k^n$, the vector
\[
\sum_{\min(I)< i\in I}(c_{p(i)}-c_i)v_i^I
\]
in $V/p(H_I^\bot)$ is in the same coset as $v$. Let us identify, for each $I$, the space $H_I$ with the space $G(I)$ in the most obvious way: $x_{i}-x_{i+1}\leftrightarrow e_{i,s(i)}$.  By projective duality a point in $\mathbb{P}(V/p(H_I^\bot))\cong\mathbb{P}(\k^n/H_I^\bot)$ corresponds to a hyperplane $\alpha_I\subset H_I\cong G(I)$. We define, for each pair $(i,j)$ with $1<i\le j<n$, a subspace $V_{i,j}\subset G(I)\subset G(\underline{n})$ by the formula
 \[
V_{i,j}=\alpha_{[p(i),s(j)]}\  . 
 \]
Because of the compatibility of our bases $\{v_k^I\}$ with the canonical projections, it is immediate that for each point of $\widehat{Y}_{ncA_{n-1},\min}$, viewed as a point in the product of projective spaces, the collection of the subspaces $V_{i,j}$ that we just constructed defines a point of the brick manifold. Moreover, this map is one-to-one: by viewing the space $V_{i,j}$ as a hyperplane in $G(p(i),s(j))$, we can assign to it a point in the appropriate projective space, so each point of the brick manifold gives rise to a point in the product of projective spaces where the wonderful model is defined. This completes the proof.    
\end{proof}

\begin{remark}\label{rem:IteratedBlowup}
Recall from \cite{DCP95} that wonderful models can also be constructed as iterated blow-ups of $\mathbb{P}(V)$. In the case we consider, we have to use all the subspaces $\mathbb{P}(p(H_I^\bot))$ as centers of blow-ups, in the following order: first we blow up the two points $\mathbb{P}(p(H_I^\bot))$, $|I|=n-1$, then we blow up the proper transforms of the three lines $\mathbb{P}(p(H_I^\bot))$, $|I|=n-2$, then we blow up the proper transforms of the four planes $\mathbb{P}(p(H_I^\bot))$, $|I|=n-3$, and so on, up to the proper transforms of $n-1$ hyperplanes $\mathbb{P}(p(H_I^\bot))$, $|I|=2$, whose blow-up does nothing to the variety. (For example, this describes  $\calB(\underline{4})$ as a blow-up of $\mathrm{P}^2$ at two points, a description which we already discussed in Example~\ref{ex:small-brick-manifolds}). Note that
the projective subspaces $\mathbb{P}(p(H_I^\bot))$ for fixed $|I|=k$ intersect transversally along the projective subspaces $\mathbb{P}(p(H_J^\bot))$ for $|J|>k$, which  guarantees that, in our sequence of blow-ups, we always blow-up non-singular subvarieties.
\end{remark}

The identification of two constructions of the wonderful model in~\cite{DCP95} and the identification of this particular wonderful model with the brick manifold that we just obtained implies that the open part of the exceptional divisor over the subspace $\mathbb{P}(p(H_J^\bot))$ is the stratum $\calB(\underline{n},T_{\min(J),\max(J)})$. On the other hand, it follows from \cite[Th.~4.3]{DCP95} that for each decomposition $\underline{n}=I\sqcup_i J$, the irreducible component $\widehat{D}_{H_J}$ of the boundary divisor $\widehat{D}$ of $\widehat{Y}_{ncA_{n-1},\min}$ can be naturally identified with the product $\widehat{Y}_{ncA_{|I|-1},\min}\times\widehat{Y}_{ncA_{|J|-1},\min}$. This leads to a ns operad structure on the collection of wonderful models. In the same way that this is done for the operad structure on the collection of spaces $\overline{\calM}_{0,n+1}$ obtained via wonderful models, see, e.g., \cite{Rains10}, one can show that this ns operad structure coincides with the ns operad structure on brick manifolds. Overall, this proves the following result (which is completely analogous to the corresponding classical result).

\begin{theorem}
The homology of the ns operad of complex projective wonderful models of the noncommutative braid arrangements is isomorphic to the ns operad of noncommutative hypercommutative algebras:
 \[
\big\{H_\bullet(\widehat{Y}_{ncA_{n-1},\min}(\mathbb{C}))\big\}\cong\ncHyperCom . 
 \]
\end{theorem}

It is known that the toric variety of the permutahedron, also known as the Losev--Manin moduli space $\overline{\calL}_n$, can be realised as a wonderful model of the coordinate subspace arrangement in $\k^n$ \cite{Hen12,Pro90}. Let us mention here that this variety can be also viewed as a wonderful model of the noncommutative braid arrangement $ncA_{n-1}$. Namely, the following proposition holds.

\begin{proposition}
The maximal projective wonderful model of $ncA_{n-1}$, i.e. the model corresponding to the maximal building set, that is the set of all possible sums of the subspaces $H_{i,i+1}$, is the Losev--Manin moduli space:
\[
\widehat{Y}_{ncA_{n-1},\max}\cong \overline{\calL}_{n-1}. 
\]
\end{proposition}

\begin{proof}
Up to a change of coordinates, this is the result of \cite[Sec.~3]{Pro90}, see also \cite[Prop.~2.9]{Hen12}.
\end{proof}

\begin{remark}
The fact that the toric varieties of permutahedra and associahedra are both wonderful models of the same arrangement should not be too surprising: the permutahedron is obtained from the associahedron by truncations, hence the toric variety of the permutahedron is obtained from the toric variety of the associahedron by blow-ups, exactly in the way that it is for the respective wonderful models. 
\end{remark}

\section{Givental group action on \texorpdfstring{$\ncHyperCom$}{ncHyperCom}-algebras}\label{sec:Givental}

In this section, we develop a convenient framework for intersection theory on brick manifolds, and use it to define a noncommutative version of the celebrated Givental group 
action \cite{Giv01}. We also identify the latter action defined geometrically via intersection theory with the gauge symmetries action in the homotopy Lie algebra controlling deformations of homotopy $\ncBV$-algebras.

\subsection{Intersection theory for complex brick manifolds}

In this section, we shall present a somewhat elegant description of the intersection product in the cohomology rings of complex brick manifolds, which we find particularly convenient for our purposes, especially for defining the Givental group action on $\ncHyperCom$-algebra structures. Of course, in principle, the   results of the previous section suggest at least two other ways to describe those cohomology rings. 
First, since brick manifolds are toric varieties of Loday polytopes (Appendix~\ref{sec:appendix-toric}), one can use the general results on cohomology rings of toric varieties \cite{Danilov78,Fulton93} to obtain presentations of those rings. (For the other realisations from~\cite{CFZ2002}, which are different from the Loday polytope realisations, this way to describe the corresponding cohomology rings was undertaken in~\cite{Chapoton05}). Second, since brick manifolds are wonderful models of the noncommutative braid arrangements, one can use the results of Feichtner and Yuzvinsky \cite{FY04} to obtain presentations of the cohomology rings. \\

Note that any pair of strata $\calB_{\mathbb{C}}(\underline{n},T_1)$ and $\calB_{\mathbb{C}}(\underline{n},T_2)$ either 
\begin{itemize}
\item intersect transversally, when there exists a tree $T_3$ that can be turned  into $T_1$ and $T_2$ by contraction of two disjoint subsets of edges, 
\item do not intersect at all, when there is no tree with a bigger number of edges that can be turned  into both $T_1$ and $T_2$ by contractions, 
\item or, intersect non-transversally, when there exists a tree $T_3$ that can be turned  into $T_1$ and $T_2$ by contraction of two intersecting subsets of edges. 
\end{itemize}
So, in order to
have a complete description of the intersection product, it is enough
to describe the self-intersections of divisors
$\overline{\calB_{\mathbb{C}}(\underline{n},T_{l,r})}$. It turns out that the description is very similar to that in the commutative case, and is best described via $\psi$-classes, i.e. the first Chern classes of appropriate line bundles.

\subsubsection{Noncommutative $\psi$-classes and topological recursion relations}

\begin{definition}\label{def:psi-classes} We define the line bundles $\calL_0,\calL_1,\dots,\calL_n$ on the complex brick manifold $\calB_{\mathbb{C}}(\underline{n})$ as follows. The fibre of the line bundle $\calL_0$ over a collection $\{V_{i,j}\}$ is given by $G(\underline{n})/V_{2,n-1}$. The fibre of of the line bundle  $\calL_j$, for $j=1,\dots,n$, is given by $V_{j,j}^*$. We denote by $\psi_j$ the first Chern class of $\calL_j$, $j=0,1,\dots,n$. 
\end{definition}

\begin{proposition}[Topological recursion relations]\label{prop:TRR}\leavevmode
	\begin{itemize}
		\item The $0$th $\psi$-class is equal to 
		$\psi_0=\sum_T [\overline{\calB_{\mathbb{C}}(\underline{n},T)}]$, where the sum is taken over all trees with one internal edges such that the leaves labelled by $i$ and $i+1$ are attached to a different vertex than the root. This formula is valid for any $i=1,\dots,n-1$. 
		\item For $j=1,\dots,n$, the $j$th $\psi$-class is equal to 
		$\psi_j=\sum_T [\overline{\calB_{\mathbb{C}}(\underline{n},T)}]$, where the sum is taken over all trees with one internal edges such that the leaf labelled by $j$ is attached to a different vertex than the root and the leaf labelled by $j+1$.  
		\item For $j=1,\dots,n$, the $j$th $\psi$-class is equal to 
	 $\psi_j=\sum_T [\overline{\calB_{\mathbb{C}}(\underline{n},T)}]$, where the sum is taken over all trees with one internal edges such that the leaf labelled by $j$ is attached to a different vertex than the root and the leaf labelled by $j-1$.  
	\end{itemize}
\end{proposition}

In particular, $\psi_1=\psi_n=0$. 

\begin{proof}
In order to compute, $\psi_0$ we have to choose a section of $\calL_0$ that is transversal to the zero section. Let us choose the section given by $e_{i,i+1}$. This section is equal to zero if and only if $G(i,i+1)\subset V_{2,n-1}$, which is possible if and only if the corresponding collection is in a divisor $\overline{\calB_{\mathbb{C}}(\underline{n},T_{l,r})}$, where $V_{l+1,r}=V_{l,r-1}\supset G(i,i+1)$. This gives the desired formula. 

For $j=1,\dots,n$, the line bundle $\calL_j$ is the pullback of the line bundle $O(1)$ on $\mathbb{P}(G(j-1,j+1))$ under the map $\{V_{i,j}\}\mapsto V_{j,j}$. So, the characteristic class $\psi_j$ is given by the pullback of any point on $\mathbb{P}(G(j-1,j+1))$. Choosing one of the two coordinate axes, we obtain the two possible formulas (cf. the argument in the proof of Theorem~\ref{th:homology-ncHyperCom}).
\end{proof}

\subsubsection{The pull-back formula} There are two natural maps from $\calB_{\mathbb{C}}(\underline{n+1})$ to $\calB_{\mathbb{C}}(\underline{n})$, $n\geq 2$, that we can ``the maps forgetting the marked points'' using the analogy with the usual moduli spaces of genus $0$ curves. Of course, in the nonsymmetric setting it is only fully natural to forget the leftmost or the rightmost marked point, and though we don't have the marked points directly in our geometric interpretation, this intuition is sufficient to propose the following definition.
\begin{definition} The \emph{left projection map} 
\begin{equation*}
\pi_L\colon \calB_{\mathbb{C}}(\underline{n+1})\to \calB_{\mathbb{C}}([2,n+1])\cong \calB_{\mathbb{C}}(\underline{n})
\end{equation*}
takes a point represented by a collection of subspaces $V_{i,j}$ forgets all spaces $V_{1,j}$ for $i=1,\dots,n-1$, forgets $V_{2,n}$, and replaces the spaces $V_{2,i}$ by $G_{2,s(i)}$ for $i=1,\dots,n-1$. The \emph{right projection map} 
\begin{equation*}
\pi_R\colon \calB_{\mathbb{C}}(\underline{n+1})\to \calB_{\mathbb{C}}(\underline{n})
\end{equation*}
takes a point represented by a collection of subspaces  $V_{i,j}$ forgets all spaces $V_{i,n+1}$ for $i=2,\dots,n$, forgets $V_{1,n}$, and replaces the spaces $V_{i,n}$ by $G_{p(i),n}$ for $i=2,\dots,n$.
\end{definition}
These maps are obviously complex analytic and surjective. 
The main feature of these maps is that with them we can correctly reproduce the natural noncommutative analogues of the pull-back formulas for the $\psi$-classes.  Recall that in the case of the usual moduli spaces of curves the difference between the $\psi$-class and the pull-back of a $\psi$-class at the point with the same label $i$ under the map that forgets the point labelled by $j$ is the natural divisor represented by a tree with two vertices, where the leaves $i$ and $j$ are attached to one of the vertices, and all other leaves are attached to the other vertex.

Extending this analogy in the case of the map $\pi_L$ we can think that we forget the point labelled by $1$ and then there are only two such natural divisors in $\calB_{\mathbb{C}}(\underline{n+1})$: $\overline{\calB_{\mathbb{C}}(\underline{n+1},T_{2,n+1})}$ that should affect $\psi_0$ and $\overline{\calB_{\mathbb{C}}(\underline{n+1},T_{1,2})}$ that should affect $\psi_2$. 
In the case of the map $\pi_R$ we can think that we forget the point labelled by $n+1$, and then there are only two such natural divisors in $\calB_{\mathbb{C}}(\underline{n+1})$ as well: $\overline{\calB_{\mathbb{C}}(\underline{n+1},T_{1,n})}$ that should affect $\psi_0$ and $\overline{\calB_{\mathbb{C}}(\underline{n+1},T_{n,n+1})}$ that should affect $\psi_n$. 

\begin{proposition} For the map $\pi_L\colon \calB_{\mathbb{C}}(\underline{n+1})\to \calB_{\mathbb{C}}([2,n+1])$ we have: 
\begin{align*}
\pi_L^*\psi_0 &= \psi_0-[\overline{\calB_{\mathbb{C}}(\underline{n+1},T_{2,n+1})}] \\ 
\pi_L^*\psi_2 & = \psi_2-[\overline{\calB_{\mathbb{C}}(\underline{n+1},T_{1,2})}] \\
\pi_L^*\psi_i & = \psi_i \qquad \text{for\ } i=3,\dots,n+1
\end{align*}
Analogously, for the map $\pi_R\colon \calB_{\mathbb{C}}(\underline{n+1})\to \calB_{\mathbb{C}}(\underline{n})$ we have: 
\begin{align} \label{eq:pullbackpsi0} 
\pi_R^*\psi_0 &=\psi_0-[\overline{\calB_{\mathbb{C}}(\underline{n+1},T_{1,n})}] \\ 
\label{eq:pullbackpsin}
\pi_R^*\psi_n & = \psi_n-[\overline{\calB_{\mathbb{C}}(\underline{n+1},T_{n,n+1})}] \\
\label{eq:pullbackpsiarbitrary} 
\pi_R^*\psi_i & = \psi_i \qquad \text{for\ } i=1,\dots,n-1
\end{align}
\end{proposition}
\begin{proof} The proof is a simple comparison for the expressions of $\psi$-classes in terms of divisors using Proposition~\ref{prop:TRR}. We do it for the map $\pi_R$, and the argument for $\pi_L$ is completely analogous. 
	
Observe that
\begin{align} \label{eq:pullbackdivuniek}
 \pi_R^{-1}(\overline{\calB_{\mathbb{C}}(\underline{n},T_{l,r})}) & = \overline{\calB_{\mathbb{C}}(\underline{n+1},T_{l,r})} & \text{for\ } 1\leq l<r\leq n-1; \\  \notag
 \pi_R^{-1}(\overline{\calB_{\mathbb{C}}(\underline{n},T_{l,n})}) & = \overline{\calB_{\mathbb{C}}(\underline{n+1},T_{l,n})}
 \cup 
 \overline{\calB_{\mathbb{C}}(\underline{n+1},T_{l,n+1})}
  & \text{for\ } 2\leq l.
\end{align}
Then, using $i=1$ in the first statement of Proposition~\ref{prop:TRR} we see that
\begin{align*}
\psi_0|_{\calB_{\mathbb{C}}(\underline{n})} & = \sum_{r=2}^{n-1} 
[\overline{\calB_{\mathbb{C}}(\underline{n},T_{1,r})}] \\
\psi_0|_{\calB_{\mathbb{C}}(\underline{n+1})} & = \sum_{r=2}^{n} 
[\overline{\calB_{\mathbb{C}}(\underline{n+1},T_{1,r})}]
\end{align*}
Thus, using Equation~\eqref{eq:pullbackdivuniek} we see that 
\[
\pi_R^*(\psi_0|_{\calB_{\mathbb{C}}(\underline{n})})
= \sum_{r=2}^{n-1} 
[\overline{\calB_{\mathbb{C}}(\underline{n+1},T_{1,r})}],
\]
which is exactly $[\overline{\calB_{\mathbb{C}}(\underline{n+1},T_{1,n})}]$ minder than $\psi_0|_{\calB_{\mathbb{C}}(\underline{n+1})}$. This proves~\eqref{eq:pullbackpsi0}. 

Since $\psi_n|_{\calB_{\mathbb{C}}(\underline{n})}=0$, we have $\pi_R^*\psi_n=0$. On the other hand, the third formula in Proposition~\ref{prop:TRR} applied for $j=n$ gives $\psi_n|_{\calB_{\mathbb{C}}(\underline{n+1})}=[\overline{\calB_{\mathbb{C}}(\underline{n+1},T_{n,n+1})}]$. These two observations imply Equation~\eqref{eq:pullbackpsin}.

For $\psi_i$, $i=1,\dots,n-1$, we use the second formula in Proposition~\ref{prop:TRR}: 
\begin{align*}
\psi_i|_{\calB_{\mathbb{C}}(\underline{n})} & = \sum_{l=1}^{i-1} 
[\overline{\calB_{\mathbb{C}}(\underline{n},T_{l,i})}]; \\
\psi_i|_{\calB_{\mathbb{C}}(\underline{n+1})} & = \sum_{l=1}^{i-1} 
[\overline{\calB_{\mathbb{C}}(\underline{n+1},T_{l,i})}].
\end{align*}
Equation~\eqref{eq:pullbackdivuniek} implies that the pull-back of the right hand side of the first of these expressions is equal to the right hand side of the second expression. This proves Equation~\eqref{eq:pullbackpsiarbitrary}
\end{proof}

\subsubsection{Excess intersection formula}

We can also describe non-transversal self-intersections in the way similar to the usual moduli spaces. 

\begin{proposition}[Excess intersection formula]
We have
	\[
[\overline{\calB_{\mathbb{C}}(\underline{n},T_{l,r})}]^2 = [\overline{\calB_{\mathbb{C}}(\underline{n},T_{l,r})}] \cdot (\circ)_* (-\psi'-\psi''), 
	\]
where $\circ=\circ^{\{l,l+1,\dots,r\}}_{\{1,\dots,l-1,\star,r+1,\dots,n\},\star}$, and $\psi'$ and $\psi''$ denote, respectively, the $\psi_\star$-class on $\calB_{\mathbb{C}}(\{1,\dots,l-1,\star,r+1,\dots,n\})$ and the $\psi_0$-class on $\calB_{\mathbb{C}}(\{l,l+1,\dots,r\})$. In other words, we can say that
	\[
(\circ)^* [\overline{\calB_{\mathbb{C}}(\underline{n},T_{l,r})}] = -\psi'-\psi''.
	\]
\end{proposition}

\begin{proof} 
	The class of the self-intersection of a divisor is the class of its intersection with a transversal perturbation in the normal bundle. In other words, it is the first Chern class of its normal bundle.  The normal bundle of $\overline{\calB_{\mathbb{C}}(\underline{n},T_{l,r})}$ is described in Section~\ref{sec:local-structure}, the first Chern classes of the involved line bundles are defined in Definition~\ref{def:psi-classes}, hence the formula. 
\end{proof}

\subsubsection{Noncommutative correlators}

In this section, we shall compute the correlators
\[
\langle \tau_{d_0}\cdots\tau_{d_n}\rangle := \int_{\calB_{\mathbb{C}}(\underline{n})}\psi_0^{d_0}\cdots\psi_n^{d_n}.  
\]

Proposition~\ref{prop:TRR} immediately translates into following relations between these numbers.

\begin{proposition}\label{prop:abstract-correlators}
The correlator numbers satisfy the following relations:
\begin{itemize}
\item the relations at the points $1\le i-1, i,i+1\le n$
\begin{gather}
\label{TRR-3pts-1}
\langle\tau_{d_0}\tau_{d_1}\tau_{d_2}\cdots\tau_{d_{i-1}}\tau_{d_i+1}\tau_{d_{i+1}}\cdots\tau_{d_n}\rangle=\sum_{1\le l\le i-1}
\langle\tau_{d_0}\cdots\tau_{d_{i-1}}\tau_{0}\tau_{d_{i+1}}\cdots\tau_{d_n}\rangle\langle\tau_0\tau_{d_l}\cdots\tau_{d_i}\rangle ,\\
\label{TRR-3pts-2}
\langle\tau_{d_0}\tau_{d_1}\tau_{d_2}\cdots\tau_{d_{i-1}}\tau_{d_i+1}\tau_{d_{i+1}}\cdots\tau_{d_n}\rangle=\sum_{i+1\le l\le n}
\langle\tau_{d_0}\cdots\tau_{d_{i-1}}\tau_{0}\tau_{d_{l+1}}\cdots\tau_{d_n}\rangle\langle\tau_0\tau_{d_i}\cdots\tau_{d_l}\rangle ,
\end{gather}
\item the relations at the root and the points $1\le i,i+1\le n$
\begin{equation}\label{TRR-root}
\langle\tau_{d_0+1}\tau_{d_1}\cdots\tau_{d_n}\rangle=
\sum_{\substack{1\le l\le i\le i+1\le m\le n\\ m-l<n-1}}
\langle\tau_{d_0}\tau_{d_1}\cdots\tau_{d_{l-1}}\tau_0\tau_{d_{m+1}}\cdots\tau_{d_n}\rangle\langle\tau_0\tau_{d_l}\cdots\tau_{d_m}\rangle .
\end{equation}
\end{itemize}
\end{proposition}

\begin{remark}
The relations between the correlators from Proposition~\ref{prop:abstract-correlators} can also be used to obtain a different presentation of the operad $\ncHyperCom$ via ``abstract correlators''. Namely, it can be shown that the operad $\ncHyperCom$ is isomorphic to the ns operad $\calQ$ whose space of generators of arity $n\ge 2$ has, for each $n$, a basis  $\alpha^n_{d_0;d_1,\ldots,d_n}$ of degree $2n-4-2(d_0+d_1+\cdots+d_n)$, where $d_0,d_1,\ldots,d_n$ are nonnegative integers, subject to the following system of relations:
	\begin{itemize}
		\item (dimensionality)  for $d_0+d_1+\cdots+d_n>n-2$, we have 
		$\alpha ^n_{d_0;d_1,d_2,\ldots,d_n}=0$,
		\item (extreme input relations) for each $n,d_0,d_1,\ldots,d_n$, we have
		 \[
		\alpha ^n_{d_0;d_1+1,d_2,\ldots,d_n}=0 \quad\text{ and }\quad
		\alpha^n_{d_0;d_1,\ldots,d_{n-1},d_n+1}=0,
		\]
		\item (non-extreme input relations) for each $i=2,\ldots,n-1$, we have
		\begin{gather*}
		\alpha^n_{d_0;d_1,\ldots,d_{i-1},d_i+1,d_{i+1},\ldots,d_n}=
		\sum_{1\le k\le i-1}\alpha^{n-i+k}_{d_0;d_1,\ldots,d_{k-1},0,d_{i+1},\ldots,d_n} \circ_{k}\alpha^{i-k+1}_{0; d_k,\ldots, d_i},\\
		\alpha^n_{d_0;d_1,\ldots,d_{i-1},d_i+1,d_{i+1},\ldots,d_n}=
		\sum_{i+1\le k\le n}\alpha^{n-i+k}_{d_0;d_1,\ldots,d_{i-1},0,d_{k+1},\ldots,d_n} \circ_{i}\alpha^{i-k+1}_{0; d_i,\ldots, d_k},
		\end{gather*}
		\item (output relations) for each $i=1,\ldots,n-1$, we have
		\begin{gather*}
		\alpha^n_{d_0+1;d_1,\ldots,d_n}=
		\sum_{\substack{1\le k\le i, i+1\le l\le n\\ l-k<n-1}} \alpha^{n-l+k}_{d_0;d_1,\ldots,d_{k-1},0,d_{l+1},\ldots,d_n} \circ_{k}\alpha^{l-k+1}_{0; d_k,\ldots, d_l} .
		\end{gather*}
	\end{itemize}
Geometrically, 	the class $\alpha^n_{d_0;d_1,\ldots,d_n}$ is the Poincar\'e dual of the class $\psi_0^{d_0}\cdots\psi_n^{d_n}$ on $\calB_{\mathbb{C}}(n)$. 
\end{remark}

Relations of Proposition~\ref{prop:abstract-correlators} also allow to compute all correlators explicitly.  

\begin{proposition}\label{prop:nc-correlators}
	We have the following expression for the generating function of correlators: 
	\[
	\sum_{d_0+\cdots+d_n=n-2}\langle \tau_{d_0}\cdots\tau_{d_n}\rangle t_0^{d_0}\cdots t_n^{d_n}=(t_0+t_2)(t_0+t_3)\cdots(t_0+t_{n-1}).
	\]
	In other words, each correlator is equal to zero or one, and nonzero correlators are those for which $d_1=d_n=0$, and $d_i\le 1$, for $2\le i\le n-1$. 
\end{proposition}

\begin{proof}
Let us prove this result by induction on $n$. First, for $d_1>0$ or $d_n>0$, the correlator is equal to zero because the corresponding $\psi$-class is equal to zero. Second, if we assume that $d_i\ge 2$, we may use Equation \eqref{TRR-3pts-1} and rewrite the correlator as a sum of products of correlators 
 \[
\langle\tau_{d_0}\cdots\tau_{d_{i-1}}\tau_{0}\tau_{d_{i+1}}\cdots\tau_{d_n}\rangle\langle\tau_0\tau_{d_l}\cdots\tau_{d_i-1}\rangle \ ,
 \]
which vanishes since $d_i-1>0$. Next, if we assume that $d_1=\ldots=d_{i-1}=0$, and $d_i=1$, then using the same equation, we note that, for degree reason, the correlator is equal to the product of a similar correlator for $n-1$, and the correlator $\langle \tau_0^3\rangle=1$.  Finally, if we assume that $d_1=d_2=\cdots=d_n=0$, and $d_0=n-2$, we may apply Equation \eqref{TRR-root} for $i=1$. For homological degree reason, the non-root corolla must be binary, and so the correlator is equal to the product of a similar correlator for the arity one less with the correlator $\langle \tau_0^3\rangle=1$. The statement follows by induction. 
\end{proof}

\subsection{Givental group action via intersection theory}

\subsubsection{Noncommutative CohFTs}

In the classical situation,  hypercommutative algebras can be vi\-ew\-ed as cohomological field theories, \emph{\`a la}  Kontsevich--Manin \cite{KontsevichManin94}. We use that as a motivation for the following definition.

\begin{definition}[tree level ncCohFT]
Let  $A$ be a  graded vector space. An $A$-valued \emph{tree level noncommutative cohomological field theory} (ncCohFT) is defined as a system of classes \[\alpha_n\in H^\bullet(\calB_{\mathbb{C}}(\underline{n}))\otimes\End_A(n)\] of total degree $0$ satisfying the following \emph{factorisation property}: the pullbacks via the mappings $\circ_i \colon \calB_{\mathbb{C}}(\underline{k})\times\calB_{\mathbb{C}}(\underline{l}) \to \calB_{\mathbb{C}}(\underline{k+l-1})$ produce the composition of the multilinear maps at the slot $i$:
 \[
(\circ_i)^*\alpha_{k+l-1}=\alpha_{k}\tilde{\circ}_i \alpha_{l}\ , 
 \]
where $\tilde{\circ}_i$ combines the composition in the endomorphism operad and the K\"unneth isomorphism. 
\end{definition}

\begin{proposition}
The datum of a $\ncHyperCom$-algebra on a graded vector space $A$ is equivalent to the datum of an $A$-valued tree level ncCohFT. In particular, the datum of an associative algebra on a graded vector space $A$ is equivalent to the datum of an $A$-valued tree level ncCohFT with all the classes $\alpha_n$ concentrated in  $H^0(\calB_{\mathbb{C}}(n))$.
\end{proposition}

\begin{proof}
This is proved by a direct inspection. 
\end{proof}

\begin{remark}
In the classical case, tree level CohFTs are algebras over the genus $0$ part of a larger structure, the modular operad $\{\overline{\calM}_{g,n}\}$. The operad $\ncHyperCom$ is not cyclic, so one should not expect that to happen in this situation: as we argued in the Introduction, ``in the noncommutative world, only genus~$0$ is visible''.
\end{remark}

\subsubsection{Formula for the Givental action} 

\begin{definition}
The Givental--Lee Lie algebra action on $\ncHyperCom$-algebras is the action of the Lie algebra $\End(A)[[z]]$, where $z$ is a formal parameter of degree~$2$, on the corresponding $A$-valued tree level ncCohFTs, given by the formula
\begin{multline}\label{Givental-action-genus-0}
(\widehat{r_k z^{k}}.\{\alpha_n\})_n=
(-1)^{k-1}r_k\circ_1\alpha_{n}\cdot\psi_0^k+\sum_{m=1}^n\alpha_{n}\cdot\psi_m^k\circ_m r_k+\\+
\sum_{\substack{1\le p<q\le n, q-p<n-1,\\ i+j=k-1}}(-1)^{i+1}\,(\alpha_{n-q+p}\cdot\psi_p^j) \tilde{\circ}_p  (r_k \circ_1\alpha_{q-p+1}\cdot \psi_0^i).
\end{multline}
The operation $\tilde{\circ}_p$ in the last sum combines the composition in the endomorphism operad and the K\"unneth isomorphism.
\end{definition}

Note that this formula makes sense for all $k\ge 0$. 

\begin{proposition}
The formulas \eqref{Givental-action-genus-0} give a well defined 
Lie algebra action on $A$-valued ncCohFTs: they provide a Lie algebra homomorphism from the Lie algebra $\End(A)[[z]]$ from the Lie algebra of infinitesimal automorphisms of $\ncHyperCom$-algebra structures.
\end{proposition}

\begin{proof} We can repeat the same argument as in~\cite{FaberShadrinZvonkine} on the level of correlators or in~\cite{Kazarian,PandharipandePixtonZvonkine} on the level of classes, since this argument uses nothing but the excess intersection formula that we have as well. 
\end{proof}

Let us denote by $\tau^{(k)}_{n}\in\End_A(n)$ the value of the element $(-1)^{k-1}(\widehat{rz^k}.\alpha)_n$ on the fundamental cycle of $\calB_{\mathbb{C}}(\underline{n})$, and by $\nu_n$ the value of the element $\alpha_n$ on the fundamental cycle of $\calB_{\mathbb{C}}(\underline{n})$. We shall prove the following recurrence relation.
 
\begin{proposition}\label{prop:recurrence-Givental}
The infinitesimal deformations $\tau^{(k+1)}_n$ of the fundamental classes by the Givental action satisfy, for each $k\ge 0$, the recurrence relation
\begin{equation}\label{recursion-Givental}
\tau^{(k+1)}_{n}=\sum_{I\subsetneq\underline{n}}\sum_{j=1}^{n-|I|+1}\left(\frac{|I|-1}{n-1}\tau^{(k)}_{n-|I|+1}\circ_j\nu_{|I|}-\frac{n-|I|}{n-1}\nu_{n-|I|+1}\circ_j\tau^{(k)}_{|I|}\right)\ ,
\end{equation}
where the sum runs over all intervals $I$ of $\underline{n}$ of cardinality at least two.
\end{proposition}

\begin{proof}
The proof is analogous to \cite[Lemma 2]{DSV2013b}. In a nutshell, the proof amounts to combining topological recursion relations from Proposition~\ref{prop:TRR} in a somewhat imaginative way.
\end{proof}

\subsection{Givental group action via gauge symmetries}

This section follows the general plan of  \cite{DSV2013b} in order to prove
that the Givental action is equal to the action of gauge symmetries in a certain homotopy Lie algebra. Let us assume that $A$ is a chain complex with zero differential. 

The deformation retract 
 \[
\xymatrix{ *{\quad \qquad\qquad \qquad \big(\overline{\qncBV^{\ac}}, \delta^{-1}\otimes d_\psi\big)\ } \ar@(dl,ul)[]^{\delta\otimes H}\ \ar@<0.5ex>[r]^-{\mathrm{P}} & *{\ \big(\overline{T}^c(\delta)\oplus(\overline{\calS^2\ncGrav})^*, 0\big)  }  \ar@<0.5ex>[l]^-{\mathrm{I}}} 
 \]	
from Theorem~\ref{th:BVDefRetract} can be extended, by considering the whole ns cooperad $\qncBV^{\ac}$ instead of its coaugmentation coideal $\overline{\qncBV^{\ac}}$, to a deformation retract
 \[
\xymatrix{ *{\quad \qquad\qquad \qquad \big(\qncBV^{\ac}, \delta^{-1}\otimes d_\psi\big)\ } \ar@(dl,ul)[]^{\delta\otimes H}\ \ar@<0.5ex>[r]^-{\mathrm{P}} & *{\ \big(\calI\oplus\overline{T}^c(\delta)\oplus(\overline{\calS^2\ncGrav})^*, 0\big) \ . }  \ar@<0.5ex>[l]^-{\mathrm{I}}} 
 \]	

It is straightforward to check that this extension changes the transferred homotopy cooperad structure only marginally. More precisely, there are extra cooperad decomposition maps (decomposition into two pieces) that decompose every element $\phi$ of $\calI\oplus\overline{T}^c(\delta)\oplus(\overline{\calS^2\ncGrav})^*(n)$ as 
 \[
\sum_{m=1}^n \phi \circ_m \id +\id \circ_1 \phi\ , 
 \]
where $\id$ denotes the only basis element of $\calI$,  while all the strictly higher structure maps of the homotopy cooperad structure remain the same, and vanish on $\calI$. 

This homotopy cooperad leads to the convolution $L_\infty$-algebra 
\[
\mathfrak{l}_{\ncBV}=\Hom\big(\calI\oplus\overline{T}^c(\delta)\oplus(\overline{\calS^2\ncGrav})^*(n), \End_A\big) \cong \End(A)[[z]] \oplus \g_{\ncHyperCom} \ .
\]
Here $z$ is a formal parameter of degree $-2$, 
$$\g_{\ncHyperCom} := \Hom\big((\overline{\calS^2\ncGrav})^*(n), \End_A\big)$$
is the convolution Lie algebra  controlling the deformations of $\ncHyperCom_\infty$-algebra structures on $A$, and the direct sum decomposition on the right-hand side is a nontrivial extension of $L_\infty$-algebras. The deformation retract above leads to a deformation retract
\begin{equation}
\label{eq:LInftyDefRetract}
\xymatrix{ *{\quad \qquad\qquad \qquad\quad \big(\mathfrak{g}_{\ncBV}, (\delta^{-1}\otimes d_\psi)^*\big)\ } \ar@(dl,ul)[]^{h}\ \ar@<0.5ex>[r]^-{p} & *{\ (\mathfrak{l}_{\ncBV}, 0) \ , }  \ar@<0.5ex>[l]^-{i}}  
\end{equation}
where $\mathfrak{g}_{\ncBV}$ is the convolution dg Lie algebra 
$\Hom(\qncBV^{\ac},\End_A)$, $h=(\delta\otimes H)^*$, $i=\mathrm{I}^*$, $p=\mathrm{P}_*$. The structure maps of $\mathfrak{l}_{\ncBV}$ are obtained by homotopy transfer formulas for $L_\infty$-algebras. 

Any Maurer--Cartan element $\alpha$ in $\g_{\ncHyperCom}$, in particular a Maurer--Cartan element representing a ``strict'' $\ncHyperCom$-algebra, is also a Maurer-Cartan element in  $\mathfrak{l}_{\ncBV}$. We shall deform $\alpha$ in the ``transversal'' direction. Let $r(z)=\sum_{l\ge 0} r_l z^l$ be a degree $0$ element of~$\End(A)[[z]]$. The general definition of gauge symmetries in homotopy Lie algebras~\cite{Get09} implies that the image $\ell^\alpha_1(r(z))$ of $r(z)$ under the twisted map
 \[
\ell^\alpha_1(x):=\sum_{p\ge 0} \frac{1}{p!} \ell_{p+1}(\underbrace{\alpha, \ldots, \alpha}_p, x)\ , 
 \]
 is an infinitesimal deformation of $\alpha$. 

\begin{theorem}\label{thm:Giv=LinftyAction}
For any $\ncHyperCom$-algebra structure on $A$ encoded by a Maurer--Cartan element $\alpha\in\g_{\ncHyperCom}$ and for any degree $0$ element $r(z)\in\End(A)[[z]] $, the Givental action of $r(z)$ on $\alpha$ is equal to the gauge symmetry action of $r(z)$ on $\alpha$ viewed as a homotopy $\ncBV$-algebra structure:
	\[\widehat{r(z)}.\alpha=\ell^\alpha_1(r(z))\ . \]
\end{theorem}

\begin{proof}
The proof is analogous to  \cite[Th.~5]{DSV2013b}. Basically, there are two main steps that one has to undertake. First, it is possible to prove a version of \cite[Lemma 1]{DSV2013b}, showing that the infinitesimal gauge symmetry action of $r(z)$ on $\alpha$ produces an element which also corresponds to a $\ncHyperCom$-algebra. Next, one should consider, for each $n\ge 2$ and $k\ge 0$, the element $\theta^{(k)}_n\in\End_A(n)$ which is the result of evaluating  $\ell_{k+2}(rz^k,\alpha,\ldots,\alpha)$ on the element representing the $n$-ary generator of $\ncHyperCom$ (those are the only nonzero contributions to the infinitesimal deformation). Similarly to \cite[Lemma 3]{DSV2013b}, there is a recurrence relation that expresses, for each $k\ge0$, the operation $\theta^{(k+1)}$ via the operations $\theta^{(i)}$ with $i\le k$ and the operations $\nu_j$ (the noncommutative hypercommutative operations that are being deformed):
 \[
\theta^{(k+1)}_{n}=\sum_{I\subset\underline{n}}\sum_{j=1}^{n-|I|+1}\left(\frac{|I|-1}{n-1}\theta^{(k)}_{n-|I|+1}\circ_j\nu_{|I|}-\frac{n-|I|}{n-1}\nu_{n-|I|+1}\circ_j\theta^{(k)}_{|I|}\right)\ ,
 \]
where the sum runs over all intervals $I$ of $\underline{n}$ of cardinality at least two. According to Proposition~\ref{prop:recurrence-Givental}, the operations $\lambda^{(k)}_n$ satisfy the same recurrence relation, and same initial condition $\lambda^{(0)}_n=\theta^{(0)}_n$, therefore $\lambda^{(k)}_n=\theta^{(k)}_n$ for all~$k$.
\end{proof}

\section{Algebra and geometry around the \texorpdfstring{$\ncBV$}{ncBV}-algebras and the Givental action}\label{sec:applications}

In this section, we discuss two applications of the Givental action, and some results directly related to these applications. We use the Givental action to re-interpret the notion of the noncommutative order of a differential operator on a noncommutative algebra due to B\"orjeson \cite{Bor13}, and to establish an explicit quasi-isomorphism between the ns operad $\ncHyperCom$ and an explicit model for the homotopy quotient of $\ncBV$ by $\Delta$. The first of those results leads to a new type of algebras, a particular type of homotopy $\ncBV$-algebras where only the operator $\Delta$ is relaxed up to homotopy; these algebras are related to homotopy involutive infinitesimal bialgebras via a certain bar construction. The second one leads to a corresponding conjecture on the geometric level, which we support by exhibiting a different geometric model for the operad $\ncBV$.

\subsection{B\"orjeson products and the Givental action}

The purpose of this section is to relate the B\"orjeson products, which are noncommutative analogues of Koszul brackets \cite{Koszul85} defining the order of differential operators, to the intersection theory of brick manifolds.

\subsubsection{Differential operators on noncommutative algebras}\label{sec:DifOp}

In this section, we recall a definition of a differential operator of noncommutative order at most $l$ on an associative algebra prompted by results of B\"orjeson~\cite{Bor13} and Markl~\cite{Markl2014}.

\begin{definition}[\cite{Bor13,Markl2014}]
Let $(A,m)$ be a graded associative algebra, and let $D\colon A\to A$ be a linear operator. We define the sequence of \emph{B\"orjeson products} $b_n^D\colon A^{\otimes n}\to A$ as follows: 
\begin{gather*}
b_1^D:=D,\\
b_2^D:=D\circ_1m-m\circ_1D-m\circ_2D,\\
b_3^D:=D\circ_1m\circ_2m-m\circ_1(D\circ_1m)-m\circ_2(D\circ_1m)+m\circ_2(m\circ_1D),\\
b_n^D:=b_{n-1}^D\circ_2m \quad (n\ge 4). 
\end{gather*} 
\end{definition}

It is easy to see that the above recursive definition results in the following explicit formula for the higher B\"orjeson products:
\begin{equation}\label{eq:BorjesonExpl}
b_n^D=D\circ_1m^{(n-1)}-m\circ_1 (D\circ_1m^{(n-2)}) - m\circ_2 (D\circ_1m^{(n-2)})+
(m \circ_1 m)\circ_2 (D\circ_1m^{(n-3)}).
\end{equation} 

This implies the following result which shows that the substitution in the slot $2$ in the formula $b_n^D:=b_{n-1}^D\circ_2m$ above is somewhat coincidental, and there are recursive formulas for the higher B\"orjeson products which use any chosen slot. 

\begin{proposition}\label{prop:Borjeson-recursion}
We have the following formulas:
\begin{gather}
b_n^D=b_{n-1}^D\circ_1 m -m \circ_2  b_{n-1}^D , \label{eq:recur-first}\\
b_n^D=b_{n-1}^D\circ_i m, \quad 2\le i\le n-2 , \label{eq:recur-mid} \\
b_n^D=b_{n-1}^D\circ_{n-1} m -m \circ_1  b_{n-1}^D , \label{eq:recur-last}
\end{gather}
\end{proposition}

\begin{proof}
Direct computation using Formula \eqref{eq:BorjesonExpl}. 
\end{proof}

\begin{definition}[Noncommutative order of an operator~\cite{Bor13,Markl2014}]
Let $D$ be a linear map on an associative algebra $A$. It is said to be a \emph{differential operator on~$A$ of noncommutative order at most~$l$} if the B\"orjeson product $b_{l+1}^D$ is identically equal to zero. 
\end{definition}

\begin{remark}
This definition should remind the reader of the definition of differential operators on commutative algebras due to Koszul~\cite{Koszul85}, which, for unital algebras, is equivalent to the classical definition of Grothendieck~\cite{Grothendieck67}. However, in contrast to that, the definition of noncommutative order of a differential operator has a very unusual feature: while for non-unital algebras one can easily construct examples of differential operators of all possible orders, e.g. coming from bar constructions of $A_\infty$-algebras \cite{Bor13}, for unital algebras, a differential operator of order at most $l$, for $l\ge 2$, is automatically a differential operator of order at most~$2$. To verify that, one can see that substituting $a_1=a$, $a_2=\cdots=a_{l}=1$, $a_{l+1}=b$ in the identity $b_{l+1}^D(a_1,\ldots,a_{l+1})=0$, one gets the identity 
 \[
D(ab)=D(a)b+(-1)^{|a||D|}aD(b)-(-1)^{|a||D|}aD(1)b ,   
 \]
which, by direct inspection, implies $b_3^D=0$. The latter identity, in turn, implies the identity \[b_{l+1}^D(a_1,\ldots,a_{l+1})=0 \quad  \text{ for all } \quad l\ge 2,\] by definition. Therefore, the notion of noncommutative order of differential operators is mainly ``interesting'' for non-unital algebras.
\end{remark}

We denote by $\Diff_{\le l}(A)$ the set of all differential operators of noncommutative order at most~$l$, and by $\Diff(A)$ the set of differential operators of all possible noncommutative orders: 
 \[
\Diff(A):=\bigcup_{l\ge0}\Diff_{\le l}(A).
 \]

\begin{remark}\label{rk:nonassoc}
In the commutative case, the set of all differential operators of all possible orders forms an algebra with respect to composition, and moreover this algebra is filtered by the order of the operators. However, in the noncommutative case, the composite of two differential operators of finite noncommutative orders is not necessarily a differential operator of a finite noncommutative order:
 \[
\Diff_{\le k}(A)\circ\Diff_{\le l}(A)\not\subset\Diff(A) , 
 \]
as one can verify by direct inspection already for $k=l=1$.  
\end{remark}

Similarly to the commutative case, values of a differential operator of order at most $l$ on any product can be expressed via values on products of at most $l$ factors. More precisely, the following result holds. 

\begin{proposition}\label{prop:identities-diff}
For a differential operator~$D$ of noncommutative order at most~$l$, and for every $n\ge l+1$ we have
 \[
D\circ m^{(n-1)}=\sum_{1\le k\le n-l+1} m^{(n-l)}\circ_k (D\circ_1 m^{(l)}) -\sum_{2\le k\le n-l+1} m^{(n-l+1)}\circ_k (D\circ_1 m^{(l-1)}). 
 \]
\end{proposition}

\begin{proof}
This is proved by an easy induction on $n$, similar to \cite[Prop.~5]{DSV2013a}, using recursive formulas of Proposition \ref{prop:Borjeson-recursion}.
\end{proof}

An important property of B\"orjeson products is given by the following proposition. 
\begin{proposition}\label{prop:LieAlgHomomorphism}
We have 
\begin{equation}\label{commutator-Koszul}
b_n^{[D_1,D_2]}= 
\sum_{\substack{n=i+j+k,\\k\ge 1, i,j\ge 0}}
b_{i+j+1}^{D_1} \left(\id^{\otimes i}\otimes b_k^{D_2}\otimes \id^{\otimes j}\right)-(-1)^{|D_1||D_2|}
b_{i+j+1}^{D_2} \left(\id^{\otimes i}\otimes b_k^{D_1}\otimes \id^{\otimes j}\right).
\end{equation}
Here $[D_1,D_2]=D_1D_2-(-1)^{|D_1||D_2|}D_2D_1$ is the graded commutator of the linear maps $D_1$ and $D_2$. 
\end{proposition}

\begin{proof}
This essentially follows from results of~\cite{Markl2014} interpreted from the viewpoint of~\cite{DSV2015}. Let us consider the pre-Lie algebra $\mathfrak{g}:=\big(\prod_{n\ge 1} \mathrm{Hom}(A^{\otimes n}, A), \star
\big)$ associated to the ns operad $\End_{A}$, see \cite[Section~$5.9.15$]{LV}.  We view the associative algebra structure on $A$ as an element $\Phi=(\id, m , m^{(2)}, \ldots)$ of $\mathfrak{g}$ which, in each arity $k\ge 1$, is the operation $m^{(k)}$ of the $k$-fold product. According to \cite[Example~2.4]{Markl2014}, the element $\Phi^{-1}$ is the operation which, in each arity $k\ge 1$, is $(-1)^{k-1}$ times the $k$-fold product, and the
 adjoint action of $\Phi$ on an element $D\in\End(A)=\End_A(1)\subset\mathfrak{g}$ is equal to $(b_1^D, b_2^D, b_3^D, \ldots)=b^D$, that is the element which, in each arity $k\ge 1$, is the $k$-th B\"orjeson product $b_k^D$. (In \cite{Markl2014}, this is applied for a map $D$ of degree $1$ satisfying $D^2=0$, but this is only important if we aim to obtain an $A_\infty$-algebra; otherwise these formulas can be applied to any $D$ whatsoever). The adjoint action of a Lie group is in always an automorphism of the corresponding Lie algebra, which is precisely what Equation~\eqref{commutator-Koszul} translates into in our case.  
\end{proof}

Proposition~\ref{prop:LieAlgHomomorphism} implies the claim of B\"orjeson that for a map $D$ of degree $-1$ satisfying $D^2=0$ the B\"orjeson products define an 
 $A_\infty$-algebra structure on the desuspension $s^{-1}A$ of $A$, since $D^2=\frac12[D,D]$. More conceptually with the framework developed in \cite{DSV2015}, $D$ can be viewed as a Maurer--Cartan element in the convolution dg pre-Lie algebra 
 \[
\left(\prod_{n\ge 1} \Hom(\As^{\ac}, \End_A), \star\right) = \left(\prod_{n\ge 1} \mathrm{Hom}(A^{\otimes n}, A), \star\right) .
 \]
The element $\Phi=(\id, m , m^{(2)}, \ldots)$ is an element of the associated gauge group, whose action on $D$ gives yet another Maurer--Cartan element, which is nothing but the $A_\infty$-algebra of B\"orjeson.
 
 Another result that one can immediately derive is the following property that is known to hold for differential operators on commutative algebras (and is somewhat surprising in the view of Remark~\ref{rk:nonassoc}). 

\begin{corollary}\label{cor:LieAdm}
The graded commutator of two differential operators of finite noncommutative order is also a  differential operator of finite noncommutative order. Moreover, the commutator ``makes the noncommutative order drop by~$1$'':
\begin{equation}\label{filtered-Lie}
[\Diff_{\le k}(A),\Diff_{\le l}(A)]\subset\Diff_{\le k+l-1}(A). 
\end{equation}	
\end{corollary}

In the commutative case, a differential operator of order at most $l$ on the algebra of formal power series in several variables is completely determined by its restriction to the subspace of polynomials of degree at most $l$. Let us establish, for a given $l\ge 1$, an explicit description of the set of all differential operators of noncommutative order at most $l$ on the non-unital algebra $\overline{T}(V)$ of noncommutative polynomials. 
For any map $f \colon V^{\otimes l} \to \overline{T}(V)$, we consider the map $D_f\colon  \overline{T}(V) \to  \overline{T}(V)$ defined by
 \[
D_f := 
\left\lbrace
\begin{array}{cl}
\displaystyle \sum_{i+j+l=n}\id_V^{\otimes i}\otimes f\otimes\id_V^{\otimes j}\ , & \text{ for } n\ge l\ , \\
0\ , &\text{ for } n<l\ .
\end{array}
\right.
 \]
The assignment $f\mapsto D_f$ defines a linear map from $\Hom(V^{\otimes l}, \overline{T}(V))$ to $\End\big( \overline{T}(V)\big)$. 
We denote its image by $\Diff^{\,(l)}\big(\overline{T}(V)\big)$.

\begin{proposition}\label{prop:NCDiffOpTensor}\leavevmode
\begin{enumerate}
\item For any map $f : V^{\otimes l} \to \overline{T}(V)$, the map $D_f\in\, \Diff_{\leq l}\big(\overline{T}(V)\big)$ is differential operator of noncommutative order at most~$l$. 
\item The vector space of differential operators of noncommutative operators of order at most~$l$ decomposes with respect to these maps: 
$$ \Diff_{\leq l}\big(\overline{T}(V)\big) \cong \bigoplus_{k=1}^l \Diff^{\,(k)}\big(\overline{T}(V)\big)\ .$$
\end{enumerate}
\end{proposition}

\begin{proof}
The first assertion is directly proved by computing $b_{l+1}^{D_f}=0$.

For the second assertion, let $D_{f_1}+\cdots + D_{f_l}=0$ in 
$\sum_{k=1}^l \Diff^{\,(k)}\big(\overline{T}(V)\big)$. By induction, its restriction to $V^{\otimes k}$, for $k=1,\ldots, l$ is equal to $f_k$. So all these latter ones vanish, proving that all the $D_{f_k}$ are equal to $0$. 
Now let $D\in \Diff_{\leq l}(\overline{T}(V))$ be a differential operator of order at most $l$. We consider its restrictions
$f_k:=D_{|V^{\otimes k}}$  on $V^{\otimes k}$, for $1\leq k \leq l$. We define by induction on $k=1, \ldots, l$, the following maps $f_{(k)}$ and differential operators $D^{(k)}$ of noncommutative order at most $k$:
 \[
f_{(k)}:=f_k-D^{(1)}_{|V^{\otimes k}}-\cdots-D^{(k-1)}_{|V^{\otimes k}}\, \quad D^{(k)}:=D_{f_{(k)}}\in \Diff^{\,(k)}\big(\overline{T}(V)\big) .
 \] 
By definition, $D^{(1)}+\cdots+D^{(l)}$ agrees with $D$ on $V\oplus \cdots \oplus V^{\otimes l}$. Finally, Proposition~\ref{prop:identities-diff} shows that $D^{(1)}+\cdots+D^{(l)}$ and $D$ agree on $V^{\otimes n}$ for all $n\ge l+1$, hence $D=D^{(1)}+\cdots+D^{(l)}$.
\end{proof}

\subsubsection{Associative $\ncBV_\infty$-algebras}

Recall that in~\cite{Kra00}, a version of homotopy $\BV_\infty$-algebras was proposed. A conceptual interpretation of that definition is contained in~\cite{GTV09}, where it is shown to be equivalent to an algebra over the Koszul resolution of the operad $\BV$, for which most higher operations vanish. In this section, we present a noncommutative counterpart of the latter definition, and show that a definition \emph{\`a la} \cite{Kra00} which utilises B\"orjeson products instead of Koszul brackets is also available. 

\begin{definition}[Homotopy $\ncBV$-algebra]
A \emph{homotopy $\ncBV$-algebra} is an algebra over the Koszul resolution $\ncBV_\infty:=\Omega \ncBV^{\ac}$.
\end{definition}

Since the underlying nonsymmetric collection of the Koszul dual is isomorphic to
\[\ncBV^{\ac}\cong T^c(\delta)\circ \As_1^{\ac} \circ \As^{\ac},\] the data of a $\ncBV_\infty$-algebra on a chain complex $A$ amounts to a set of operations $m_{i_1, \ldots, i_k}^l$, where $l\ge 0$ is the weight in $\delta$, $k\ge 1$ is the arity in $\As_1^{\ac}$, and $i_1,\ldots,i_k \ge 1$ are the arities in $\As^{\ac}$ (where $\mu_1^0=d_A$).

The following result is a noncommutative counterpart of \cite[Th.~20]{GTV09}.

\begin{proposition}
A homotopy $\ncBV$-algebra is a chain complex $(A, d_A)$ equipped  with operations
 \[
m_{i_1, \ldots, i_k}^l\ :\ A^{\otimes n} \to A \ , \quad \text{for}\quad n=i_1+\cdots+i_k, \ l\ge 0,\  k\ge 1, \ i_j\ge 1\ ,
 \] 
of degree $n+k+2l-3$, with  $m^0_1=d_A$, satisfying the following relations (indexed by the same set) 
 \[
\sum_{\substack{1\leq j\leq k\\ i_j={i_j}'+{i_j}''}}
 \varepsilon_1\, m^{l-1}_{i_1, \ldots, {i_j}', {i_j}'', \ldots,  i_k} +
 \sum_{\substack{
l=l'+l''\\
 a\leq b }}
  \varepsilon_2\,
m^{l'}_{i_1, \ldots, {i_\alpha}'+1+{i_\beta}'', \ldots,  i_k}
 \circ_{i_1+\cdots+{i_\alpha}'+1}
 m^{l''}_{{i_\alpha}'', i_{\alpha+1}, \ldots, i_{\beta-1}, {i_\beta}'} =0\ ,
 \]
where ${i_\alpha}'$, ${i_\alpha}''$, ${i_\beta}'$, and ${i_\beta}''$ are defined by $i_\alpha={i_\alpha}'+{i_\alpha}''$,
$i_\beta={i_\beta}'+{i_\beta}''$, $a=i_1+\cdots+{i_\alpha}'+1$, and $b=i_1+\cdots+{i_\beta}'$, and where
$ \varepsilon_1= (-1)^{i_1+\cdots+i_{j-1}+{i_j}'-j+1}$ and
\[ \varepsilon_2= (-1)^{i_1+\cdots+i_{\alpha-1}+i_{\beta}+\cdots+i_k
+[({i_\alpha}''-1)+\cdots+({i_\beta}'-1)].[({i_\beta}''-1)+\cdots+({i_k}-1)]
+k-1}\ .\]
We denote this relation by $\mathcal{R}_{i_1, \ldots, i_k}^l$.  Note that for $l=0$, the first (linear) term is not present in the relation.
\end{proposition}

\begin{proof}
A structure of a $\ncBV_\infty$-algebra on a chain complex $A$ is encoded by a Maurer--Cartan element in the convolution dg Lie algebra $\Hom(\ncBV^{\ac},\End_A)$.
Let us denote by $m_{i_1, \ldots, i_k}^l$ the image of the element $\mu_{i_1, \ldots, i_k}^l\in\ncBV^{\ac}$ in the endomorphism operad under a given $\ncBV_\infty$-algebra structure. The Maurer--Cartan equation, once evaluated on $\mu_{i_1, \ldots, i_k}^l$, leads to a certain relation in the given algebra. The first term of that relation corresponds to the image of $\mu_{i_1, \ldots, i_k}^l$ under the coderivation $d_\varphi$. The second term of that relation is given by the infinitesimal decomposition map of the ns cooperad $\ncBV^{\ac}$. These are precisely the terms in the relation $\mathcal{R}_{i_1, \ldots, i_k}^l$ above. 
\end{proof}

\begin{definition}[Associative $\ncBV_\infty$-algebra]\label{nc-homotopy-bv}
An $\ncBV_\infty$-algebra is called \emph{associative} if all its generating operations vanish except for the binary product $m^0_2$ and the operations of the form $m^l_{1, \ldots, 1}$.
\end{definition}

\begin{proposition}\label{prop:asBV-ala-Kra}
An associative $\ncBV_\infty$-algebra is an associative algebra $(A, m)$
 equipped with a collection of operators $\Delta_l$, $l\ge 0$, each $\Delta_l$ being a differential operator of noncommutative order at most~$l+1$ and of degree~$2l-1$, such that
 \[
\sum\limits_{i+j=l}\Delta_i\Delta_j=0\ ,
 \]
 for each $l\ge 0$.
\end{proposition}

\begin{proof}
The proof is similar to the proof of \cite[Prop.~$23$]{GTV09}. When all the operations vanish except for $m:=m^0_2$ and the $m^l_{1, \ldots, 1}$'s, the only non-trivial relations are $\mathcal{R}^0_2$, $\mathcal{R}^0_3$, $\mathcal{R}^l_{1, \ldots, 1}$, and $\mathcal{R}^l_{1, \ldots,1,2, 1, \ldots 1}$. The relations $\mathcal{R}^0_2$ and $\mathcal{R}^0_3$ together  express the fact that $(A, m^0_2, m^0_1)$ is a differential graded associative algebra.

We consider $\Delta_0:=d_A$ and $\Delta_l:=m^l_1$, for $l\ge 1$, which satisfy $|\Delta_l|=2l-1$. For $l\ge1$, the relations $\mathcal{R}^l_2$, $\mathcal{R}^l_{1,2}$ (or equivalently $\mathcal{R}^l_{2,1}$), and $\mathcal{R}^l_{1,2, 1,\ldots, 1}$ are equivalent to the formulas
 \[
m^{l-1}_{1, \ldots, 1}=b_n^{\Delta_{l+n-2}}, 
 \]  
which show that the operations $m^{l}_{1, \ldots, 1}$ are completely determined by the operators $\Delta_l$. Examining Formula~\eqref{eq:recur-mid}, we find that the relations $\mathcal{R}^l_{1, \ldots, 1, 2, 1,\ldots, 1}$ do not bring anything new.
The relations $\mathcal{R}^0_2$, $\mathcal{R}^0_{1,2}$ (or equivalently $\mathcal{R}^0_{2,1}$), and $\mathcal{R}^0_{1,2, 1,\ldots, 1}$ (or equivalently $\mathcal{R}^0_{1, \ldots, 1, 2, 1,\ldots, 1}$) give respectively that $b_n^{\Delta_{n-2}}=0$, for $n\ge 2$, that is $\Delta_l$ is a differential operator of noncommutative order at most $l+1$.
The relations $\mathcal{R}^l_1$ are precisely $\sum\limits_{i+j=l}\Delta_i\Delta_j=0$.
Using Relation~\eqref{commutator-Koszul}, it is easy to see that relations $\mathcal{R}^l_{1, \ldots, 1}$ formally follow from the relations $\mathcal{R}^l_{1,2, 1,\ldots, 1}$ and $\mathcal{R}^l_1$ (if we write the latter in the equivalent form $\sum\limits_{i+j=l}[\Delta_i,\Delta_j]=0$).

The relations $\mathcal{R}^l_{2, 1,\ldots, 1}$ and $\mathcal{R}^l_{1,\ldots, 1,2}$, where $l\ge1$ and the subscripts add to $n\ge 4$, give, respectively,
 \[
b_n^{\Delta_{n+l-2}}=b_{n-1}^{\Delta_{n+l-2}}\circ_1 m -m \circ_2  b_{n-1}^{\Delta_{n+l-2}} \quad \text{ and }\quad
b_n^{\Delta_{n+l-2}}=b_{n-1}^{\Delta_{n+l-2}}\circ_{n-1} m -m \circ_1  b_{n-1}^{\Delta_{n+l-2}} , 
 \]
which follow immediately from Formulas~\eqref{eq:recur-first} and \eqref{eq:recur-last}.
Finally, the relations $\mathcal{R}^0_{2, 1,\ldots, 1}$ and $\mathcal{R}^0_{1,\ldots, 1,2}$, where the subscripts add to $n\ge 4$, give, respectively,
 \[
b_{n-1}^{\Delta_{n-2}}\circ_1 m= m\circ_2  b_{n-1}^{\Delta_{n-2}} \quad \text{ and } \quad 
b_{n-1}^{\Delta_{n-2}}\circ_{n-1} m =m \circ_1  b_{n-1}^{\Delta_{n-2}},
 \]
which immediately follow from Relations \eqref{eq:recur-first} and \eqref{eq:recur-last} since the operator $\Delta_{n-2}$ has noncommutative order at most~$n-1$.
\end{proof}


Let us denote by $\ncBV_\infty^{\mathrm{\, assoc}}$ the dg ns operad encoding associative $\ncBV_\infty$-algebras. As we just established, dg ns operad is obtained from the ns operad of associative algebras by adjoining generators $\Delta_l$ of degree $2l-1$, for $l\ge 1$, satisfying the noncommutative order at most $l+1$ condition $b_{l+2}^{\Delta_l}=0$; the differential $\partial=[\Delta_0,\cdot]$ of this operad satisfies
\begin{gather*}
\partial(m)=0 ,\\
\partial(\Delta_l)=-\sum_{\substack{i+j=l, \\ i,j \ge 1}}  \Delta_i \Delta_j .
\end{gather*} 
Let us establish a noncommutative analogue of \cite[Th.~5.3.1]{CMV2014} and show that the dg ns operad $\ncBV_\infty^{\mathrm{\, assoc}}$ is a (non-cofibrant) resolution of the ns operad $\ncBV$. 

\begin{theorem}
The canonical projection 
 \[
\rho\colon \ncBV_\infty^{\mathrm{\, assoc}} \rightarrow \ncBV\ ,  
 \]
which sends the generator $\Delta_1$ to $\Delta$ and the generators $\Delta_l$ to $0$, for $l \ge 2$, is a quasi-isomorphism of dg ns operads. 
\end{theorem}

\begin{proof}
We follow here the strategy of the proof that is similar to the one for the commutative case outlined in \cite[Appendix~C]{CMV2014}; the below Gr\"obner basis argument is a noncommutative analogue of the argument in \cite[Appendix~A]{DSV2013a}. 

Let us determine a basis of the underlying ns collection of $\ncBV_\infty^{\mathrm{\, assoc}}$. We first note that since all the operators $\Delta_i$ are of odd degrees, we have
 \[
-\sum_{\substack{i+j=l, \\ i,j \ge 1}} \Delta_i \Delta_j = -\frac12\sum_{\substack{i+j=l, \\ i,j \ge 1}} [\Delta_i, \Delta_j]
 \]
and hence by Corollary \ref{cor:LieAdm} the element 
 \[
\partial(\Delta_l)=-\sum_{\substack{i+j=l, \\ i,j \ge 1}} \Delta_i \Delta_j 
 \]
is a differential operator of noncommutative order at most $i+1+j+1-1=l+1$. It follows that the ideal of relations of the operad $\ncBV_\infty^{\mathrm{\, assoc}}$ is a differential ideal, and hence there are no new relations arising as boundaries of existing ones. Thus, for the purpose of computing a basis, we may forget about the differential $\partial$, and consider the operad which is a quotient of the free operad $\calT(m,\Delta_1, \Delta_2, \ldots)$ by the ideal generated by the associativity relation for the operation $m$ and the noncommutative order $l+1$ relation for $\Delta_l$, for each $l\ge 1$. 

Next, let us compute the reduced Gr\"obner basis of relations for the path degree-lexicographic ordering \cite{BD} corresponding to the ordering of generators
 \[ 
m<\Delta_1<\Delta_2<\cdots
 \] 
Note that the leading term of the associativity relation is $m\circ_1m$, and the leading term of the noncommutative order $l+1$ relation is $\Delta_l\circ_1 m^{l+1}$. Thus, the only overlaps of these leading terms are the overlap of the associativity relation with itself and the $l+1$ overlaps of the noncommutative order $l+1$ relation with the associativity relation, corresponding to the common multiples $\Delta_l\circ_1 m^{l+1}\circ_im$ for $1\le i\le l+1$, for each $l\ge 1$. The first overlap leads to an S-polynomial that can be trivially reduced to zero. The common multiple $\Delta_l\circ_1 m^{l+1}\circ_i m$ lead to S-polynomials that can be reduced to zero using proposition \ref{prop:Borjeson-recursion}.  Hence, the defining relations of our operad form a reduced Gr\"obner basis. 

Finally, we observe that the set of normal monomials with respect to the leading terms of the Gr\"obner basis we found are in one-to-one correspondence with the set of all planar trees whose  internal edges are labelled by arbitrary noncommutative monomials in generators~$\Delta_l$, $l\ge 1$, where each vertex above an operator $\Delta_l$ can have at most $l+1$ leaves. These normal monomials form a basis of the underlying ns collection of $\ncBV_\infty^{\mathrm{\, assoc}}$.

\smallskip 

Let us denote by $\As_1^\Delta$ the ns sub-operad of $\ncBV$ generated by $b$ and $\Delta$. Its relations define a distributive law $\As_1^\Delta\cong \As_1\circ \D$ between $\As_1$ and $\D$, see Propostion~\ref{prop:qncBVDistLaw}. Therefore, it is Koszul with Koszul dual ns cooperad is isomorphic to 
$(\As_1^\Delta)^{\ac}\cong \D^{\ac}\circ \As_1^{\ac}\cong T^c(\delta)\otimes \As^*$, the Hadamard product of the cooperads $T^c(\delta)$ and $\As^*$. We denote by 
$(\As_1^\Delta)_\infty$ its Koszul  resolution. Let us denote by $\mu^l_n$ the generating element of $\k\, \delta^l\otimes \As^*(n)$ and by 
$\beta_n^{\Delta_l}$ the element
 \[
\Delta_l\circ_1m^{(n-1)}-m\circ_1 (\Delta_l\circ_1m^{(n-2)}) - m\circ_2 (\Delta_l\circ_1m^{(n-2)})+
(m \circ_1 m)\circ_2 (\Delta_l\circ_1m^{(n-3)})
 \] 
in $\ncBV_\infty^{\mathrm{\, assoc}}$, the right hand side of Formula~\eqref{eq:BorjesonExpl}. From Proposition~\ref{prop:LieAlgHomomorphism}, it easily follows that the morphism of ns operads 
 \[
\chi : (\As_1^\Delta)_\infty \to \ncBV_\infty^{\mathrm{\, assoc}}
 \] 
defined on the free generators by setting $\chi(\mu^l_n)=\beta_n^{\Delta_{l+n-1}}$ is a morphism of dg ns operads, in particular a morphism of dg ns collections. It follows that the composite 
 \[
\As\circ(\As_1^\Delta)_\infty\xrightarrow{\id\circ \chi} \As\circ  \ncBV_\infty^{\mathrm{\, assoc}} \hookrightarrow \ncBV_\infty^{\mathrm{\, assoc}}\circ  \ncBV_\infty^{\mathrm{\, assoc}} \to \ncBV_\infty^{\mathrm{\, assoc}} ,
 \]
where the last arrow is the composition map of the ns operad $\ncBV_\infty^{\mathrm{\, assoc}}$, is also a morphism of dg ns collections. A direct inspection shows that the images of the natural basis elements of $\As\circ(\As_1^\Delta)_\infty$ have pairwise distinct leading terms, and among those leading terms each normal monomial determined above occurs exactly once. It follows that our map is actually an isomorphism. This isomorphism induces a homology isomorphism 
 \[
H_\bullet(\ncBV_\infty^{\mathrm{\, assoc}}) \cong H_\bullet(\As\circ(\As_1^\Delta)_\infty)\cong \As \circ \As^\Delta_1 \cong \qncBV\cong \ncBV .
 \]
It remains to note that the diagram 
 \[\xymatrix@C=45pt@R=20pt{ 
\As\circ(\As_1^\Delta)_\infty  \ar[r]_{\gamma(\As\circ \chi)}^\cong \ar[d]^{\sim} & \ncBV_\infty^{\mathrm{\, assoc}} 
\ar[d]^\rho\\
\As \circ \As^\Delta_1 \ar[r]^\cong&  \ncBV}
 \]
commutes, and hence map $\rho$ is a quasi-isomorphism. 
\end{proof}

\subsubsection{Differential operators on the noncommutative algebras and the Givental action}
In this section, we combine results and definitions of the previous sections, obtaining noncommutative analogues of the genus $0$ results of \cite{DSV2013a}.

The first result we prove relates B\"orjeson products to intersection theory on brick manifolds. 
  
\begin{theorem}\label{th:BorjesonIntersection}
Let $A$ be a graded associative algebra, which we regard as $\ncHyperCom$-algebra where all the 
structure operations $\nu_k$ of higher arities $k>2$ vanish, and let  $\textstyle r(z)=\sum_{l\ge 0} r_lz^l$ be an element of the Lie algebra $\End(A)[[z]]$ of certain degree $d$. The Givental action \eqref{Givental-action-genus-0} of $r(z)$ on the corresponding tree level ncCohFT preserves it if and only if each operator $r_l$ is a differential operator on $A$ of noncommutative order at most $l+1$.
\end{theorem}
 
\begin{proof}
The proof is analogous to \cite[Th.~1]{DSV2013a}. Basically, the main idea is that conditions in the top homological degree are precisely identities of Proposition~\ref{prop:identities-diff}, and the other conditions follow from them via the recursion relations of Proposition~\ref{prop:abstract-correlators}.	
\end{proof} 
 
The next result is a conceptual formulation of a noncommutative analogue of \cite[Cor.~1]{DSV2013a}; it relates associative $\ncBV_\infty$-algebras to the homotopy  theory of $\ncBV$-algebras.

\begin{theorem}\label{th:as-bv-infty}
Consider a graded associative algebra structure on a graded vector space $A$, which we view as a
Maurer--Cartan element of the dg Lie subalgebra $\mathfrak{g}_{\ncHyperCom}\subset\mathfrak{l}_{\ncBV}$, and let  $r(z)=\sum_{l\ge0}r_lz^l$ be a Maurer--Cartan element of the dg Lie subalgebra $\End(A)[[z]]\subset\mathfrak{l}_{\ncBV}$. The sum  $r(z)+\alpha$ is a Maurer--Cartan element of the $L_\infty$-algebra $\mathfrak{l}_{\ncBV}$ if and only if the associative algebra on $A$ equipped with the operators $\Delta_i:=r_{i-1}$ is an associative $\ncBV_\infty$-algebra.
\end{theorem}

\begin{proof}
The condition of $r(z)+\alpha$ to be a Maurer--Cartan element  amounts to 
 \[
\sum_{n\ge 1}\sum_{p=0}^n\frac{1}{p!(n-p)!}\ell_{n}(\underbrace{r(z),\ldots,r(z)}_{p \text{ times}},\underbrace{\alpha,\ldots,\alpha}_{n-p \text{ times}})=0\ .
 \]
The terms $\frac{1}{n!}\ell_{n}(\underbrace{\alpha,\ldots,\alpha}_{n \text{ times}})$, once grouped together, vanish because $\alpha$ is a Maurer--Cartan element. Let us demonstrate that the terms \[\ell_{n}(\underbrace{r(z),\ldots,r(z)}_{p \text{ times}},\underbrace{\alpha,\ldots,\alpha}_{n-p \text{ times}}) \text{ with } p\ge 2\] vanish individually. Since $\alpha$ and $r(z)$ are Maurer--Cartan elements of the corresponding dg Lie subalgebras of $\mathfrak{l}_{\ncBV}$, any tree in the homotopy transfer formulas which has two leaves with the same parent both labelled by~$i(\alpha)$ or both labelled by $i(r(z))$ gives the zero contribution. Therefore the only potentially non-zero terms in the homotopy transfer formulas along the deformation retract~\eqref{eq:LInftyDefRetract} are made up of trees which have, at each leaf, either $i(r(z))$ or $i(\alpha(z))$, or left combs 
 
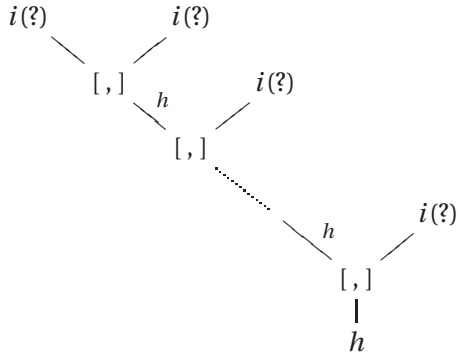
\[
 \xymatrix@=1em{i(?) \ar@{-}[dr] & &i(?)\ar@{-}[dl] & && \\
	&\ar@{-}[dr]^h [ \; , \, ]&&i(?)  \ar@{-}[dl]& &\\
	&&\ar@{..}[dr][ \; , \, ]&& &\\
	&&&\ar@{-}[dr]^h& &i(?)  \ar@{-}[dl]\\
	&&&&[ \; , \, ]\ar@{-}[d]& \\
	&&&&h &   } 
 \]
where each question mark is either $r(z)$ or $\alpha$. Similarly to how it is done in \cite[Th.~5]{DSV2013b}, one may utilise the weight grading of the cooperad $\qncBV^{\ac}$ given by the number of vertices labelled by $\mu$, and the fact that $\delta\otimes H$ increases this weight grading by one. Since $i(\alpha)$ vanishes on elements of weight grading greater than one, we conclude that all the trees with at least two leaves labelled $i(r(z))$ vanish. 

Let us now examine the terms where $r(z)$ occurs exactly once. Similarly to Theorem~\ref{thm:Giv=LinftyAction}, one can show that the contribution of all these terms together is equal to the result of applying the Givental--Lee formulas to $r(z)$ and $\alpha$ (that theorem is proved for $r(z)$ of degree $0$, and we are applying it to $r(z)$ of degree $-1$, so it really is an analogous statement rather than the same one). It remains to apply Theorem~\ref{th:BorjesonIntersection} to complete the proof. 
\end{proof}

\begin{remark}
Theorem \ref{th:as-bv-infty} shows that Definition \ref{nc-homotopy-bv} of associative $\ncBV_\infty$-algebras can be equivalently given from the minimal model viewpoint: an associative $\ncBV_\infty$-algebra is an algebra over the minimal model of $\ncBV$ for which the only operations that do not vanish are the higher homotopies for $\Delta^2=0$ and the binary product. 
\end{remark}

\subsection{Associative \texorpdfstring{$\ncBV_\infty$}{ncBVinfty}-algebras and \texorpdfstring{$\IIB_\infty$}{IIB}-bialgebras}

The purpose of this section is to relate in a precise way the two notions of associative $\ncBV_\infty$-algebras and bialgebras over the properad $\IIB_\infty$, the cobar construction of the Koszul dual of the properad of infinitesimal involutive bialgebras. That properad may be viewed as a noncommutative analogue of the properad $\IBL_\infty$ of homotopy involutive Lie bialgebras, which plays a crucial role in string topology, symplectic field theory, and Lagrangian Floer theory, see, for instance, \cite{EGH00,CFL15}. 

In the first part of this section, we compute the differential Lie algebra which controls the deformation theory of bialgebras over $\IIB_\infty$. This can be seen as the noncommutative analogue of the result that describes homotopy involutive Lie bialgebras as Maurer--Cartan elements in the completion of the Weyl algebra of differential operators; see \cite{MV1} for the general theory and \cite[Th.~4.10]{DCTT2010} for this particular case.

In the second part of this section, we define the bar construction of a $\IIB_\infty$-bialgebra, and establish that this construction provides us with an equivalence of categories between $\IIB_\infty$-bialgebras  and associative $\ncBV_\infty$-algebras with free underlying associative algebra. In the commutative case, such an equivalence between $\IBL_\infty$-bialgebras and commutative $\BV_\infty$-algebras whose underlying commutative associative algebra is free is known; see, for example,  \cite[Prop.~5.4.1]{CMV2014} and \cite[Ex.~10]{MV}.

Note that while, in the classical  case,  the properad $\IBL$ of involutive Lie bialgebras is Koszul~\cite{CMV2014}, the ns properad $\IIB$ of involutive infinitesimal bialgebras is not~\cite{MV1}; for our purposes, this merely means that we always work with the properadic cobar complexes, and do not infer results on the homotopy category of infinitesimal unimodular bialgebras from our results on algebras over the cobar complexes. 

\subsubsection{Involutive infinitesimal bialgebras and Frobenius bialgebras} In this section, we discuss the properads of involutive infinitesimal bialgebras and Frobenius bialgebras, and their basic properties.  

The following definition goes back to Joni and Rota \cite{JR79} (where of course the language of bialgebras, not the language of properads, is used).

\begin{definition}\label{def:InfUni}
The properad $\IIB$ of \emph{involutive infinitesimal bialgebra} is the ns properad with two generators $m\in\IIB_0(\underline{2},\underline{1})$ and $\delta\in\IIB_0(\underline{1},\underline{2})$ which satisfy the following relations:
\begin{gather}
m\circ (m\otimes\id)=m\circ (\id\otimes m) ,\label{eq:assoc}\\
(\delta\otimes\id)\circ\delta=(\id\otimes\delta)\circ\delta ,\label{eq:coassoc}\\
m\circ\delta=0 , \label{eq:unimodular}\\
\delta\circ m=(m\otimes\id)\circ(\id\otimes\delta)+(\id\otimes m)\circ(\delta\otimes\id) . \label{eq:infHopf}
\end{gather} 
\end{definition}
Relations \eqref{eq:assoc} and \eqref{eq:coassoc} express, respectively, associativity of $m$ and coassociativity of $\delta$.  Relation~\eqref{eq:infHopf} is an infinitesimal version of the compatiblity relation in a bialgebra, see \cite{Ag01}. The property expressed  by Relation~\eqref{eq:unimodular} is usually termed ``involutive''.  

\smallskip

With a re-grading of one of the operations, it is possible to construct $\ncBV$-algebras from involutive infinitesimal bialgebras; this is a noncommutative analogue of the construction of \cite[Sec.~2.10]{Ginz06}.

\begin{example}\label{ex:IIB-BV} 
Suppose that $(A,m,\delta)$ be an involutive infinitesimal  bialgebra, which is the same as Definition~\ref{def:InfUni} in which the coassociative coproduct $\delta$ has degree~$2$. In this case, the $\ncGerst$-algebra structure of the bar construction $\big(\overline{T}(s^{-1}A),m\big)$ of Example \ref{ex:bar-ncGerst} may be extended to a $\ncBV$-algebra given by  the formula
  \[
\Delta=
\sum_{1\leqslant i<n}\id^{\otimes(i-1)}\otimes(s^{-1}m\circ(s\otimes s))\otimes\id^{\otimes(n-1-i)}+\sum_{1\leqslant i\leqslant n}\id^{\otimes(i-1)}\otimes((s^{-1}\otimes s^{-1})\delta s)\otimes\id^{\otimes(n-i)} \ .
 \]
\end{example}

\begin{definition}
The properad $\ncFrob$ of \emph{noncommutative Frobenius bialgebras} is the ns properad with two generators $m\in\ncFrob_0(\underline{2},\underline{1})$ and $\delta\in\ncFrob_0(\underline{1},\underline{2})$ which satisfy the following relations: 
\begin{gather*}
m\circ (m\otimes\id)=m\circ (\id\otimes m) ,\\
\delta\circ m=(\id\otimes m)\circ(\delta\otimes\id),\\
\delta\circ m=(m\otimes\id)\circ(\id\otimes\delta),\\
(\delta\otimes\id)\circ\delta=(\id\otimes\delta)\circ\delta .
\end{gather*}
\end{definition}

\begin{proposition}\label{prop:ncFrob}\leavevmode
\begin{itemize}
\item[(i)] For each $m,n\ge 1$, and each $g\ge 0$,  the genus $g$ component with $m$ inputs and $n$ outputs $\ncFrob_g(m,n)$ of the ns properad $\ncFrob$ is one-dimensional. 
\item[(ii)] The ns properads $\ncFrob$ and $\IIB$ are Koszul dual to each other.
\item[(iii)] The ns properad $\ncFrob$ is not Koszul. 
\item[(iv)] The Koszul dual coproperad $\IIB^{\ac}$ of the properad $\IIB$ has as its underlying space the collection $(\As^{\ac})^{\mathrm{op}} \otimes T^c(z)\otimes \As^{\ac}$, where ``$\mathrm{op}$'' refers to the ``opposite'' cooperad obtained by reversing inputs and outputs, and $z$ is a formal variable of degree $2$;  its powers  account for the genus in the coproperad. 
\end{itemize}
\end{proposition}

\begin{proof}
Statements (i) and (ii), and (iv) are proved by a direct computation, statement (iii) is established in \cite[Lemma~46]{MV1}.  
\end{proof}

\begin{definition}
We shall refer to algebras over the cobar construction $\Omega(\IIB^{\ac})$ as \emph{$\IIB_\infty$-algebras}.
\end{definition}

\begin{remark}
A suitable homotopy coherent notion of an involutive infinitesimal bialgebra should involve strictly more generating operations (homotopies), corresponding to more syzygies in the cofibrant properadic resolution. However, even though the ns properad $\IIB$ is not Koszul, bialgebras over the cobar construction $\Omega (\IIB^{\ac})$ share nice homotopy properties, since this latter properad is cofibrant. This gives at least one reason to consider this algebraic structure.
\end{remark}

\subsubsection{Noncommutative Weyl algebras}

Let $V$ be a graded vector space, which we consider to be either finite-dimensional or bounded from below with finite-dimensional graded components. In the commutative case, the Weyl algebra of a vector space $V$ is the algebra of differential operators on $V$. A coordinate-free definition of (a completion of) this algebra is presented in~\cite{DCTT2010,CFL15}; we mimic the latter definition here.

\begin{definition}
Let $f\in\Hom(V^{\otimes p},V^{\otimes q})$, $g\in\Hom(V^{\otimes r},V^{\otimes s})$, and let $k\ge 1$ be an integer. The \emph{partial $k$-composition $f\circ^{(k)}g$} is defined by the formula
	\[
f\circ^{(k)}g := 
\begin{cases}
0, \text{ if } k>\min(p,s),\\
\sum\limits_{i+j=s-p}(\id^{\otimes i}\otimes f\otimes\id^{\otimes j})\circ g, \text{ if } k=p<s, \\
\sum\limits_{i+j=p-s}f\circ (\id^{\otimes i}\otimes g\otimes\id^{\otimes j}), \text{ if } k=s<p, \\
(f\otimes\id^{\otimes(s-k)})\circ(\id^{\otimes(p-k)}\otimes g)+ 
(\id^{\otimes(s-k)}\otimes f)\circ(g\otimes\id^{\otimes(p-k)}), \text{ if } k<\min(p,s) . 
\end{cases} 
	\]
\end{definition}

Besides the non-unital algebra of noncommutative polynomials
 \[
\overline{T}(V):=\bigoplus_{n\ge 1}V^{\otimes n}\ ,
 \]  
 we shall also use non-unital formal noncommutative power series 
\[
\widehat{{{\overline{T}}(V)}}:=\prod_{n\ge 1}V^{\otimes n} \ .
\]  
It is clear that the partial $k$-composition extends to a well defined map
 \[
\circ^{(k)}\colon \Hom(\overline{T}(V),\widehat{\overline{T}(V)})\otimes \Hom(\overline{T}(V),\widehat{\overline{T}(V)})\to \Hom(\overline{T}(V),\widehat{\overline{T}(V)}).
 \]

\begin{definition}\label{def:ncWeyl}
Let $\hbar$ be a formal variable. The \emph{noncommutative Weyl algebra} $(\ncW(V),\star)$ of a graded vector space $V$ is the $\k[[\hbar]]$-module 
 \[
\ncW(V):=\Hom\left(\overline{T}(V),\widehat{\overline{T}(\hbar V)}\right)\otimes \tfrac1{\hbar}\k[[\hbar]]\ ,
 \]
endowed with  the $\k[[\hbar]]$-linear binary operation $\star$ defined by the formula
 \[
f\star g:=\sum_{k\ge 1}\big(f\circ^{(k)}g\big)\, \hbar^{k-1}\ , 
 \]
 for $f,g\in \ncW(V)$.
\end{definition} 

\begin{theorem}\label{th:IIBandMC}
Let $V$ be a chain complex bounded from below. 
\begin{enumerate}
\item If we assign to the formal variable $\hbar$ the homological degree $-2$, the noncommutative Weyl algebra $\ncW(sV)$, with the differential induced from that of $V$, becomes a differential graded Lie-admissible algebra. 
\item 
There is a one-to-one correspondence between $\IIB_\infty$-bialgebra structures  on $V$ and Maurer--Cartan elements of the differential graded Lie algebra associated to $\ncW(sV)$. 
\end{enumerate}
\end{theorem}

\begin{proof}
We prove both assertions at the same time using the formalism of \cite{MV1} as follows. Let $\End_V$ denote the endomorphism properad of $V$.
The convolution algebra encoding $\IIB_\infty$-bialgebras is isomorphic to 
 \[
\Hom(\IIB^{\ac}, \End_V)\cong \Hom(T^c(z)\otimes\As^{\ac}(V),\widehat{(\As^{\ac})^*(V)})\cong s^2\Hom(\overline{T}(sV),\widehat{\overline{T}(s^{-1}V)})[[\hbar]] .
 \]
Here $\hbar=z^*$, either of these formal variables essentially counts the genus of operations in the properad. The coproperad structure on $\IIB^{\ac}$ induces a Lie-admissible product on this convolution algebra. 
Notice that 
 \[
s^2\Hom(\overline{T}(sV),\widehat{\overline{T}(s^{-1}V)})[[\hbar]] \cong \Hom\left(\overline{T}(sV),\widehat{\overline{T}(\hbar sV)}\right)\otimes \tfrac1{\hbar}\k[[\hbar]]=\ncW(sV) ,
 \]
and moreover, in the present case, it is straightforward to check that the Lie-admissible product of the convolution algebra is equal to the product $\star$ under this isomorphism. (Essentially, it boils down to the fact that the elements appearing in partial $k$-composition are the only compositions in the ns properad $\ncFrob$.) Both  statements immediately follow. We invite the reader to compare this with \cite{CFL15,DCTT2010}.
\end{proof}

\begin{corollary}\label{cor:MCWeylinfinvBiinf}
An $\IIB_\infty$-bialgebra structure structure on a graded vector space $V$ amounts to a collection of maps 
$$\mu_{n,m}^g\ : \ V^{\otimes n} \to V^{\otimes m}$$
of degree $|\mu_{n,m}^g|=n+m+2g-3$, for any $n, m\ge 1$ and $g\ge 0$, satisfying 
$$\mu \star \mu=0\ , $$
where $\mu=\displaystyle\sum_{\substack{n,m\ge 1,\\ g\ge 0}} \mu_{n,m}^g$.
In particular, the square-zero map $\mu^0_{1,1}$ is the underlying differential.
\end{corollary}

\subsubsection{Bar construction of $\IIB_\infty$-bialgebras}\label{subsubsec:BarHoInvInfBi}

In this section, we extend the bar construction of $A_\infty$-algebras to $\IIB_\infty$-bialgebras. 

Recall that the bar construction of an $A_\infty$-algebra $\lbrace m_n\rbrace_{n\ge 1}$ on $V$ is the (non-counital) cofree conilpotent coassociative coalgebra $\overline{T}^c(sV)$ equipped with the unique square-zero degree $-1$ coderivation $d=d_{m_1}+d_{m_2}+\cdots$ extending the operations $m_n : V^{\otimes n} \to V$. In \cite{TTW11}, it is proved that if one considers the underlying vector space $\overline{T}^c(sV)$ as a commutative associative algebra with respect to the shuffle product, then each map $d_{m_n}$ is a differential operator of order at most $n$.  In \cite{Bor13}, the cohomological version of the bar complex $\overline{T}^c(s^{-1}V)$ was used; this is necessary if one wants to match the degree convention of $\ncBV_\infty$-algebras. The same phenomenon will be observed below. 

\smallskip 

In general, for any Koszul operad $\calP$, the bar construction of a $\calP_\infty:=\Omega(\calP^{\ac})$-algebra $(A, \mu)$ is the cofree $\calP^{\ac}$-coalgebra $\calP^{\ac}(A)$ equipped with the unique coderivation $d_\mu$ extending the structure map~$\mu$. This holds true thanks to the bijection 
 \[
\mathrm{MC}\left(
\Hom_{\mathbb{S}}(\calP^{\ac}, \End_A)\right)\cong \mathrm{Codiff}\left(
\calP^{\ac}(A)
\right) 
 \]
between Maurer--Cartan elements in the corresponding convolution algebra and degree $-1$ square-zero coderivations, called \emph{codifferentials}, on cofree $\calP^{\ac}$-coalgebras. 
Such a relationship cannot be true in general on the level of properads since the notion of a (co)free bialgebra over a (co)properad is not available. However, in particular cases like $\IIB_\infty$-bialgebras, one can fix this as follows. The key point relies in the form of the coproperad $\IIB^{\ac}$ that we described in Proposition~\ref{prop:ncFrob}; as we already saw in the proof of Theorem~\ref{th:IIBandMC}, because of that form of the coproperad $\IIB^{\ac}$ we have 
 \[
\Hom(\IIB^{\ac}, \End_V)\cong \Hom\left(\overline{T}(sV),\widehat{\overline{T}(\hbar sV)}\right)\otimes \tfrac1{\hbar}\k[[\hbar]]=\ncW(sV)  .
 \]
As a consequence, we will define a bar construction for $\IIB_\infty$-bialgebras on $V$ utilising the tensor module $\overline{T}(sV)$. The concatenation product on that module allows us to use differential operators on noncommutative algebras, and obtain a certain equivalence between $\IIB_\infty$-algebras and $\ncBV_\infty^{\mathrm{\, assoc}}$-structures on free algebras.

\smallskip 

Let us consider the graded vector space $T^c(z)\otimes\As^{\ac}(V)=\overline{T}(sV)\otimes\k[z]$ which already appeared in the proof of Theorem~\ref{th:IIBandMC}.
We equip it with the product induced by the concatenation product of $\overline{T}(sV)$:
 \[
m(sv_1\otimes \cdots\otimes sv_k\,z^{g'}, sv_{k+1}\otimes\cdots \otimes sv_n\, z^{g''})=sv_1\otimes \cdots\otimes sv_n\,z^{g'+g''}
 \]
The statement of Proposition~\ref{prop:NCDiffOpTensor} extends \emph{mutatis mutandis} to $\overline{T}(sV)[z]$, if we adjust the definition of the mapping  $f\mapsto D_f$ appropriately. For $f \colon (sV)^{\otimes l}\otimes z^g \to \overline{T}(sV)$, we consider the map $D_f\colon \overline{T}(sV)[z] \to  \overline{T}(sV)[z]$ defined by
 \[
D_f (sv_1\otimes \cdots\otimes sv_n\otimes z^k):= 
\left\lbrace
\begin{array}{cl}
\displaystyle \sum_{i+j+l=n}(\id_V^{\otimes i}\otimes f\otimes\id_V^{\otimes j})(sv_1\otimes \cdots\otimes sv_n\otimes z^k)\cdot z^{k-g}\ , & \text{ for } n\ge l,\, k\ge g , \\
0\ , &\text{ otherwise} .
\end{array}
\right.
 \]

\begin{proposition}\label{prop:LieWeyl}
The mapping $f \mapsto D_f$,
 \[
\Hom(\overline{T}(sV),{\overline{T}(sV)})[[\hbar]]\cong  
\Hom(\overline{T}(sV)[z],{\overline{T}(sV)}\big) \to \End\big(\overline{T}(sV)[z]\big)
 \]
is a morphism of differential graded Lie algebras. 
\end{proposition}

\begin{proof}
A direct computation shows that any pair of maps $\phi, \psi \colon \overline{T}(sV)[z] \to  \overline{T}(sV)$  satisfy the following relation in $\End(\overline{T}(sV)[z])$:
 \[
[D_\phi, D_\psi]=D_{\lbrack \phi, \psi\rbrack_\star}\ . 
 \]
\end{proof}

\begin{definition}\leavevmode
\begin{itemize}
\item An $\IIB_\infty$-bialgebra $\big(V, \big\lbrace\mu_{n,m}^g\big\rbrace\big)$ is called \emph{locally nilpotent} if, for any $n\ge 1$ and any $v_1, \ldots, v_n \in V$, only a finite number of $\mu^g_{n,m}(v_1, \ldots, v_n)$, for $g\ge 0$ and $m\ge 1$, is non-trivial.
\item The \emph{bar construction} of a locally nilpotent $\IIB_\infty$-bialgebra $\big(V, \big\lbrace\mu_{n,m}^g\big\rbrace\big)$ is the $\k[z]$-extension of the cofree coassociative coalgebra on $sV$, that is $\overline{T}(sV)[z]$ equipped with the operators $\Delta_l$  defined by 
 \[
\Delta_l:=\sum_{\substack{g\ge 0, \\ m\ge 1}} D_{f^g_{l+1,m}}\ , 
 \]
where $f^g_{l+1,m}:= s^{\otimes m}z^{-g}\,\mu^g_{l+1,m} (s^{-1})^{\otimes(l+1)}\colon (sV)^{\otimes (l+1)}\, z^g\to (sV)^{\otimes m}$.
\end{itemize}
\end{definition}
 
To state the following result, we shall consider the cohomological version of the bar complex where we replace $sV$ by $s^{-1}V$ to match the degree convention of $\ncBV_\infty$-algebras. We denote by $\Diff^{\,(l)}\big(\overline{T}(s^{-1}V)[z]\big)$ the image of $\Hom((s^{-1}V)^{\otimes l}[z],\overline{T}(s^{-1}V))$ under the mapping $f\mapsto D_f$.

\begin{theorem}\leavevmode
\begin{enumerate}
\item The $\Delta_l$'s operators of the the cohomological bar construction of a locally nilpotent $\IIB_\infty$-bialgebra endows it with a structure of  an associative $\ncBV_\infty$-algebra in which each operator $\Delta_l$ vanishes on $(s^{-1}V)^{\otimes n}[z]$ for $n\le l$, i.e. $\Delta_l\in \Diff^{\,(l+1)}\big(\overline{T}(s^{-1}V)[z]\big)$.
\item The datum of an associative $\ncBV_\infty$-algebra structure on the associative algebra $T(s^{-1}V)[z]$ in which each operator $\Delta_l$ vanishes on $(s^{-1}V)^{\otimes n}[z]$ for $n\le l$, is equivalent to the datum of a locally nilpotent $\IIB_\infty$-bialgebra structure on $V$.
\end{enumerate}
\end{theorem}

\begin{proof}\leavevmode
The first statement is a direct consequence of Corollary~\ref{cor:MCWeylinfinvBiinf} and Proposition~\ref{prop:LieWeyl}.

To establish the second statement, let $\big(\overline{T}(s^{-1}V)[z], \lbrace \Delta_l\rbrace\big)$ be an associative $\ncBV_\infty$-algebra structure such that $\Delta_l\in \Diff^{\,(l+1)}\allowbreak\big(\overline{T}(s^{-1}V)[z]\big)$. The $\k[z]$-extension of  Proposition~\ref{prop:NCDiffOpTensor} shows the data of $\Delta_l$ is equivalent to the data of maps $\mu_{l+1, m}^g : V^{\otimes l+1} \to V^{\otimes m}$ of degree $l+1+m+2g-3$, for any $g\ge 0$ and $m\ge1$, satisfying 
 \[
\Delta_l=\sum_{\substack{\g\ge 0, \\ m\ge 1}} D_{(s^{-1})^{\otimes m}z^{-g}\,\mu^g_{l+1,m} s^{\otimes(l+1)}} .
 \]
Since the mapping $f\mapsto D_f$ is injective, Proposition~\ref{prop:LieWeyl} and Corollary~\ref{cor:MCWeylinfinvBiinf} show that the maps $\mu^g_{n,m}$ satisfy the relations of an $\IIB_\infty$-bialgebra. This latter bialgebra is locally nilpotent by construction. 
\end{proof}

\begin{remark}\leavevmode
\begin{itemize}
\item In the commutative case, such a result plays a crucial role in the formulation of  symplectic field theories, see \cite{EGH00}.
\item This result is the full generalization of construction from Example~\ref{ex:IIB-BV}.
\end{itemize}
\end{remark}

\subsection{The operad $\ncHyperCom$ as a homotopy quotient of $\ncBV$}

From results of Section \ref{sec:minimal-model-ncBV}, it can already be deduced, in the same way it is done in \cite{DCV2013}, that the operad $\ncHyperCom$ is a homotopy quotient of the operad $\ncBV$ by the operation $\Delta$. However, using the Givental action, it is possible to construct an explicit quasi-isomorphism from $\ncHyperCom$ to a dg operad representing this homotopy quotient. In this section, we accomplish that. We mainly follow the approach of~\cite{KMS2013}. However, since we work with nonsymmetric operads, we prefer to write down all the formulas in such a way that they do not contain unnecessary denominators.  A geometric version of the statement on the homotopy quotient is discussed in the next section.

\subsubsection{Homotopy trivialization of the circle action} Following~\cite{DSV2013b}, we view trivialised circle actions as follows. For a chain complex $A$, we consider the associative convolution dg algebra
 \[
\overline{\mathfrak{a}}:=\left(\Hom\left(\overline{T}^c(\delta), \End(A)\right), \star\right)\ ,
 \] 
which gives rise to a dg Lie algebra controlling homotopy circle actions on $A$. We also consider the extended convolution algebra
 \[
\mathfrak{a}:=\left(\Hom\left({T}^c(\delta), \End(A)\right), \star\right)\ .
 \]  
In the latter convolution algebra, we denote by $1$ and by $\partial$ the maps defined respectively by 
 \[
1 \ : \ 1 \mapsto \id_A,  \quad     \partial \ : \ 1 \mapsto \partial_A,\quad  \text{and}\quad \delta^k\mapsto 0, \ \ \text{for}\ \ k\ge 1\ . 
 \]  
The data of a trivialisation of a homotopy circle action $\phi \in \MC(\overline{\mathfrak{a}})$ on $A$ amounts to a degree $0$ element 
$f\in \overline{\mathfrak{a}}$ satisfying the following equation in the algebra $\mathfrak{a}$:
 \[
(1+f)\star(\partial+\phi)=\partial \star (1+f)\ . 
 \]
This data is equivalent to  a module structure over the quasi-free dg algebra 
 \[
\mathsf{T}:=\left(T\left(s^{-1} \overline{T}^c(\delta) \oplus  \overline{T}^c(\delta) \right), d_1+d'_2+d_2'' \right)\ ,
 \]
where the differential $d_1$ is the unique derivation which extends the desuspension map on the space of generators 
$s^{-1}\colon  \overline{T}^c(\delta) \to s^{-1}\overline{T}^c(\delta)$,  
the differential $d'_2$ is the unique derivation which extends the comultiplication map followed by the desuspension map  
 \[
\overline{T}^c(\delta) \to \overline{T}^c(\delta)\otimes \overline{T}^c(\delta)
 \xrightarrow{s^{-1}}  s^{-1}\overline{T}^c(\delta)\otimes \overline{T}^c(\delta)\ , 
 \]
and the differential $d_2''$ is the differential coming form the Koszul resolution
$$\Omega \, H^\bullet(S^1)^{\ac}=\left(T\left(s^{-1} \overline{T}^c(\delta) \right), d_2 \right)\ .$$

\begin{proposition}
The homotopy quotient of the ns operad $\ncBV$ by the action of $\Delta$ is represented by the coproduct  
 \[
\ncBV/_h\,\Delta:=\ncBV\vee_{\Omega \, H^\bullet(S^1)^{\ac}}\mathsf{T}\ . 
 \]
\end{proposition}

\begin{proof}
Both the algebra $\mathsf{T}$ and the operad $\ncBV$ admit a map from the  
the algebra $\Omega \, H^\bullet(S^1)^{\ac}$. In the case of $\mathsf{T}$, it is tautological, while in the case of $\ncBV$, it sends $s^{-1}\delta^k$ to $\Delta$ for $k=1$, and to $0$ for $k>1$. The result then follows from the arguments of \cite[Prop.~2.2]{KMS2013}. 
\end{proof}

In plain words, an algebra over the ns operad $\ncBV/_h\,\Delta$ amounts to  a $\ncBV$-algebra $A$ equipped with extra unary operators $a_1,a_2,\ldots$ satisfying the equation
 \[
a(z)\circ(\partial_A+\Delta z)=\partial_A \circ a(z) ,
 \]
where $a(z)=1+a_1z+a_2z^2+\cdots$. 

\subsubsection{The ns operad $\ncHyperCom$ as representative of the homotopy quotient}

\begin{theorem}\label{th:ncHyperComQI}
The two dg ns operads  $\ncHyperCom$ and $\ncBV/\Delta$ are quasi-isomorphic.
\end{theorem}

\begin{proof}
Let us  define a map $\Theta$ of dg ns operads from  $\ncHyperCom$ to $\ncBV/_h\, \Delta$. On generators, $\Theta(\nu_n)$ is equal to a certain sum of terms indexed by all planar rooted trees $t$ with $n$ leaves, taken with appropriate combinatorial coefficients. To describe that sum, let us consider the ns operad $\ncBV$ over the ring of formal power series in variables $\lambda_e$ indexing all the half-edges of the tree $t$ (each input and each output of each internal vertex of $t$). We decorate the tree $t$ as follows:
\begin{itemize}
\item each internal vertex of $t$ with $p$ inputs is decorated by the $p$-ary operation of $\ncBV$ which is the $k$-fold product,
\item each half-edge $e$ of $t$ which is a leaf is decorated by the unary operation $a^{-1}(-\lambda_e)$, 
\item the half-edge $e$ which is the output of the root of $t$ is decorated by the unary operation $a(\lambda_e)$,
\item each internal edge of $t$ made of an output half-edge $e'$ and an input half-edge $e''$ is decorated by the unary operation  
 \[
\frac{a^{-1}(-\lambda_{e''})\circ a(\lambda_{e'})-1}{\lambda_{e'}+\lambda_{e''}} \ .
 \]
\end{itemize}
Such a decorated tree $t$ produces an element of the operad $\ncBV$ in the following way. We first expand this decorated tree as an element of $\ncBV$ with coefficients being formal power series in variables $\lambda_e$, and evaluate these power series in the commutative ring 
 \[
\bigotimes_{v \text{ an internal vertex of } t}H^*(\calB_{\mathbb{C}}(I_v)) \ , 
 \]
where $I_v$ is the ordered set of inputs of $v$, substituting for $\lambda_e$ the $\psi$-class $\psi_e$ of the corresponding half-edge. After that, we integrate each product of $\psi$-classes over the corresponding manifold $\calB_{\mathbb{C}}(I_v)$ using formulas of Proposition~\ref{prop:nc-correlators}. Finally, the element $\Theta(\nu_n)$ is equal to the sum of all thus obtained elements of the operad $\ncBV$ . We extend $\Theta$ to a map $\ncHyperCom\to\ncBV/\Delta$ as a morphism of operads. (Note that the actual formulas are much simpler than the commutative case, since, for example, Proposition~\ref{prop:nc-correlators} ensures that $\lambda_e^2$ evaluates as zero for each input half-edge). The result is then proved analogously to~\cite[Th.~5.8]{KMS2013}.
\end{proof}

\subsection{The real oriented blow-up construction}\label{sec:KSV}

In this section, we consider the operadic structure on the principal $(S^1)^{n+1}$-bundle over the real oriented blow-up of the boundary divisor that mimics the constructions for the usual moduli spaces of curves due to Kimura--Stasheff--Voronov and Zwiebach, see~\cite{KSV,LS,LSS,Zwi93}.

\subsubsection{Basic definitions}

\begin{definition}
Let $n\ge 1$. We denote by $\calB^{\,\mathrm{or}}(\underline{n})\to\calB_\mathbb{C}(\underline{n})$ the real oriented blow-up of $\calB_\mathbb{C}(\underline{n})$ along the boundary divisor. Pointwise, we just choose along each divisor $\overline{\calB_\mathbb{C}(\underline{n},T_{l,r})}$ a unit vector in its normal bundle 
 \[
\left(V_{l,r}/G(l,r)\right)\otimes \left(G(l,r)/V_{l+1,r-1}\right)^*.
 \]
We denote by $\calZ(\underline{n})$ the total space of the  $(S^1)^{n+1}$-bundle over $\calB^{\,\mathrm{or}}(\underline{n})$ corresponding to additionally choosing unit vectors in each of the spaces
 \[
V_{1,1},\dots,V_{n,n}, \text{ and } \left(G(\underline{n})/V_{2,n-1}\right)^*.
 \]
\end{definition}

\begin{proposition}
The collection of topological spaces $\{\calZ(\underline{n})\}$ is naturally equipped with a structure of a ns operad.
\end{proposition}

\begin{proof}
In order to define the operadic composition $\circ^J_{I,i}\colon \calZ(I)\times \calZ(J)\to \calZ(I\sqcup_i J)$, we use the operadic composition on $\calB_\mathbb{C}$ from Definition~\ref{def:operadic-composition} and the tensor product of the unit circles in $V_{i,i}^1$ and $(G(J)/V_{s(\min(J)),p(\max(J))})^*$ in order to get a unit vector in the normal bundle of the corresponding divisor in the target space $\calB_\mathbb{C}(I\sqcup_i J)$. The other circles in the fibres of $\calZ(I)\to\calB^{\,\mathrm{or}}(I)$ and  $\calZ(J)\to\calB^{\,\mathrm{or}}(J)$ are canonically identified with the corresponding circles in the fibre of $\calZ(I\sqcup_i J)\to\calB^{\,\mathrm{or}}(I\sqcup_i J)$. The ns operad axioms are straightforward.
\end{proof}

\subsubsection{Homotopy equivalence}  As a topological space, the real oriented blow-up of $\calB_\mathbb{C}(\underline{n})$ is homotopically equivalent to  $\calB_\mathbb{C}(\underline{n})$ with (a tubular neighbourhood of) the boundary divisor removed.
In other words, we can consider  $\calB_\mathbb{C}(\underline{n},T_n)\cong(\mathbb{C}^\times)^{n-2}$. To obtain a real manifold with boundary, we should compactify each copy of $\mathbb{C}^\times$ (they corresponds to the choices, for  $i=2,\dots,n-1$, of subspaces $V_{i,i}\subset G(i-1,i+1)$ different from coordinate axes) by two circles, obtaining cylinders that we denote by $C_i$, for $i=2,\dots,n-2$. 

Thus, the space $\calZ(\underline{n})$ is the total space of a $(S^1)^{n+1}$-bundle over $C_2\times\cdots\times C_{n-2}$. More precisely, $(S^1)^{n+1}=S^1_0\times S^1_1\times\cdots\times S^1_n$, where $S_1^1$ and $S_n^1$ are fibred trivially (they are the circles in $V_{1,1}$ and $V_{n,n}$, respectively), and where $S^1_i$, for $i=2,\dots,n-1$, are the unit circles in the fibres of $\calO(-1)|_{C_i}$, and $S^1_0$ is the unit circle in the line $(G(\underline{n})/\oplus_{i=2}^{n-1}V_{i,i})^*$. Still, homotopically the space $\calZ(\underline{n})$ is the product of $2n-1$ circles. Indeed, each cylinder $C_i$ can be contracted to a circle. Then we get the total space of a nontrivial principal torus bundle over a torus, which is the standard torus (with a different choice of the standard basis in the homology than the one suggested naturally by the fibre bundle).

\begin{proposition}
The ns topological operads $\calZ$ and $\As_{S^1}\rtimes S^1$ are homotopy equivalent.
\end{proposition}

\begin{proof}
We denote by $a_0,a_1,\dots,a_n$ the elements of $H_1(\calZ(\underline{n}))$ that correspond, respectively, to the circles $S^1_0, S^1_1, \dots, S^1_n$. We also denote by $b_i,c_i\in H_1(\calZ(\underline{n}))$ the two boundary circles of the cylinder $C_i$,for  $i=2,\dots,n-1$, where $b_i$ (respectively $c_i$) is the circle in $G(i-1,i)$ (respectively $G(i,i+1)$). The cycles  
 \[
a_0,\dots,a_n, b_2,c_2,\dots,b_{n-1},c_{n-1}  
 \]
generate $H_1(\calZ(\underline{n}))$, but they are not linearly independent. 
	
It is easy to see by direct computation of the Chern class that all circles $S^1_i$, for $i=0,\dots,n$ are fibred trivially over $C_j$ except for $S^1_0$ and $S^1_j$, which are the circles in the fibres of $\calO(-1)$ on $\mathbb{P}(G(j-1,j+1))$, for $j=2,\dots,n-1$. This implies the following relations between the cycles:
\begin{equation}\label{eq:relations-H1-blow-up}
a_0+a_j=b_j+c_j,\qquad j=2,\dots,n-1,
\end{equation}
and there are no further relations. Note that this way we recover the equality $\dim H_1 = 2n-1$, as we expect, so $\calZ(\underline{n})$ is homotopically equivalent to the $2n-1$-dimensional torus.
	
In order to establish this equivalence explicitly, we consider a torus that is a product of $(n-1)+n$ circles, whose generators in $H_1$ are denoted by $x_{1,2},\dots,x_{n-1,n}$ and $y_1,\dots,y_n$ (we choose this notation in order to make it closer to the structure of $\As_{S^1}\rtimes S^1$). Then we observe that
\begin{gather*}
a_0  =\sum_{i=1}^{n-1} x_{i,i+1} + \sum_{i=1}^n y_i,\\ 
a_j  = y_j, \quad  j=1,\dots,n, \\ 
b_j  = \sum_{i=1}^{j-1} x_{i,i+1} + \sum_{i=1}^j y_i, \quad j=2,\dots,n-1, \\ 	
c_j  = \sum_{i=j}^{n-1} x_{i,i+1} + \sum_{i=j}^n y_i, \quad j=2,\dots,n-1.
\end{gather*}
satisfy the relations~\eqref{eq:relations-H1-blow-up}. This establishes the explicit homotopy equivalence of $\As_{S^1}\rtimes S^1$ and $\calZ(\underline{n})$.
	
Let us now how the operad compositions match. Consider the operadic product in $\As_{S^1}\rtimes S^1$ corresponding to the tree $T_{l,r}$. On the level of the corresponding topological spaces, it is the product of the tori $T_1$ and $T_2$ over an identification of some circles. The torus $T_1$ is the product of the circles corresponding to the cycles $x_{1,2},\dots,x_{l-1,l},x_{r,r+1},\dots,x_{n-1,n}$ and  $y_1,\dots,y_{l-1},y_{\star},y_{r+1},\dots,y_n$.	
The torus $T_2$ is the product of the circles corresponding to the cycles
$x_{l,l+1},\dots,x_{r-1,r} $ and $y_{l},\dots,y_{r}$.  
The product of $T_1$ and $T_2$ is taken over the identification of $y_\star$ with the diagonal circle in the torus~$T_2$.
	
In order to compare this operadic product to the one in $\calZ(\underline{n})$, we have to recall the local description of the stratum $\calB(\underline{n},T_{l,r})$ in Section~\ref{sec:local-structure} (which implies the description of its real oriented blow-up). On the level of cycles $a_\bullet$, $b_\bullet,c_\bullet$, the interpretation of the local structure of
	\[ 
	\circ^{\{l,l+1,\dots,r\}}_{\{1,\dots,l-1,\star,r+1,\dots,n\},\star}(\calZ(\{1,\dots,l-1,\star,r+1,\dots,n\})\times \calZ(\{l,\dots,r\}))
	\]
inside $\calZ(\underline{n})$ is the following. 
	
Assume, for convenience, that $l\not=1$, $r\not=n-1$ (in the special cases when $l=1$ or $r=n-1$ the description is completely analogous). The first space $\calZ(\{1,\dots,l-1,\star,r+1,\dots,n\})$ is formed by the cycles $a_0,a_1,\dots,a_{l-1},a_{r+1},\dots,a_n$, the cycles $b_2,c_2$, \dots, $b_{l-1},c_{l-1}$, $b_\star:=c_l$, $c_\star:=b_r$, $b_{r+1},c_{r+1}$, \dots, $b_{n-1},c_{n-1}$, and an extra cycle $a_\star$ (corresponding to the circle in the normal bundle), satisfying the equation $b_\star+c_\star=a_0+a_\star$. Observe that this relation is satisfied for 
	\[
	a_\star=\sum_{i=l}^{r-1} x_{i,i+1}+\sum_{i=l}^r y_i.
	\]
	
The second space $\calZ(\{l,\dots,r\})$ is formed  by the cycles 
	 \[
a_\star,a_l,\dots,a_r, \text{ and  } b'_{l+1},c'_{l+1},\dots,b'_{r-1},c'_{r-1}	 
	 \]
 satisfying the relations
	\[ 
	b_j'+c_j'=a_\star+a_j,\qquad j=l,\dots,r.
	\]
Observe that this relation is satisfied for 
\begin{gather*}
b_j' =\sum_{i=l}^{j-1} x_{i,i+1} + \sum_{i=1}^j y_i, \qquad j=l+1,\dots,r-1,\\
c_j' =\sum_{i=j}^{r-1} x_{i,i+1} + \sum_{i=j}^r y_i, \qquad j=l+1,\dots,r-1.
\end{gather*}
	
Thus we see that the circle $a_\star$ acts diagonally on the torus corresponding to $\calZ(\{l,\dots,r\})$, which is exactly the definition of the operadic composition in Section~\ref{sec:g-action-operadic-composition}.
\end{proof}

\subsubsection{Conjecture on the homotopy quotient}

The analogous statement for the moduli space of curves is that the ns operad formed by the principal $S^{n+1}$-bundles over the real oriented blow-ups of the spaces $\overline{\calM}_{0,n+1}(\mathbb{C})$, for $n=2,\dots$, is homotopy equivalent to the framed little disks operad~\cite{KSV}. It is proved in~\cite{DC2014} that the homotopy quotient of the framed little disks operad modulo the circle action is equivalent to the operad of the moduli spaces $\overline{\calM}_{0,n+1}(\mathbb{C})$. We conjecture that the same holds in the noncommutative case. 

\begin{conjecture}
The homotopy quotient of the ns operad  $\As_{S^1}\rtimes S^1$ by the circle action is equivalent to the complex brick operad $\calB_\mathbb{C}$. 	
\end{conjecture}

This conjecture is a geometric version of the result that holds true on the algebraic level, see Theorems \ref{th:BVDefRetract} and \ref{th:ncHyperComQI}. 
	
\section{Formality theorems}\label{sec:Formal}

In this section, we prove that the noncommutative analogues of the operads of little $2$-disks, framed little $2$-disks, and Deligne--Mumford spaces are formal. 

\subsection{Formality of the operad \texorpdfstring{$\As_{S^1}$}{AsS1}}\label{sec:ncLD}

In this section, we shall demonstrate that the operad $\As_{S^1}$ is formal over integers. We shall first establish a general formality result for the operads~$\As_M$ over rationals, and then indicate how the proof needs to be altered to prove formality of $\As_{S^1}$ over integers.

\begin{proposition}
For every topological space $M$, the topological operad $\As_M$ is formal over $\mathbb{Q}$. In fact, there exists a quasi-isomorphism of dg ns operads over $\mathbb{Q}$: 
\[C^{\mathrm{sing}}_\bullet(\As_M,\mathbb{Q})  \ \stackrel{\sim}{\longleftarrow} H^{\mathrm{sing}}_\bullet(\As_M,\mathbb{Q})\ ,\]
where $C^{\mathrm{sing}}$ stands for the singular chain functor. 
\end{proposition}

\begin{proof}
For brevity, we denote by $C_\bullet$ and $H_\bullet$ singular chains and singular homology over $\mathbb{Q}$.
Since we work over a field, we can make, once and for all, some choices of representatives for  the homology classes of $M$, obtaining a quasi-isomorphism of chain complexes $\varphi_1 : H_\bullet(M) \to C_\bullet(M)$. Recall from \cite[Th~$5.2$]{EilenbergMacLane53} the existence of the cross product map of chain complexes 
$\nabla:=\nabla_{M,M} : C_\bullet(M)^{\otimes 2} \to C_\bullet(M^2)$, which is symmetric and associative. Let us denote by 
$\nabla^{n-2} : C_\bullet(M)^{\otimes (n-1)} \to C_\bullet(M^{n-1})$ the iteration  of the cross product. Also recall that the crossed product induces a symmetric and associative isomorphism on the level of homology 
$H(\nabla) : H_\bullet(M)^{\otimes 2} \stackrel{\cong}{\longrightarrow} H_\bullet(M^2)$; its iteration produces isomorphisms $H(\nabla)^{n-2} : H_\bullet(M)^{\otimes (n-1)} \stackrel{\cong}{\longrightarrow} H_\bullet(M^{n-1})$.

Let us now define the   quasi-isomorphisms 
$\varphi_n \ : \ H_\bullet(\As_M(n)) \qi C_\bullet(\As_M(n))$ of chain complexes by following diagram
\[\xymatrix@R=30pt@C=50pt{C_\bullet(\As_M(n))=C_\bullet(M^{n-1})   &  \ar[l]_(0.42){\nabla^{n-2}} C_\bullet(M)^{\otimes (n-1)}\ \  \\ 
H_\bullet(\As_M(n))=N_\bullet(M^{n-1}) \ar@{..>}[u]^{\varphi_n}  \ar[r]^(0.58){\left(H(\nabla)^{n-2}\right)^{-1}}  &   H_\bullet(M)^{\otimes (n-1)} \ar[u]_(0.55){\varphi_1^{\otimes (n-1)}} \ .}\]

These maps assemble to form a quasi-isomorphism of dg ns operads; for this,  we only have to check the commutativity of the following diagram  \[\xymatrix@R=30pt@C=60pt{
C_\bullet(M^{n_1-1})\otimes C_\bullet(M^{n_2-1})  \ar[r]^(0.52){\nabla_{M^{n_1-1}, \, M^{n_2-1}}} &C_\bullet(M^{n_1-1}\times M^{n_2-1}) 
\ar[r]^(0.53){C_\bullet(\circ_i)} & C_\bullet(M^{n_1+n_2-2})  \ \  \\
H_\bullet(M^{n_1-1})\otimes H_\bullet(M^{n_2-1})  \ar[r]^(0.52){H\left(\nabla_{M^{n_1-1}, \, M^{n_2-1}}\right)} 
\ar[u]^{\varphi_{n_1}\otimes \varphi_{n_2}} &H_\bullet(M^{n_1-1}\times M^{n_2-1}) 
\ar[r]^(0.53){H_\bullet(\circ_i)} & H_\bullet(M^{n_1+n_2-2})  \ar[u]_{\varphi_{n_1+n_2-2}}}  \ .
\]
This is done by first noticing that the operadic composition map $\circ_i$ on the ns operad $\As_M$ only amounts to permuting elements and then, by applying the symmetry property of the cross products~$\nabla$.  
\end{proof}

\begin{corollary}\label{cor:FormAsM}
The topological ns operad $\As_{S^1}$ is formal over $\mathbb{Z}$: there exists a  quasi-isomorphism of dg ns operads over $\mathbb{Z}$: 
\[C^{\mathrm{sing}}_\bullet(\As_{S^1}, \mathbb{Z})  \ \stackrel{\sim}{\longleftarrow}\  H^{\mathrm{sing}}_\bullet(\As_{S^1}, \mathbb{Z})\ .\]
\end{corollary}

\begin{proof}
All the above arguments hold over the rings of integers, especially the ones relative to the cross product map. It remains to consider a suitable quasi-isomorphism $\varphi : H^{\mathrm{sing}}_\bullet(S^1, \mathbb{Z}) \to C^{\mathrm{sing}}_\bullet(S^1, \mathbb{Z})$. There are only two non-trivial homology groups here, one concentrated in degree $0$ and one concentrated in degree $1$. We can represent them by the following simplices:  the $0$-simplex $\1\colon \Delta^0 \to S^1$ with value $1$ 
and  the $1$-simplex $\E\colon \Delta^1 \to S^1$ given by 
$$\theta \in [0,1] \mapsto e^{2\pi i \theta}\in S^1\ .$$
\end{proof}

\subsection{Formality of the operad \texorpdfstring{$\As_{S^1}\rtimes S^1$}{SDAsS1}}\label{sec:ncFLD}

By contrast with the previous section, the structure maps of the operad $\As_{S^1}\ltimes S^1$ mix algebra and topology in a nontrivial way via the group actions, so proving formality for this operad is a much more intricate task. Using mixed Hodge structures, Joana Cirici and Geoffroy Horel proved in \cite{CH17} that the topological ns operad $\As_{S^1}\rtimes S^1$ is formal over $\mathbb{Q}$. In this section, we offer an elementary proof of formality over $\mathbb{R}$ which furnishes a direct quasi-isomorphism (as opposed to a zig-zag of quasi-isomorphisms of~\cite{CH17}).

\begin{theorem}\label{th:ncBVFormality}
The topological ns operad $\As_{S^1}\rtimes S^1$ is formal over the field $\mathbb{R}$ of real numbers; more precisely, there is a  quasi-isomorphism of dg ns operads over $\mathbb{R}$: 
\[C^{\mathrm{sing}}_\bullet\left(\As_{S^1}\rtimes S^1, \mathbb{R} \right)  \stackrel{\sim}{\longrightarrow}  H^{\mathrm{sing}}_\bullet\left(\As_{S^1}\rtimes S^1, \mathbb{R}\right)\ ,\]
where $C^{\mathrm{sing}}$ stands for the singular chain functor. 
\end{theorem}

\begin{proof}
For brevity, we denote by $C_\bullet$ and $H_\bullet$ singular chains and singular homology over $\mathbb{R}$. We shall write $\calO$ for the operad $\As_{S^1}\rtimes S^1$, so that $\calO(n)=(S^1)^{2n-1}$. We keep the notation of the proof of Corollary~\ref{cor:FormAsM}. By a slight abuse of notation, we use the same notation for the homology classes $\1\in H_0(S^1)$ and $\E\in H_1(S^1)$. We shall also consider the chain map  $\psi : C_\bullet(S^1) \to H_\bullet(S^1)$ defined as follows. For any continuous map $f : [0,1] \to S^1$, there exists a continuous map $\theta : [0,1]\to \mathbb{R}$ such that $f(t)=e^{i\theta(t)}$. Now the global variation $\theta(1)-\theta(0)$ of the angle does not depend of the choice of $\theta$ and is intrinsic to the map $f$. We denote such a variation by 
$$\V_{[0,1]} f:=\frac{1}{2\pi}\big(\theta(1)-\theta(0)\big)\ $$
and we extend it by linearity to produce a map from $\C_1(S^1)$ to $\mathbb{R}$. 
The map $\psi : C_\bullet(S^1) \to H_\bullet(S^1)$
is 
defined respectively by 
\begin{eqnarray*}
f\in C_0(S^1) &\mapsto& \1 \in H_0(S^1)\ ,\\
g\in C_1(S^1) &\mapsto& \left(\V_{[0,1]} g\right) \E   \in H_1(S^1)\ ,\\
h\in C_k(S^1) &\mapsto& 0  \in H_k(S^1), \ \text{for} \ k\geqslant 2\ .
\end{eqnarray*}
The only non-trivial point to check in order to prove that this is a chain map is the vanishing of the global angle variation of a $2$-cycle. The image of a $2$-simplex  $f : \Delta^2 \to S^1$ under the boundary map is equal to $f\delta_0+f\delta_2-f\delta_1$. Then,  the image of the first two terms under $\psi$ is 
$\V_{[0,1]} \left(f\delta_0+f\delta_2\right)= \V_{[v_0,v_1]}f+\V_{[v_1,v_2]} f$, where $v_0, v_1, v_2$ denote the vertices of the standard geometric $2$-simplex. The  restriction of the map $f$ to the union 
$[v_0,v_1]\cup [v_1,v_2]$ is homotopic relatively to the boundary to the restriction  of the map $f$ to $[v_0,v_2]$. Therefore $\V_{[v_0,v_1]}f+\V_{[v_1,v_2]} f=\V_{[v_0,v_2]}f=\V_{[0,1]} f\delta_1$, which concludes this point of the proof.
One can automatically notice that $\psi$ is left inverse to $\varphi$, i.e. $\psi \circ\varphi =\id_{H_\bullet(S^1)}$, which implies that chain map $\psi$ is a quasi-isomorphism, for instance. 

From now on, we will identify $H_\bullet((S^1)^n)$ and $H_\bullet(S^1)^{\otimes n}$ using the iteration of the cross product map. We still consider the following types of quasi-isomorphisms $\varphi_n$ : 
\[\xymatrix@R=30pt@C=60pt{C_\bullet\left((S^1)^n\right)   &  \ar[l]_{\nabla^{n-1}} C_\bullet\left(S^1\right)^{\otimes n }\\
H_\bullet\left((S^1)^n\right) \ar@{..>}[u]^{\varphi_n}  \ar[r]_{\left(H(\nabla)^{n-1}\right)^{-1}}^\cong  &   H_\bullet\left(S^1\right)^{\otimes n} \ar[u]_(0.55){\varphi^{\otimes n}} \ .}\]
In the other way round, to define a left inverse to the map $\varphi_n$, we will consider the following symmetrized version of the Alexander--Whitney map. For an $n$-tuple $(i_1, \ldots, i_n)$ of integers such that $i_1+\cdots +i_n=k$, we consider the following maps in the simplex category: 
\begin{eqnarray*}
\iota_j \ : \ [i_j] &\to & [k]\\
l &\mapsto & i_1 +\cdots + i_{j-1} +l\ ,
\end{eqnarray*}
for $1\leqslant j \leqslant n$.
The symmetric ``Alexander--Whitney'' maps, that we need, are defined by 
\begin{eqnarray*}
\mathrm{AW}_n\ : \ C_\bullet(X_1\times \cdots \times X_n) &\to& C_\bullet(X_1)\otimes \cdots \otimes C_\bullet(X_n)\\
(x_1, \ldots, x_n) &\mapsto&
\frac{1}{n!} \sum_{\substack{i_1+\cdots+i_n=k\\ \sigma \in \Sy_n}} (-1)^{\sigma(i_1, \ldots, i_n)} \,
X_1\left(\iota_{\sigma(1)}^\mathrm{op}\right)(x_1)\otimes \cdots \otimes X_n\left(\iota_{\sigma(n)}^\mathrm{op}\right)(x_n)
\ ,
\end{eqnarray*}
where $X_1, \ldots, X_n$ are simplicial sets, where $(x_1, \ldots, x_n)$ is a $k$-simplex of the product $X_1\times \cdots \times X_n$, where $C_\bullet$ stands for the Moore chain complex of simplicial sets,  and where the sign is the one obtained by permuting $n$ elements of respective degrees $i_1,\ldots, i_n$ according to the permutation $\sigma$. 
This map $\mathrm{AW}_n$ can be obtained directly by the classical Alexander--Whitney map $\mathrm{AW}$, see \cite[Section~VIII.8]{MacLane95}, as the following composite 
\[\xymatrix@R=10pt@C=45pt{
C_\bullet(X_1\times \cdots\times X_n) \ar[r]^(0.42){\frac{1}{n!} \sum_{\sigma \in \Sy_n}C_\bullet(\sigma)}
&
\displaystyle \bigoplus_{\sigma\in \Sy_n}C_\bullet\left(X_{\sigma(1)}\times \cdots\times X_{\sigma(n)}\right)
\ar[r]^(0.48){\mathrm{AW}^{n-1}}
&
\displaystyle \bigoplus_{\sigma\in \Sy_n}C_\bullet\left(X_{\sigma(1)}\right) \otimes
 \cdots
 \otimes C_\bullet\left(X_{\sigma(n)}\right)
  \\
\qquad \qquad \qquad \ar[r]^(0.4){\bigoplus_{\sigma_\in\Sy_n}\sigma^{-1}}&\displaystyle \bigoplus_{\sigma\in \Sy_n}C_\bullet\left(X_{1}\right) \otimes
 \cdots
 \otimes C_\bullet\left(X_n\right)
 \ar@{->>}[r]&\displaystyle C_\bullet\left(X_{1}\right) \otimes
 \cdots
 \otimes C_\bullet\left(X_n\right)\ ,\qquad 
&&
}\]
which proves that it is a chain map that factors through the normalised chain complexes. 
(Notice that the classical Alexander--Whitney map is associative but not symmetric, on the opposite to the cross product map. The present ``symmetrized Alexander--Whitney map'' is indeed symmetric but fails to be associative; this will not be an issue  for the present proof since we do not use the coalgebra structure on singular chains or singular homology.) 

We consider the following chain maps $\psi_n:=\psi^{\otimes n}\circ \mathrm{AW}_n$:
\[\xymatrix@R=30pt@C=60pt{C_\bullet((S^1)^n)   \ar[r]^{\mathrm{AW}_n} \ar@{..>}[d]_{\psi_n} &   C_\bullet(S^1)^{\otimes n } \ar[d]^(0.45){\psi^{\otimes n}} \\
H_\bullet((S^1)^n)     & \ar[l]^{H(\nabla)^{n-1}}_\cong H_\bullet\left(S^1\right)^{\otimes n}  \ .}\]
We claim that the chain maps $\psi_{2n-1} : C_\bullet\left(\calO(n)\right) \to H_\bullet\left(\calO(n)\right)$ make up into a quasi-isomorphism of dg ns operads. 

Since the cross product map is symmetric and since the composite with the Alexander--Whitney map is equal to the identity on the normalised chain complex, 
$\mathrm{AW}\circ \nabla=\id_{N_\bullet(S^1)^{\otimes 2}}$\ , the same relation holds true for the present maps, that is 
$$\mathrm{AW}_n\circ \nabla^{n-1}=\id_{N_\bullet(S^1)^{\otimes n}}$$ on the level of the normalised chain complex. Since $\varphi$ and $\psi$ are both quasi-isomorphisms, this proves first  that the maps $\psi_{2n-1}$ are quasi-isomorphisms. Moreover with the relation $\psi \circ\varphi =\id_{H_\bullet(S^1)}$, this  implies also the relation $\psi_n \circ\varphi_n =\id_{H_\bullet((S^1)^n)}$.

Let us now prove that the maps $\psi_{2n-1} : C_\bullet\left(\calO(n)\right) \to H_\bullet\left(\calO(n)\right)$ preserve the  operadic partial compositions, that is, for any $1\leqslant i \leqslant n$: 
\begin{equation}\label{eq:Compo}
\xymatrix@R=25pt@C=35pt{
C_k\left(\calO(n)\right) \otimes C_l\left(\calO(m)\right) \ar[r]^(0.55){\nabla} 
\ar[d]^{\psi_{2n-1}\otimes \psi_{2m-1}}& 
C_{k+l}\left(\calO(n)\times \calO(m)\right)
\ar[r]^{C_\bullet(\circ_i)} & 
C_{k+l}\left(\calO(n+m-1)\right) \ar[d]^{\psi_{2(n+m-1)-1}} \\
H_k\left(\calO(n)\right) \otimes H_l\left(\calO(m)\right) \ar[r]^(0.53){H(\nabla)} & 
H_{k+l}\left(\calO(n)\times \calO(m)\right)
\ar[r]^{H_\bullet(\circ_i)} & 
H_{k+l}\left(\calO(n+m-1)\right)\ .
}\tag{*}
\end{equation}
We check the commutativity of this diagram on an element $(f^1, \ldots, f^{2n-1}):  \Delta^k \to (S^1)^{2n-1}$ of $C_k\left(\calO(n)\right)$ and on  
an element $(g^1, \ldots, g^{2m-1}):  \Delta^l \to (S^1)^{2m-1}$ of $C_l\left(\calO(m)\right)$.

For $k>2n-1$, the image of the former element under the map $\psi_{2n-1}$ is equal to $0$ and, similarly, 
for $l>2m-1$, the image of the latter element under the map $\psi_{2m-1}$ is equal to $0$. To describe the other images, we will use the following notations. We denote by 
$$I_1^n:=[v_0,v_1], \ I_2^n:=[v_1,v_2],\  \ldots, \ I_n^n:=[v_{n-1},v_n] $$
the principal edges of the standard geometric $n$-simplex, where $v_i$ stands for the $i^{\mathrm{th}}$ vertex. 
For any continuous map $f : \Delta^n \to S^1$ and for any principal edge $I=[v_i, v_{i+1}]$ of $\Delta^n$, 
we denote simply by $\V_I f$ the angle variation of the restriction $f_{|I}$ viewed as a map from $[0,1]$ to $S^1$, so that
the image of  ${f}_{|I}$ under the map $\psi$ is equal to $\psi\left({f}_{|I}\right)=\left(\V_I f\right) \E$. 
To any subset $\Phi=\left\{\Phi_1 < \cdots < \Phi_k\right\}$ of $\underline{2n-1}$, we associate the element 
$$\E^\Phi:=\1\otimes \cdots \otimes \1\otimes \underbrace{\E}_{\text{position}\ \Phi_1} \otimes \1\otimes \cdots \otimes \1 \otimes 
\underbrace{\E}_{\text{position}\  \Phi_k}
\otimes \1 \otimes \cdots \otimes \1  \ \in  \left(H_\bullet(S^1)^{\otimes 2n-1}\right)_k\ .$$
We consider the assignment  which send a subset $\Phi=\left\{\Phi_1 < \cdots < \Phi_k\right\}\subset \underline{2n-1}$ to the $(2n-1)$-tuple $(i_1, \ldots, i_{2n-1})$ defined by $i_j:=0$, if $j\notin \Phi$, and by $i_j:=1$, if $j\in \Phi$. This defines a bijection, denoted by  $\mathrm{S}(\Phi):=(i_1, \ldots, i_{2n-1})$, between the subset of $\underline{2n-1}$ with $k$ elements and $(2n-1)$-tuples made up of $k$ times $1$ and $2n-k-1$ times $0$. (Since the indices $k$ and $n$ will be obvious and might change, we will use the generic notation $\mathrm{S}$.) 
We denote the signature of a permutation $\sigma$ by $(-1)^\sigma$. 

For $k\leqslant 2n-1$, since the map $\psi$ vanishes outside $0$-simplices and $1$-simplices, only the terms labelled by $(i_1, \ldots, i_{2n-1})$ such that $i_1+\cdots+i_{2n-1}=k$ and $i_j=0$ or $i_j=1$, might contribute to the image of $(f^1, \ldots, f^{2n-1})$ under the map $\psi_{2n-1}=\psi^{\otimes 2n-1}\circ \mathrm{AW}_{2n-1}$.
Since the map $\psi$ is constant, equal to $\1$, on $0$-simplices, we actually get the following formulae for the lower-left corner of \eqref{eq:Compo}:
$$ 
\psi_{2n-1}\left(f^1, \ldots, f^{2n-1}\right)=\frac{1}{k!}\sum_{\substack{\Phi\subset \underline{2n-1} \\ |\Phi|=k}} 
\left(\sum_{\sigma\in \Sy_k}
(-1)^\sigma \V_{I^k_{\sigma(1)}} f^{\Phi_1}\ \cdots\  \V_{I^k_{\sigma(k)}} f^{\Phi_k}\right)
 \E^\Phi
$$
and similarly
$$ 
\psi_{2m-1}\left(g^1, \ldots, g^{2m-1}\right)=\frac{1}{l!}\sum_{\substack{\Psi\subset \underline{2m-1} \\ |\Psi|=l}} 
\left(\sum_{\tau\in \Sy_l}
(-1)^\tau \V_{I^l_{\tau(1)}} g^{\Psi_1}\ \cdots\  \V_{I^l_{\tau(l)}} g^{\Psi_l}\right)
 \E^\Psi \ ,
$$
for $l\leqslant 2m-1$.

Given $\Phi\subset \underline{2n-1}$ such that $|\Phi|=k$ and $\Psi\subset \underline{2m-1}$ such that $|\Psi|=l$, we denote their respective image under $\mathrm{S}$ by $(h_1, \ldots, h_{2n-1}):=\mathrm{S}(\Phi)$ and 
$(j_1, \ldots, j_{2m-1}):=\mathrm{S}(\Psi)$. When ${n+i-1}\notin \Phi$, we use the following notation: 
\[
\mathrm{S}(\Phi)\circ_i \mathrm{S}(\Psi):=\left(
h_1, \ldots, h_{i-1}, j_1, \ldots, j_{m-1},h_i, \ldots, h_{n-1}, h_n, \ldots, h_{n+i-2}, j_m, \ldots, j_{2m-1}, h_{n+i}, \ldots, h_{2n-1}
\right)\ .
\]

For the lower-right corner of \eqref{eq:Compo}, we need to distinguish three cases:  
any subset $\Upsilon \subset \underline{2n+2m-3}$ of cardinal $|\Upsilon|=k+l$ is of the following form:
\begin{enumerate}
\item[(i)] $\Upsilon = \mathrm{S}^{-1}\left(\mathrm{S}(\Phi)\circ_i\mathrm{S}(\Psi)\right)$\ , when ${n+i-1}\notin \Phi$\ ,
\item[(ii)] $\Upsilon = \mathrm{S}^{-1}\left(\mathrm{S}(\Phi\setminus{\{n+i-1\}})\circ_i\mathrm{S}(\Psi\sqcup\{j\})\right)$\ , for $j\in \underline{2m-1}$ and $j\notin \Psi$\ ,\ , when ${n+i-1}\in \Phi$\ , 
\item[(iii)] not of these types.
\end{enumerate}
In Case~(i), the coefficient of $\E^\Upsilon$ is equal to 
\begin{equation}\label{eq:i}\tag{i}
(-1)^{(k''+k''')l'+k'''l''}\frac{1}{k!l!}
\left(\sum_{\sigma\in \Sy_k}
(-1)^\sigma \V_{I^k_{\sigma(1)}} f^{\Phi_1}\ \cdots\  \V_{I^k_{\sigma(k)}} f^{\Phi_k}\right)
\left(\sum_{\tau\in \Sy_l}
(-1)^\tau \V_{I^l_{\tau(1)}} g^{\Psi_1}\ \cdots\  \V_{I^l_{\tau(l)}} g^{\Psi_l}\right)\ ,
\end{equation}
for $k=k'+k''+k'''$, 
where $k'$ is the number of elements of $\Phi$ strictly  less than $i$, 
where $k'$ is the number of elements of $\Phi$   between $i$ and $n+i-2$, and
where $k'''$ is the number of elements of $\Phi$ strictly  greater than $n+i-1$, 
and for $l=l'+l''$, 
where $l'$ is the number of elements of $\Psi$ strictly  less than $m$, and
where $l''$ is the number of elements of $\Psi$  greater than $m$.
 
In Case~(ii), there are two sub-cases. In sub-Case~(a), the element $j$ added to $\Psi$ lives in 
$1\leqslant j\leqslant m-1$ and 
the coefficient of $\E^\Upsilon$ is equal to 
\begin{equation}\label{eq:iia}\tag{iia}
(-1)^{k'''+|\Psi<j|+(k''+k''')(l'+1)+k'''l''}\frac{1}{k!l!}
\left(\sum_{\sigma\in \Sy_k}
(-1)^\sigma \V_{I^k_{\sigma(1)}} f^{\Phi_1}\ \cdots\  \V_{I^k_{\sigma(k)}} f^{\Phi_k}\right)
\left(\sum_{\tau\in \Sy_l}
(-1)^\tau \V_{I^l_{\tau(1)}} g^{\Psi_1}\ \cdots\  \V_{I^l_{\tau(l)}} g^{\Psi_l}\right)\ ,
\end{equation}
where the same kind of decomposition as above $k-1=k'+k''+k'''$ and $l=l'+l''$ and where 
$|\Psi<j|$ denotes the number of elements of $\Psi$ less than $j$. In sub-Case~(b), the element $j$ added to $\Psi$ lives in 
$m\leqslant j\leqslant 2m-1$ and 
the coefficient of $\E^\Upsilon$ is equal to 
\begin{equation}\label{eq:iib}\tag{iib}
(-1)^{k'''+|\Psi<j|+(k''+k''')l'+k'''(l''+1)}\frac{1}{k!l!}
\left(\sum_{\sigma\in \Sy_k}
(-1)^\sigma \V_{I^k_{\sigma(1)}} f^{\Phi_1}\ \cdots\  \V_{I^k_{\sigma(k)}} f^{\Phi_k}\right)
\left(\sum_{\tau\in \Sy_l}
(-1)^\tau \V_{I^l_{\tau(1)}} g^{\Psi_1}\ \cdots\  \V_{I^l_{\tau(l)}} g^{\Psi_l}\right)\ ,
\end{equation}

Finally, in Case~(iii), there is no element of the form $\E^\Upsilon$.

For the element situated in the upper-right corner of \eqref{eq:Compo}, we need to recall the formula for the cross product map \cite[Theorem~VIII.8.8]{MacLane95}. To any subset $S=\{S_1, \ldots, S_k\}\subset \underline{k+l}$, one associates the $(k,l)$-shuffle $\shuffle(S)$ which is the permutation of $\Sy_{k+l}$ sending $j\in \underline{k}$ to $S_j$ and   $j\in \underline{k+l}\setminus\underline{k}$ to the remaining elements in an order-preserving way. The simplicial map of simplicial complexes $\sigma_S : \Delta^{k+l} \to \Delta^l$ is the one which contracts  all the principal edges $I^{k+l}_s$, for $s\in S$. Similarly, we consider the map of simplicial complexes $\sigma_{S^c} : \Delta^{k+l} \to \Delta^k$ which contracts the remaining principal edges. Under this notation, the  element situated in the upper-right corner of \eqref{eq:Compo} is equal to 
\begin{eqnarray*}
&&\sum_{\substack{S\subset \underline{k+l}\\ |S|=k}}
(-1)^{\shuffle(S)}
\big(
f^1\sigma_{S^c}, \ldots, f^{i-1}\sigma_{S^c}, 
\big(f^{n+i-1}\sigma_{S^c}\big)\big(g^1\sigma_S\big), \ldots, 
\big(f^{n+i-1}\sigma_{S^c}\big)\big(g^{m-1}\sigma_S\big), \\
&&
f^{i}\sigma_{S^c}, \ldots, f^{n+i-2}\sigma_{S^c}, 
\big(f^{n+i-1}\sigma_{S^c}\big)\big(g^m\sigma_S\big), \ldots, 
\big(f^{n+i-1}\sigma_{S^c}\big)\big(g^{2m-1}\sigma_S\big), 
f^{n+i}\sigma_{S^c}, \ldots, f^{2n-1}\sigma_{S^c}
\big)\ .
\end{eqnarray*}

If we denote simply by $h^1_S, \ldots, h^{2n+2m-3}_S : \Delta^{k+l} \to S^1$ the above list of maps, we get the following formula for the lower-right term of \eqref{eq:Compo}:
$$
\frac{1}{(k+l)!}
\sum_{\substack{\Upsilon\subset \underline{2n+2m-3} \\ |\Upsilon|=k+l}} \left(
\sum_{\substack{S\subset \underline{k+l}\\ |S|=k}}
\sum_{\omega\in \Sy_{k+l}}
(-1)^{\shuffle(S)\omega} \V_{I^{k+l}_{\omega(1)}} h_S^{\Upsilon_1}\ \cdots\  \V_{I^{k+l}_{\omega(k+l)}} h_S^{\Upsilon_{k+l}}\right)
 \E^\Upsilon\ .
 $$
It remains to show that in the above three cases, the coefficients agree. Let us first notice the following facts. First of all, we have $\V_I fg =\V_I f +\V_I g$: we write $f(t)=e^{i\theta(t)}$ and $g(t)=e^{i\rho(t)}$, so that $(fg)(t)=e^{i(\theta(t)+\rho(t))}$ and thus
\[\V_I fg =
\frac{1}{2\pi}\big(\theta(1)+\rho(1)-\theta(0)-\rho(0)\big)
=\V_I f +\V_I g\ .\]
Secondly, when $i\in S$, the variation $\V_{I^{k+l}_i} f\sigma_S=0$ vanishes, since the restriction $(f\sigma_S)_{|I^{k+l}_i}$ is constant, and when $i \notin S$, the variation is equal to 
$$\V_{I^{k+l}_i} f\sigma_S=\V_{I^l_{i-|S<i|}} f\ ,$$
where $|S<i|$ is the number of elements of $S$ strictly less than $i$.

In Case~(i), the coefficient of $\E^\Upsilon$ is equal to 
\begin{eqnarray*}
&&\frac{1}{(k+l)!}
\sum_{\substack{S\subset \underline{k+l}\\ |S|=k}}
\sum_{\omega\in \Sy_{k+l}}
(-1)^{\shuffle(S)\omega} 
\V_{I^{k+l}_{\omega(1)}} f^{\Phi_1}\sigma_{S^c} \ldots \V_{I^{k+l}_{\omega(k')}} f^{\Phi_{k'}}\sigma_{S^c}
\V_{I^{k+l}_{\omega(k'+1)}}\left(f^{n+i-1}\sigma_{S^c}\right)\left(g^{\Psi_1}\sigma_S\right) \ldots \\&&
\V_{I^{k+l}_{\omega(k'+l')}}\left(f^{n+i-1}\sigma_{S^c}\right)\left(g^{\Psi_{l'}}\sigma_S\right)
\V_{I^{k+l}_{\omega(k'+l'+1)}} f^{\Phi_{k'+1}}\sigma_{S^c} \ldots \V_{I^{k+l}_{\omega(k'+l'+k'')}} f^{\Phi_{k'+k''}}\sigma_{S^c} \\&&
\V_{I^{k+l}_{\omega(k'+l'+k''+1)}}\left(f^{n+i-1}\sigma_{S^c}\right)\left(g^{\Psi_{l'+1}}\sigma_S\right)\ldots
\V_{I^{k+l}_{\omega(k'+l'+k''+l'')}}\left(f^{n+i-1}\sigma_{S^c}\right)\left(g^{\Psi_{l'+l''}}\sigma_S\right)
\V_{I^{k+l}_{\omega(k'+l'+k''+l''+1)}} f^{\Phi_{k'+k''+1}}\sigma_{S^c} \\
&& \ldots\V_{I^{k+l}_{\omega(k'+l'+k''+l''+k''')}} f^{\Phi_{k'+k''+k'''}}\sigma_{S^c}\ .
\end{eqnarray*}
For any given $S\subset \underline{k+l}$, the only non-vanishing terms come when the permutation $\omega$ sends the $\{1, \ldots, k', k'+l'+1, \ldots, k'+l'+k'', k'+l'+k''+l''+1, \ldots, k'+l'+k''+l''+k'''\}$ to the $k$  elements of $S$. This happens if and only if the permutation $\omega$ is equal to the following type of composite:  
$\omega=\shuffle(S)(\sigma\times\tau)\beta$ where $\beta$ is the block-permutation which sends 
$\{1, \ldots, k', k'+l'+1, \ldots, k'+l'+k'', k'+l'+k''+l''+1, \ldots, k'+l'+k''+l''+k'''\}$ to 
$\{1, \ldots, k\}$ and 
$\{k'+1, \ldots, k'+l', k'+l'+k'', \ldots,k'+l'+k''+l'' \}$
to $\{k+1, \ldots, k+l\}$, 
 and where $\sigma\in \Sy_k$, $\tau\in \Sy_l$. 
Since there are $\binom{k+l}{k}$ ways to choose the subsets $S\subset \underline{k+l}$ and since the signature of the permutation $\beta$ is $(-1)^{(k''+k''')l'+k'''l''}$ and that of $\sigma\times\tau$ is equal to $(-1)^\sigma(-1)^\tau$, this term is equal to 
\begin{equation*}
(-1)^{(k''+k''')l'+k'''l''}\frac{1}{k!l!}
\sum_{\substack{\sigma\in \Sy_k\\ \tau\in \Sy_l}}
(-1)^\sigma(-1)^\tau \V_{I^k_{\sigma(1)}} f^{\Phi_1}\ \cdots\  \V_{I^k_{\sigma(k)}} f^{\Phi_k}
 \V_{I^l_{\tau(1)}} g^{\Psi_1}\ \cdots\  \V_{I^l_{\tau(l)}} g^{\Psi_l}\ ,
\end{equation*}
that is term~\eqref{eq:i}.

In Case~(ii) and sub-Case~(a), the coefficient of $\E^\Upsilon$ is equal to
\begin{eqnarray*}
&&\frac{1}{(k+l)!}
\sum_{\substack{S\subset \underline{k+l}\\ |S|=k}}
\sum_{\omega\in \Sy_{k+l}}
(-1)^{\shuffle(S)\omega} 
\V_{I^{k+l}_{\omega(1)}} f^{\Phi_1}\sigma_{S^c} \ldots \V_{I^{k+l}_{\omega(k')}} f^{\Phi_{k'}}\sigma_{S^c}
\V_{I^{k+l}_{\omega(k'+1)}}\left(f^{n+i-1}\sigma_{S^c}\right)\left(g^{\Psi_1}\sigma_S\right) \ldots \\&&
\ldots\V_{I^{k+l}_{\omega(k'+|\Psi<j|+1)}}\left(f^{n+i-1}\sigma_{S^c}\right)\left(g^{j}\sigma_S\right) \ldots  
\V_{I^{k+l}_{\omega(k'+l'+1)}}\left(f^{n+i-1}\sigma_{S^c}\right)\left(g^{\Psi_{l'}}\sigma_S\right)
\V_{I^{k+l}_{\omega(k'+l'+2)}} f^{\Phi_{k'+1}}\sigma_{S^c} \ldots \\&&
\V_{I^{k+l}_{\omega(k'+l'+k''+1)}} f^{\Phi_{k'+k''}}\sigma_{S^c} 
\V_{I^{k+l}_{\omega(k'+l'+k''+2)}}\left(f^{n+i-1}\sigma_{S^c}\right)\left(g^{\Psi_{l'+1}}\sigma_S\right)\ldots\\&&
\V_{I^{k+l}_{\omega(k'+l'+k''+l''+1)}}\left(f^{n+i-1}\sigma_{S^c}\right)\left(g^{\Psi_{l'+l''}}\sigma_S\right)
\V_{I^{k+l}_{\omega(k'+l'+k''+l''+2)}} f^{\Phi_{k'+k''+1}}\sigma_{S^c}  \ldots \\&&\V_{I^{k+l}_{\omega(k'+l'+k''+l''+k'''+1)}} f^{\Phi_{k'+k''+k'''}}\sigma_{S^c}\ .
\end{eqnarray*}
For any given $S\subset \underline{k+l}$, the only non-vanishing terms come when the permutation $\omega$  is equal to the following type of composite:  
$\omega=\shuffle(S)(\sigma\times\tau)\gamma\beta$, where $\beta$ is the same kind of block-permutation as above but associated to the decomposition $k+l=k'+(l'+1)+k''+l''+k'''$ and where $\gamma$ is the permutation 
 which puts $k+|\Psi<j|+1$ at place $k'+k''+1$ and leaves the other elements in the order.
Since the signature of this latter  permutation  is equal to  $(-1)^{k'''+|\Psi<j|}$, this term is equal to 
\begin{equation*}
(-1)^{k'''+|\Psi<j|+(k''+k''')(l'+1)+k'''l''}\frac{1}{k!l!}
\sum_{\substack{\sigma\in \Sy_k\\ \tau\in \Sy_l}}
(-1)^\sigma(-1)^\tau \V_{I^k_{\sigma(1)}} f^{\Phi_1}\ \cdots\  \V_{I^k_{\sigma(k)}} f^{\Phi_k}
 \V_{I^l_{\tau(1)}} g^{\Psi_1}\ \cdots\  \V_{I^l_{\tau(l)}} g^{\Psi_l}\ ,
\end{equation*}
that is term~\eqref{eq:iia}.
In sub-Case~(b), the coefficient of $\E^\Upsilon$ is equal to
\begin{eqnarray*}
&&\frac{1}{(k+l)!}
\sum_{\substack{S\subset \underline{k+l}\\ |S|=k}}
\sum_{\omega\in \Sy_{k+l}}
(-1)^{\shuffle(S)\omega} 
\V_{I^{k+l}_{\omega(1)}} f^{\Phi_1}\sigma_{S^c} \ldots \V_{I^{k+l}_{\omega(k')}} f^{\Phi_{k'}}\sigma_{S^c}
\V_{I^{k+l}_{\omega(k'+1)}}\left(f^{n+i-1}\sigma_{S^c}\right)\left(g^{\Psi_1}\sigma_S\right) \ldots \\&&  
\V_{I^{k+l}_{\omega(k'+l')}}\left(f^{n+i-1}\sigma_{S^c}\right)\left(g^{\Psi_{l'}}\sigma_S\right)
\V_{I^{k+l}_{\omega(k'+l'+1)}} f^{\Phi_{k'+1}}\sigma_{S^c} \ldots \V_{I^{k+l}_{\omega(k'+l'+k'')}} f^{\Phi_{k'+k''}}\sigma_{S^c}\\&&
\V_{I^{k+l}_{\omega(k'+l'+k''+1)}}\left(f^{n+i-1}\sigma_{S^c}\right)\left(g^{\Psi_{l'+1}}\sigma_S\right)
\ldots\V_{I^{k+l}_{\omega(k'+k''+|\Psi<j|+1)}}\left(f^{n+i-1}\sigma_{S^c}\right)\left(g^{j}\sigma_S\right) \ldots\\&&
\V_{I^{k+l}_{\omega(k'+l'+k''+l''+1)}}\left(f^{n+i-1}\sigma_{S^c}\right)\left(g^{\Psi_{l'+l''}}\sigma_S\right)
\V_{I^{k+l}_{\omega(k'+l'+k''+l''+2)}} f^{\Phi_{k'+k''+1}}\sigma_{S^c}  \ldots\V_{I^{k+l}_{\omega(k'+l'+k''+l''+k'''+1)}} f^{\Phi_{k'+k''+k'''}}\sigma_{S^c}\ .
\end{eqnarray*}
For any given $S\subset \underline{k+l}$, the only non-vanishing terms come when the permutation $\omega$  is equal to the following type of composite:  
$\omega=\shuffle(S)(\sigma\times\tau)\gamma\beta$, where $\beta$ is the same kind of block-permutation as above but associated to the decomposition $k+l=k'+l'+k''+(l''+1)+k'''$ and where $\gamma$ is the same permutation as above.
Therefore, this term is equal to 
\begin{equation*}
(-1)^{k'''+|\Psi<j|+(k''+k''')l'+k'''(l''+1)}\frac{1}{k!l!}
\sum_{\substack{\sigma\in \Sy_k\\ \tau\in \Sy_l}}
(-1)^\sigma(-1)^\tau \V_{I^k_{\sigma(1)}} f^{\Phi_1}\ \cdots\  \V_{I^k_{\sigma(k)}} f^{\Phi_k}
 \V_{I^l_{\tau(1)}} g^{\Psi_1}\ \cdots\  \V_{I^l_{\tau(l)}} g^{\Psi_l}\ ,
\end{equation*}
that is term~\eqref{eq:iib}.

In Case~(iii), there are two subcases. First, when 
\[|\Upsilon\cap\{1, \ldots, i-1, i+m-1,\ldots, i+m+n-3, i+2m+n-2, \ldots, 2n+2m-3\}|\geqslant k+1\ ,\]
 there is at least one variation $\V_{I^{k+l}_b} f^a\sigma_{S^c}$ involved which is equal to $0$ for any $S\subset\underline{k+l}$, since there is at least an element $b\in \Upsilon$ such that $b\notin S$. Then, in the other case, we have 
 \[|\Upsilon\cap\{i, \ldots, i+m-2, i+m+n-2, \ldots, i+2m+n-3\}|=l+N\ ,\] with $2\leqslant N\leqslant k$. In this case, the coefficient of the term $\E^\Upsilon$ factors through 
$$ 
\sum_{\substack{T\subset \underline{k} \\ |T|=N}} \sum_{\theta\in \Sy_N} 
(-1)^\theta
\V_{I^k_{T_{\theta(1)}}} f^{n+i-1}\cdots \V_{I^k_{T_{\theta(N)}}} f^{n+i-1}=
\left(\sum_{\theta\in \Sy_N} 
(-1)^\theta\right)\sum_{\substack{T\subset \underline{k} \\ |T|=N}} 
\V_{I^k_{T_{1}}} f^{n+i-1}\cdots \V_{I^k_{T_{N}}} f^{n+i-1}=0
\ ,$$
which concludes the proof.
\end{proof}

\subsection{Formality of the brick operad \texorpdfstring{$\calB_{\mathbb{C}}$}{BComp}}

In the same way as the classical operad  $\{\overline{\calM}_{0,n+1}\}$ of Deligne--Mumford compactifications can be shown to be formal~\cite{GSNPR05}, the ns brick operad is formal.

\begin{proposition}
The ns brick operad $\calB_{\mathbb{C}}$ is formal over the rational numbers, that is there is a zigzag of quasi-isomorphisms of dg ns operads over $\mathbb{Q}$: 
$$C_\bullet(\calB_{\mathbb{C}})  \ \stackrel{\sim}{\longleftarrow} \ \cdots \ \stackrel{\sim}{\longrightarrow}\  H_\bullet(\calB_{\mathbb{C}}) \cong \ncHyperCom\ .$$ 
\end{proposition}

\begin{proof}
First, since the brick manifolds are the toric varieties associated to Loday realisations of the associahedra, and since the Loday polytopes are Delzant, \cite[Theorem~2.5]{Guillemin94} implies that brick manifolds are K\"ahler manifolds. Then, one can apply the general result \cite[Corollary~6.3.1]{GSNPR05} to conclude that the ns brick operad is formal over $\mathbb{Q}$. 
\end{proof}

\section{Noncommutative \texorpdfstring{$\overline{\calM}_{0,n+1}(\mathbb{R})$}{M0nR}}\label{sec:RealCase}

Let us note that the ns operad structure of the brick operad (Section~\ref{subsec:brick}) is defined over any field. In particular, it makes sense to consider it over the field $\mathbb{R}$, in which case they form an ns operad in the category of smooth real algebraic varieties. If one recalls that the brick manifold $\calB(\underline{n})$ is isomorphic to the toric varieties of the Loday polytope $L_n$, one can use known results on rational Betti numbers of the toric varieties $X_{\mathbb{R}}(L_n)$ to derive the following result.

\begin{proposition}
The rational Betti numbers of the real brick manifolds are given by the formula
	\[
b_i\big(\calB_{\mathbb{R}}(\underline{n})\big)=
	\begin{cases}
	1, \text{ if } \ i=0,\\
	\binom{n-1}{i}-\binom{n-1}{i-1}, \text{ if } \ 1\le i\le \lfloor (n-1)/2\rfloor,\\
	0, \text{ if } \ i>\lfloor (n-1)/2\rfloor.
	\end{cases} 
	\]  
The Euler characteristics of  $\calB_{\mathbb{R}}(\underline{n})$ is equal to $0$ for odd~$n$, and to $(-1)^{\frac{n}{2}-1}C_{\frac{n}{2}-1}$ for even~$n$. Here $C_k$ is the $k^\text{th}$ Catalan number $C_k=\frac{1}{k+1}\binom{2k}{k}$.
\end{proposition}

\begin{proof}
Since $\calB_{\mathbb{R}}(\underline{n})\cong X_\mathbb{R}(L_n)$, this follows directly from \cite[Cor.~1.4]{ChoiPark15}.
\end{proof}

In this section, we proceed much further than computing the Betti numbers, and describe the ns operad structure on the rational homology of the real brick operad, which provides us with a ``noncommutative counterpart'' of the $2$-Gerstenhaber operad~\cite{EHKR10}. We remark that the cohomological results of \cite{EHKR10} do not admit a proper noncommutative analogue: similarly to the statement  for real toric varieties of permutahedra in~\cite{Hen12}, it can be shown that the cohomology algebras of real toric varieties of Loday polytopes are not generated by elements of degree one, in contrast to the case of $\overline{\calM}_{0,n+1}(\mathbb{R})$.

\subsection{Some relations in the rational homology of the real brick operad}

In this section, we explain geometrically some relations in the rational homology of the real brick operad. Let $m\in H_0\big(\calB_{\mathbb{R}}(\underline{2})\big)$ be the class of the point $pt=\calB_{\mathbb{R}}(\underline{2})$, and $c\in H_1\big(\calB_{\mathbb{R}}(\underline{3})\big)$ be the fundamental class of the circle $S^1=\calB_{\mathbb{R}}(\underline{3})$. 

\begin{proposition}\label{prop:rel-real-brick}
We have 
\begin{gather*}
m\circ_1m-m\circ_2m=0,\\
c\circ_1m=m\circ_2c,  \quad c\circ_2m=0,\quad c\circ_3m=m\circ_1c, \\
c\circ_1 c + c\circ_2 c + c\circ_3 c = 0 .
\end{gather*}
\end{proposition}

\begin{proof} The relation of arity $3$ is obvious: the compositions $m\circ_1m$ and $m\circ_2m$ both represent the only zero-dimensional class, so we have $m\circ_1m=m\circ_2m$ in homology.
	
Let us establish the relations of arity $4$. For that, it is convenient to view the brick manifold $\calB(\underline{4})$ as the blow-up of $\mathbb{P}^1\times\mathbb{P}^1$ at a point. For the set of real points $\calB_{\mathbb{R}}(\underline{4})$, this means that this manifold is obtained from the torus $\mathbb{T}^2$ by cutting out a small $2$-disk and gluing in the M\"obius strip $\mathbb{M}^2$ along its boundary. On the level of homology, the elements $\alpha_1$ and $\alpha_2$ arising from the generators of the first integral homology group of the torus correspond to the choices of $V_{2,2}\subset G(1,3)$ and $V_{3,3}\subset G(2,4)$, and the $2$-torsion element $\beta$ arising from the generator of the first integral homology group of the M\"obius strip corresponds to the choice of $V_{2,3}/G(2,3)\subset\myspan(e_{1,2},e_{3,4})$ when $V_{2,2}=V_{3,3}=G(2,3)$. Next, we notice that our definition of the operadic composition in the brick operad immediately implies that the element $c\circ_2m$ is represented by the circle
	\[
\{V_{22}=G(2,3), V_{33}=G(2,3), V_{23}=\myspan(\sin(t)e_{1,2}+\cos(t)e_{3,4},e_{2,3}) \mid t\in[0,2\pi) \} 
	\]
and hence corresponds to the class $\beta$, so $c\circ_2m$ vanishes in homology with real or rational coefficients. Furthermore, the elements $m\circ_1c$ and $c\circ_3m$
are represented by the circles
	\[
\{V_{22}=\myspan(\sin(t)e_{1,2}+\cos(t)e_{2,3}), V_{33}=G(2,3), V_{23}=G(1,3) \mid t\in[0,2\pi)\} ,  
	\]
and 
	\[
\{V_{22}=\myspan(\sin(t)e_{1,2}+\cos(t)e_{2,3}), V_{33}=G(3,4), V_{23}=V_{22}\oplus V_{33} \mid t\in[0,2\pi) \} ,
	\]
so in the rational homology $m\circ_1c=\alpha_1=c\circ_3m$, and 
the elements $c\circ_1m$ and $m\circ_2c$ are represented by the circles
	\[
\{V_{22}=G(1,2), V_{33}=\myspan(\sin(t)e_{2,3}+\cos(t)e_{3,4}), V_{23}=V_{22}\oplus V_{33} \mid t\in[0,2\pi) \} , 
	\]
and
	\[
\{V_{22}=G(2,3), V_{33}=\myspan(\sin(t)e_{2,3}+\cos(t)e_{3,4}), V_{23}=G(2,4) \mid t\in[0,2\pi)\} .  
	\]
so in the rational homology $c\circ_1m=\alpha_2=m\circ_2c$.

Finally, let us consider the relation of arity $5$. For that, it is more convenient to switch to the language of wonderful models, and use the construction of $\calB_{\mathbb{R}}(\underline{n})\cong\widehat{Y}_{ncA_{n-1},\min}(\mathbb{R})$ via iterated blow-ups of $\mathbb{P}(V)$, as in Remark~\ref{rem:IteratedBlowup}. For $n=5$, the space $\mathbb{P}(V)$ is isomorphic to $\mathbb{RP}^3$; it contains a configuration of three projective lines $\mathbb{P}(p(H_{\{1,2,3\}}^\bot))$, $\mathbb{P}(p(H_{\{2,3,4\}}^\bot))$, and $\mathbb{P}(p(H_{\{3,4,5\}}^\bot))$. 
	
Our construction suggests that we have to blow up $\mathbb{RP}^3$ first at two points, $\mathbb{P}(p(H_{\{1,2,3,4\}}))=\mathbb{P}(p(H_{\{1,2,3\}}^\bot))\cap
\mathbb{P}(p(H_{\{2,3,4\}}^\bot))$ and $\mathbb{P}(p(H_{\{2,3,4,5\}}))=\mathbb{P}(p(H_{\{2,3,4\}}^\bot))\cap \mathbb{P}(p(H_{\{3,4,5\}}^\bot))$  
and then along the proper transforms of the three projective lines $\mathbb{P}(p(H_{\{1,2,3\}}^\bot))$, $\mathbb{P}(p(H_{\{2,3,4\}}^\bot))$, and $\mathbb{P}(p(H_{\{3,4,5\}}^\bot))$. The fundamental classes of the three exceptional divisors of the second blow-up represent $c\circ_1 c$, $c\circ_2 c$, and $c\circ_3 c$. 
	
Consider the ``open cell'' 
 \[
 \mathsf{C}:=\mathbb{RP}^3\setminus\left(\mathbb{P}(p(H_{\{1,2,3\}}^\bot))\cup
 \mathbb{P}(p(H_{\{2,3,4\}}^\bot))\cup\mathbb{P}(p(H_{\{3,4,5\}}^\bot))\right)
 \]
(note that this open cell, after all blow-ups, is not the same as the open part of the brick manifold). After the blow-ups its boundary is the union of all exceptional divisors. The exceptional divisors over the two points enter the boundary with the coefficient $0$ because of the orientation. For the same reason, the exceptional divisors over the three $\mathbb{RP}^1$  enter the boundary with the coefficient $2$. Thus the boundary of $\mathsf{C}$ is twice the sum of the three divisors representing $c\circ_1 c$, $c\circ_2 c$, and $c\circ_3 c$. Therefore, the sum $c\circ_1 c+c\circ_2 c+c\circ_3 c$ is zero in the rational homology.
\end{proof}

\subsection{Nonsymmetric operads and partition posets}

In~\cite{Vallette07}, a framework relating poset homology with operadic homological algebra was proposed. Let us briefly outline the version of this formalism for ns operads. 

Let $\calP$ be a ns set operad. An \emph{nonsymmetric (ns) $\calP$-partition} of a totally ordered set $I$ is a partition of $I$ into an order sum of intervals $I=I_1+\cdots+I_k$ together with a choice, for each $j=1,\ldots,k$, of an element $b_j\in\calP(I_j)$. (In the language of ns collections of sets, the collection of ns $\calP$-partitions of $I$ is the collection $\As\circ\calP(I)$). 

Let us define a partial order on ns $\calP$-partitions of the given set $I$ as follows. Suppose that $\alpha=(a_1,\ldots,a_k)$, $a_r\in\calP(I_r)$, and $\beta=(b_1,\ldots,b_l)$, $b_s\in\calP(J_s)$, are two ns $\calP$-partitions of~$I$. We say that $\alpha\le\beta$ if there exist positive integers $n_1,\ldots,n_k$ with $n_1+\cdots+n_k=l$ and elements $c_i\in\calP(\underline{n_i})$ such that for each $i=1,\ldots,k$ we have
\begin{gather*}
I_i=J_{n_1+\cdots+n_{i-1}+1}+\cdots+J_{n_1+\cdots+n_{i-1}+n_i},\\
a_i=c_i(b_{n_1+\cdots+n_{i-1}+1},\ldots,b_{n_1+\cdots+n_i}) .
\end{gather*}
The poset of ns $\calP$-partitions of $I$ is denoted by $\Pi_\calP(I)$. Its only maximal element $\hat{1}$ corresponds to the partition into singletons, and the (only available) choice of the identity operation for each singleton, and its minimal elements are indexed by $\calP(I)$.  

A ns set operad $\calP$ is called a \emph{basic-set operad} if for each $(b_1,
\ldots,b_s)\in\calP(I_1)\times\cdots\times \calP(I_s)$ the map 
\begin{gather*}
\gamma\colon\calP(\underline{s})\to\calP(I_1+\cdots+I_s), \\
c\mapsto c(b_1,\ldots,b_s),
\end{gather*}
is injective. 

\begin{proposition}
Suppose that $\calP$ is a basic-set quadratic operad generated by $\calP(t)$ for some integer $t\ge 2$. The linear span of $\calP$ in the category of vector spaces is a Koszul operad if and only if, for all $I$, each maximal interval $[a,\hat{1}]\subset\Pi_\calP(I)$ is Cohen--Macaulay. Here $a\in\min(\Pi_\calP(I))\cong\calP(I)$.  Moreover, in this case, we have
 \[
\calP^{\ac}(I)\cong \bigoplus_{a\in\min(\Pi_\calP(I))}H_\bullet([a,\hat{1}]) .
 \] 
\end{proposition}

\begin{proof}
Completely analogous to \cite[Th.~9]{Vallette07}. 
\end{proof}

To formulate our next result, let us recall the definitions of two different types of associativity for $k$-ary algebras~\cite{LV}.

\begin{definition}\leavevmode
\begin{enumerate}
\item The operad $\tAs_k$ of \emph{totally associative $k$-ary algebras} is the ns operad with one generator $\alpha$ of arity $k$ and degree $0$, $|\alpha|=0$, satisfying the relations $\alpha\circ_i\alpha=\alpha\circ_k\alpha$ for $i=1,\ldots,k-1$. 	
\item The operad $\pAs_k$ of \emph{partially associative $k$-ary algebras} is the ns operad with one generator $\alpha$ of arity $k$ and degree $k-2$, $|\alpha|=k-2$, satisfying the only relation 
 \[
\sum_{i=1}^k(-1)^{(k-1)(i-1)}\alpha\circ_i\alpha=0 \ .  
 \]	
\end{enumerate}
\end{definition} 

\begin{proposition}[\cite{MR}]\label{prop:tAs-Koszul}
The operad $\tAs_k$ and $\pAs_k$ are Koszul dual to each other. Each of these operads is Koszul. 
\end{proposition}

This proposition immediately implies the following result which is a noncommutative analogue of the result of a celebrated result of Hanlon and Wachs~\cite{HanlonWachs}.

\begin{theorem}
Fix a positive integer $m$. Let $I$ be a finite ordinal, and let $\mathop{\mathrm{part}}^m(I)$ denote the poset whose elements are partitions $I=I_1+\cdots+I_k$, where $|I_j|\equiv 1\pmod{m}$ for all $j=1,\ldots,k$, and the partial order is defined by refinement, as above. Then the poset  $\mathop{\mathrm{part}}^m(I)$ is Cohen--Macaulay.  Moreover, we have
 \[
H_k(\mathop{\mathrm{part}}\nolimits^m(I))\cong
\begin{cases}
\tAs_{m+1}^{\ac}(I), \quad |I|=km+1, \\
\quad 0, \qquad \text{ otherwise}.
\end{cases}
 \] 
\end{theorem}

\begin{proof}
This immediately follows from Proposition~\ref{prop:tAs-Koszul} combined with the fact that the partition poset $\mathop{\mathrm{part}}^m(I)$ is isomorphic to the poset $\Pi_{\tAs_{m+1}}(I)$.
\end{proof}

\subsection{Noncommutative 2-Gerstenhaber algebras}

In \cite{EHKR10}, it was proved that the homology operad of the operad of real compactified moduli spaces $\overline{\calM}_{0,n+1}(\mathbb{R})$ is the operad of $2$-Gerstenhaber algebras. We shall now introduce the nonsymmetric counterpart of the latter operad which is relevant in our case. Since the operad of $2$-Gerstenhaber algebras is obtained from the operad of commutative algebras and the operad of Lie $2$-algebras~\cite{HanlonWachs} by a distributive law, the following definition should not be too surprising in the view of the previous sections.

\begin{definition}
The operad $\twoncGerst$ of \emph{noncommutative $2$-Gerstenhaber algebras} is the ns operad with a binary generator $m$ and a ternary generator $c$, $|m|=0$, $|c|=1$, satisfying the relations
	\begin{gather*}
	m\circ_1m-m\circ_2m=0,\\ 
	c\circ_1m=m\circ_2c,\\ 
	c\circ_2m=0,\\ 
	c\circ_3m=m\circ_1c,\\ 
	c\circ_1c+c\circ_2c+c\circ_3c=0 
	\end{gather*} 
of Proposition \ref{prop:rel-real-brick}.
\end{definition}

\begin{remark}
The arity $4$ Stasheff relation in a minimal $A_\infty$-algebra is~\cite{LV}
  \[
 m_3\circ_1 m_2- m_3\circ_2 m_2+m_3\circ_3m_2=m_2\circ_1 m_3+m_2\circ_2m_3.
  \]
It follows that the assignment $m_2\mapsto m$, $m_3\mapsto c$,  $m_r\mapsto 0$ for $r>3$, extends to a well defined homomorphism from the operad governing minimal $A_\infty$-algebras to the operad $\twoncGerst$. In particular, every noncommutative $2$-Gerstenhaber algebra can be regarded as a minimal $A_\infty$-algebra. This is a feature of the noncommutative situation that is not present in the commutative case.  
\end{remark}

\begin{theorem}
The homology operad of the real brick operad is isomorphic to the operad of noncommutative 2-Gerstenhaber algebras:
	\[
	H_\bullet(\calB_{\mathbb{R}})\cong \twoncGerst . 
	\]
\end{theorem}

\begin{proof}
Recall \cite[Th.~3.7]{Rains10} that for every real building set $\calG$, there is an isomorphism of rational homology
 \[
H_\bullet(\widehat{Y}_{\calG}(\mathbb{R}))\cong\bigoplus_{A\in\Pi^{(2)}_{\calG}} s^{\dim(A)}H^{\dim(A)-\bullet}([0,A]).  
 \]	
Here $\Pi^{(2)}_{\calG}$ is a subposet of $L_\calG$ consisting of the elements that can be written as direct sums of elements $G\in\calG$ with even $\dim(G)$, and $[0,A]$ denotes the corresponding interval in the poset $\Pi^{(2)}_{\calG}$. For every building set, this isomorphism is ``operadic'', that is compatible with maps of the corresponding vector spaces that are defined for various operations on building sets. 

Let us apply this result to the collection of arrangements $ncA_{n-1}$ for various $n$. In each of these cases, the building set $\calG$ is the collection of all subspaces
\[
H_I=\myspan(x_{p(i)}-x_i\mid \min(I)\ne i\in I) 
\]
for \emph{intervals} $I\subset\{1,2,\ldots,n\}$ with $|I|\ge 2$, and it is easy to see that the poset $\Pi^{(2)}_{\calG}$ is naturally identified with the poset $\mathop{\mathrm{part}}^2(\underline{n})$. In particular, this poset has just one maximal interval,  the homology of that interval is trivial for even $n$, and is $\tAs_3^{\ac}(\underline{n})$ concentrated in the homological degree $k$ for odd $n=2k+1$. 
In particular, this implies that on the level of nonsymmetric collections, we have
 \[
\left\{H_\bullet(\calB_\mathbb{R}(\underline{n}))\right\}\cong\left\{H_\bullet(\widehat{Y}_{ncA_{n-1}}(\mathbb{R}))\right\}\cong\As\circ\calS^{-1}(\tAs_3^{\ac})^*(\underline{n}) \}\cong\As\circ\pAs_3\ , 
 \]
and once restricted to the subcollection $\pAs_3=\id\circ\pAs_3\subset\As\circ\pAs_3$ and the top degree homology of $\calB_\mathbb{R}$, that is cohomology of maximal intervals, this is an operad isomorphism $\left\{H_{\mathrm{top}}(\calB_\mathbb{R})\right\}\cong\pAs_3$.
By direct inspection, we also have an isomorphism of nonsymmetric collections
 \[
\twoncGerst \cong \As\circ\pAs_3\ , 
 \]
since the relations of the operad $\twoncGerst$ can be easily check to give a distributive law between the operads $\As$ and $\pAs_3$. Thus, the operads $\{H_\bullet(\calB_\mathbb{R}(\underline{n}))\}$ and $\twoncGerst$ are generated by the same set of generators, have the same dimensions of components, and by Proposition~\ref{prop:rel-real-brick} all relations of the operad $\twoncGerst$ hold in the operad $\{H_\bullet(\calB_\mathbb{R}(\underline{n}))\}$. Therefore, these operads are isomorphic. 
\end{proof}

\appendix

\section{Brick manifolds and Loday polytopes}\label{sec:appendix-toric}

The goal of this section is to show that under the general correspondence between the toric varieties, fans, and polytopes, the polytope associated to the brick manifold $\calB(\underline{n})$ is the Loday polytope $L_n$ realising the Stasheff polytope. Some of the arguments below are similar to those of Ma'u~\cite[Sec.~3.5]{Mau2008}, who studied the non-negative part of the real toric variety associated to the Loday polytope and identified it with the moduli space of semistable nodal disks with marked points on the boundary.

\subsection{Torus action on the brick manifolds}

There is a natural action of the torus $(\mathbb{G}_m)^{n-1}$ on the brick manifold $\calB(\underline{n})$ arising from the action
\begin{equation}\label{eq:torus-action}
(\mathbb{G}_m)^{n-1}\ni (\lambda_{1},\dots,\lambda_{n-1})\colon e_{j,j+1}\mapsto \lambda_{j} e_{j,j+1},\qquad  j=1,\dots,n-1,
\end{equation} on $G(\underline{n})$. It turns out that the fixed points and the one-dimensional orbits of this action can be described in terms of our stratification of brick manifolds.

\begin{proposition}\leavevmode 

\begin{enumerate}
\item The fixed points of the torus action on the brick manifold $\calB(\underline{n})$ are given by the strata $\calB(\underline{n},T)$, where $T$ is a planar binary rooted tree. 
\item The one-dimensional orbits of the torus action on the brick manifold $\calB(\underline{n})$ are given by the strata $\calB(\underline{n}, T)$, where $T$ is a planar rooted tree all of whose vertices except for one are binary, and the remaining vertex is ternary.
\end{enumerate}
\end{proposition}

\begin{proof}
This follows by a direct inspection from the definition of brick manifolds, construction of their stratification in Section~\ref{sec:stratification} and the definition of the torus action.	
\end{proof}

\begin{remark}
	In other words, the set of fixed points (respectively, one-dimensional orbits) of the torus action coincides with the set of the dimension zero (respectively, one) strata of stratification given in Section~\ref{sec:stratification}.
\end{remark}

\subsection{Fan and dual polytope}

In this Section, we prove that the fan dual to the Loday polytope $L_n$ is the fan of the brick manifold $\calB(\underline{n})$.

Let us start with a few simple observations. First, the main diagonal of the torus $(\mathbb{G}_m)^{n-1}$ acts trivially on $\calB(\underline{n})$, so the fan is invariant under the translation 
\[
(y_1,\dots,y_{n-1})\mapsto (y_1+1,\dots,y_{n-1}+1)\ , 
\]
where $y_1,\dots,y_{n-1}$ are the coordinates on $\mathbb{Z}^{n-1}$ dual to the coordinates $x_1,\dots,x_{n-1}$. This means that the dual polytope must lie in an affine hyperplane $\sum_{i=1}^{n-1} x_i=\mathrm{constant}$. Second, the cones of  top dimension are indexed by the fixed points of the torus action, so the same must be true for the vertices of the dual polytope. Note that the Loday polytope satisfy both properties.

\begin{theorem}\label{thm:LodayPolytope}
	The dual fan of the Loday polytope $L_n$ is the fan of the toric variety $\calB(\underline{n})$. 
\end{theorem}

\begin{proof}
Observe that the fan of the brick manifold $\calB(\underline{n})$ is the intersection of the fans of the projective spaces formed by the choices of $V_{l,r}\subset G(l,r+1)$, for $l,r=2,\dots,n-1$, $r\geq l$. Indeed, the $n-1$-dimensional cones of the fan correspond to the fixed points of the torus action, and it implies that each $V_{l,r}$ should be fixed by the torus.
	
This means that the $n-2$ dimensional cones that bound a particular $n-1$ dimensional cone in the fan of $\calB(\underline{n})$ are the $n-1$ dimensional cones of some of these projective spaces. They also should correspond to the one-dimensional orbits of the torus in $\calB(\underline{n})$ and also in some of these projective spaces.
	
Consider a planar rooted tree with one ternary vertex, and all other vertices being binary. Denote by $D^l$ (respectively, $D^c$, $D^r$) the set of leaves that are the left (respectively, central, right) descendants of this special vertex. Assume that $D^l=\{{l}_{i},\dots, {l}_{j-1}\}$, $D^c=\{{l}_{j},\dots, {l}_{k}\}$, $D^r=\{{l}_{k+1},\dots, {l}_{m}\}$. The construction of the stratification in Section~\ref{sec:stratification} implies that, in this case, all spaces $V_{l,r}$ are spanned by the basis vectors (and, therefore, invariant under the torus action) except for the spaces with $l=i+1,\dots,j$ and $r=k,\dots,m-1$. In this range of $l$ and $r$, these spaces must be equal to the direct sum of some invariant vector space of dimension $r-l$ spanned by the coordinate vectors and an arbitrary  one-dimensional space $W\subset \myspan(e_{j-1,j},e_{k,k+1})$ (which must be the same for all these spaces). Thus the hyperspace that contains the corresponding $n-2$ dimensional cone is given by the equation $y_{j-1}-y_{k}=0$. 
	
On the other hand, the corresponding one-dimensional face of the Loday polytope $L_n$ has the direction equal to the difference of the two vertices that correspond to the two possible ways to split the ternary vertex  of the corresponding planar rooted tree into two binary vertices. Consider these two trees, $T_1$ and $T_2$. By the definition of the Loday polytope, all coordinates of $p_{T_1}$ and $p_{T_2}$ coincide except for the $x_{j-1}$ and $x_{k}$. 
Furthermore, 
	\begin{align*}
	x_{j-1}(p_{T_1})& =(j-i)(k-j+1) & x_{k}(p_{T_1}) & =(k-i+1)(m-k) \\ 
	x_{j-1}(p_{T_2})& =(j-i)(m-j+1) & x_{k}(p_{T_2}) & =(k-j+1)(m-k),
	\end{align*}
	so the vector along the corresponding one-dimensional face of $L_n$ is equal to 
	\[
	p_{T_1}-p_{T_2}=(0,\dots,0,-(j-i)(m-k),0,\dots,0,(j-i)(m-k),0,\dots,0)\ ,
	\]
which clearly generated the subspace $y_{j-1}-y_{k}=0$ in the dual space. 
	
Since the Loday polytopes and the fans of brick manifolds are strictly convex, it is enough to show the duality of faces of dimension $\leq 1$ and the cones of codimension $\leq 1$ in order to prove the the fan dual to the Loday polytope is the fan of the brick manifold.
\end{proof}

\begin{remark}
Note that the correspondence between planar rooted trees, faces of the Loday polytopes, and the strata of $\calB(\underline{n})$ extends further. Namely, for each planar rooted tree $T$ we have a canonically associated face of the Loday polytope, that is associated to a torus invariant subvariety of the same dimension, which is precisely the stratum $\calB(\underline{n},T)$. 
\end{remark}

\begin{remark}
A linearly dual version of the same proof re-told in the language of wonderful models goes as follows. Clearly, there is an action of the torus $(\mathbb{G}_m)^{n-1}$ on $\widehat{Y}_{ncA_{\underline{n}}}$ arising from the action on the product 
	\[
	\mathbb{P}(V)\times \prod_{I \text{ an interval of } \underline{n},\, |I|\ge 2}\mathbb{P}(V/p(H_I^\bot)) \cong  \mathbb{P}(V)\times \prod_{I \text{ an interval of } \underline{n},\, |I|\ge 2}\mathbb{P}(G(I)^*)
	\]
coming from the action~\eqref{eq:torus-action} on $G(\underline{n})$. (Note that in our case $V=G(\underline{n})^*$, so in fact the factor $\mathbb{P}(V)$ is redundant). Of course, the diagonal torus acts trivially, so we have the action of a torus $(\mathbb{G}_m)^{n-1}/\mathbb{G}_m\cong(\mathbb{G}_m)^{n-2}$. Moreover, from the proof of Theorem~\ref{th:wonderful-brick}, it is immediate that for each interval $I\subset\underline{n}$, the image of $\widehat{M}_{ncA_{\underline{n}}}\subset\mathbb{P}(V)$ in $\mathbb{P}(V/H_I^\bot)\cong\mathbb{P}(G(I))$ consists of all points whose homogeneous coordinates are all distinct from zero. In particular, $\widehat{M}_{ncA_{\underline{n}}}\subset\mathbb{P}(V)\cong\mathbb{P}(G(\underline{n})^*)$ is identified with the torus $(\mathbb{G}_m)^{n-2}$. This implies that the wonderful model is a toric variety. Moreover, for each individual factor $\mathbb{P}(V/p(H_I^\bot))$, the closure of the corresponding torus orbit is the whole $\mathbb{P}(V/p(H_I^\bot))$. Hence by \cite[Proposition~$8.1.4$]{GKZ94}, the polytope associated to $\widehat{Y}_{ncA_{\underline{n}}}$ is the Minkowski sum of simplices corresponding to intervals of $\underline{n-1}\cong\Gap(\underline{n})$. By Proposition~\ref{prop:Minkowski}, the former polytope coincides with the Loday polytope~$L_n$. 	
\end{remark}

\subsection{Combinatorial definition of the Loday polytope revisited}

In this section, we give, for completeness, a way  to construct the Loday polytope $L_n$ directly from the fixed points of the torus action on $\calB(\underline{n})$. This construction could be used for an alternative proof of Theorem~\ref{thm:LodayPolytope}.

Consider a fixed point  of the torus action on $\calB(\underline{n})$ associated to a planar binary rooted tree $T$. The vectors spaces $V_{l,r}$ of this fixed point are the vector spaces spanned by the basis vectors $e_{1,2},\dots,e_{n-1,n}$, moreover, we know that $V_{l,r}\subset G(l-1,r+1)$, for $l,r=2,\dots, n-1$, $l\leq r$.
Consider the lattice $M:=\mathbb{Z}\langle e_{1,2},\dots,e_{n-1,n}\rangle$. We denote by $v_{l,r}\in M$ the ``missing basis vector'' in $V_{l,r}$, that is, $v_{l,r}\in \{e_{l-1,l},\dots,e_{r,r+1}\}$ such that $v_{l,r}\oplus V_{l,r}=G(l-1,r+1)$.

\begin{proposition} The vertex $p_T$ of the Loday polytope $L_n$ is given by $e_{1,2}+\cdots+e_{n-1,n}+\sum_{l=2}^{n-1}\sum_{r=l}^{n-1} v_{l,r}$. That is, up to a shift by the vector $e_{1,2}+\cdots+e_{n-1,n}$ the vector $p_T$ is the sum of ``missing basis vectors'' in $\calB(\underline{n},T)$. 
\end{proposition}

\begin{proof} Consider the tree $T$. Let ${v}_i$ be root vertex of $T$. It is sufficient to prove that the $i$-th coordinate of $p_T$ coincides with the $i$th coordinate of the sum   $e_{1,2}+\cdots+e_{n-1,2}+\sum_{l=2}^{n-1}\sum_{r=l}^{n-1} v_{l,r}$, 	
since then the same argument applies to every other coordinate and the subtree where the corresponding vertex is the root vertex. 
	
Since ${v}_i$ is the root vertex, we have $V_{i,i}=G(i-1,i)$ and $V_{i+1,i+1}=G(i+1,i+2)$, with a natural adjustment for $i=1$ and $n-1$. This implies that the vector $e_{i,i+1}$ is the missing basis vector in all spaces $V_{l,r}$, where $l=2,\dots, i+1$, and $r=i,\dots,n-1$, There are $i(n-i)-1$ such spaces (there is no space $V_{l,r}$ for $l={i+1}$, $r=i$), so the coefficient of $e_{i,i+1}$ in the sum $e_{1,2}+\cdots+e_{n-1,n}+\sum_{l=2}^{n-1}\sum_{r=l}^{n-1} v_{l,r}$ is equal to $i(n-i)$. Since ${v}_i$ is the root vertex, this latter number is precisely $|D^l({v}_i)|\cdot |D^r({v}_i)|$.
\end{proof}

\bibliographystyle{amsplain}
\providecommand{\bysame}{\leavevmode\hbox to3em{\hrulefill}\thinspace}

\end{document}